\documentclass[12pt,twoside,a4paper]{article}
\usepackage{amsmath,amsthm,amssymb}
\usepackage{fullpage}
\usepackage{graphicx}
\usepackage{dsfont}
\usepackage{color}
\usepackage{pgfplots}
\usepackage{caption}
\usepackage{subcaption}
\usepackage{float}
\usepackage{enumitem}

\def\CA{\mathcal{A}}
\def\CC{\mathcal{C}}

\def\CM{\mathcal{M}}

\def\CF{\mathcal{F}}
\def\CL{\mathcal{L}}
\def\CD{\mathcal{D}}

\def\CO{\mathcal{O}}

\def\CPD{\mathcal{PD}}

\def\x{\textup{x}}
\def\y{\textup{y}}

\def\AC{\mathcal{AC}}

\def\Fa{\mathfrak{a}}
\def\Fb{\mathfrak{b}}

\def\Fc{\mathfrak{c}}

\def\N{\mathbb{N}}
\def\Z{\mathbb{Z}}
\def\P{\mathbb{P}}
\def\S{\mathbb{S}}

\def\1{\mathds{1}}

\def\kh{\hat{\kappa}}
\def\kt{\tilde{\kappa}}

\renewcommand{\emptyset}{\varnothing}

\def\R{\mathbb{R}}

\def\E{\mathbb{E}}

\newtheorem{counter}{}[section]
\newtheorem{theorem}[counter]{Theorem}
\newtheorem{definition}[counter]{Definition}
\newtheorem{lemma}[counter]{Lemma}
\newtheorem{claim}[counter]{Claim}
\newtheorem{remark}[counter]{Remark}
\newtheorem{proposition}[counter]{Proposition}
\newtheorem{corollary}[counter]{Corollary}
\newtheorem{fact}[counter]{Fact}

\def\eps{\varepsilon}

\renewcommand{\complement}{\mathsf{c}}

\newcommand{\diam}{\textup{diam}}
\newcommand{\tail}{\textup{tail}}
\newcommand{\head}{\textup{head}}

\newcommand{\inv}{\textup{inv}}

\newcommand{\arc}{\textup{arc}}

\newcommand{\indt}{\textup{i}^\textup{t}}
\newcommand{\indh}{\textup{i}^\textup{h}}

\newcommand{\var}{\textup{Var}}

\title{On the Cycle Structure of Mallows Permutations}

\begin{document}
\author{Alexey Gladkich\textsuperscript{1} \and Ron Peled\textsuperscript{1}}

\maketitle

\footnotetext[1]{School of Mathematical Sciences, Tel Aviv
University, Tel Aviv, 69978, Israel. Supported by ISF grants 1048/11
and 861/15 and IRG grant SPTRF. alexeygl@mail.tau.ac.il,
peledron@post.tau.ac.il. }

\begin{abstract}
We study the length of cycles of random permutations drawn from the
Mallows distribution. Under this distribution, the probability of a
permutation $\pi \in \S_n$ is proportional to $q^{\inv(\pi)}$ where
$q>0$ and $\inv(\pi)$ is the number of inversions in $\pi$.

We focus on the case that $q<1$ and show that the expected length of
the cycle containing a given point is of order $\min\{(1-q)^{-2},
n\}$. This marks the existence of two asymptotic regimes: with high
probability, when $n$ tends to infinity with $(1-q)^{-2} \ll n$ then
all cycles have size $o(n)$ whereas when $n$ tends to infinity with
$(1-q)^{-2}\gg n$ then macroscopic cycles, of size proportional to
$n$, emerge. In the second regime, we prove that the distribution of
normalized cycle lengths follows the Poisson-Dirichlet law, as in a
uniformly random permutation. The results bear formal similarity
with a conjectured localization transition for random band matrices.

Further results are presented for the variance of the cycle lengths,
the expected diameter of cycles and the expected number of cycles.
The proofs rely on the exact sampling algorithm for the Mallows
distribution and make use of a special diagonal exposure process for
the graph of the permutation.
\end{abstract}


\section{Introduction} The cycle structure of a random permutation
picked uniformly from $\S_n$, the permutation group on $n$ elements,
is a classic topic in probability theory. Of the abundant literature
on it we mention two key facts: The distribution of the length of a
cycle containing a given point is uniform on $\{1,\ldots, n\}$.
Moreover, the joint distribution of the lengths of the longest
cycles in the permutation has an explicit limit; the sorted vector
of cycle lengths, normalized by $n$, converges in distribution to
the Poisson-Dirichlet distribution with parameter one.

In this work we study the cycle structure of a random permutation
distributed according to the Mallows distribution. The Mallows
distribution is a non-uniform distribution on permutations which was
introduced by Mallows in statistical ranking theory \cite{M57}. It
has recently been the focus of several studies in varied contexts
including mixing times of Markov chains \cite{BBHM05,DR00},
statistical physics \cite{S09, SW15}, learning theory \cite{BM09},
q-exchangeability \cite{GO10,GO12} and the problem of the longest
increasing subsequence \cite{MS13, BP13, BB16}. Borodin, Diaconis
and Fulman \cite[Section 5]{BDF10} considered a class of models of
random permutations (denoted $P_\theta$ there) for which the Mallows
distribution is the prime example. They noted that many of the
``usual questions'' of applied probability and enumerative
combinatorics remain open for such models and asked ``Picking a
permutation randomly from $P_{\theta}(\cdot)$, what is the
distribution of the cycle structure, longest increasing subsequence,
\ldots?''. Our work answers the part of this question pertaining to
the cycle structure of the Mallows distribution.

The Mallows distribution on $\S_n$ is parameterized by a real number
$q>0$ and is denoted $\mu_{n,q}$. It is given by
\begin{equation}\label{eq:mu_n_q_def}
\mu_{n,q}[\pi] := \frac{q^{\inv(\pi)}}{Z_{n,q}}
\end{equation}
where
\begin{equation*}
  \inv(\pi) := |\{(s,t)\mid\text{$s<t$ and $\pi_s>\pi_t$}\}|
\end{equation*}
denotes the number of inversions in $\pi$, and $Z_{n,q}$ is a
normalizing constant, given explicitly by the following formula
\cite[Corollary 1.3.13]{S12},
\begin{equation*}
Z_{n,q} = \displaystyle\prod_{i=1}^{n}\left(1 + q + \cdots +
q^{i-1}\right)=\displaystyle\prod_{i=1}^{n}\frac{1-q^{i}}{1-q}.
\end{equation*}
The Mallows distribution with parameter $q=1$ coincides with the
uniform distribution on $\S_n$. In this paper we restrict attention
to the case that $0<q<1$ (a brief discussion of the case $q>1$ is
given in Section~\ref{sec:discussion}). As is well known,
$\inv(\pi)$ equals the minimal number of \emph{adjacent}
transpositions required to bring $\pi$ to the identity. Thus, when
$0<q<1$, the Mallows distribution gives higher weight to
permutations which are closer to the identity in an underlying
one-dimensional geometry.

We are mainly interested in the properties of the Mallows
distribution for $q$ close to~$1$, usually as a function of $n$,
although our results apply in the full range of $0<q<1$.
Figure~\ref{fig:mallows_sample} depicts samples of the Mallows
distribution.
 One simple feature of a Mallows distribution is that it
typically displaces elements by a small amount. This is quantified
in the following statement: there exists an absolute constant $c>0$
such that if $\pi\sim\mu_{n,q}$ then for all $0<q<1$ and $1\le s\le
n$,
\begin{equation}\label{eq:displacements}
  c \cdot \min \left\{ \frac{q}{1-q}, n-1\right\} \le \E|\pi_s - s| \le \min \left\{ \frac{2q}{1-q},n-1\right\},
\end{equation}
see \cite{BP13} for a proof and related concentration bounds or
\cite{BM09, GO12} for similar statements.
\pgfplotsset{width=\textwidth}
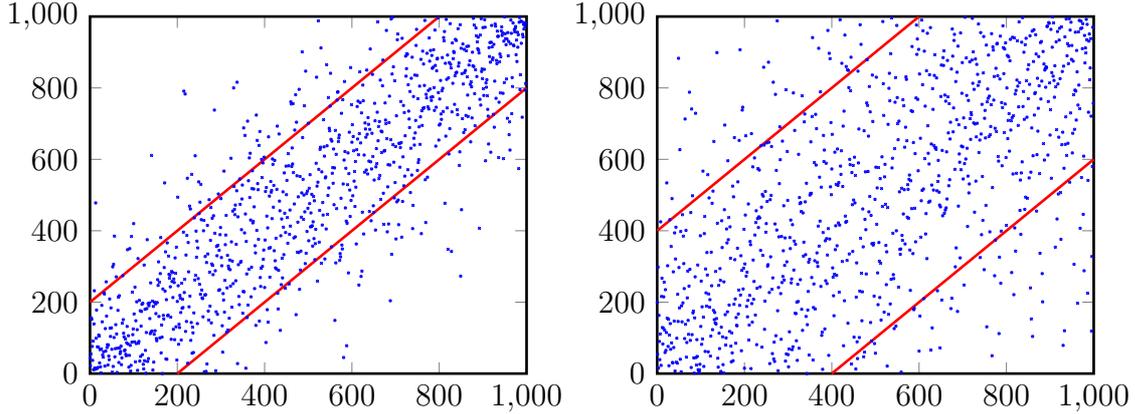
\begin{figure}
\begin{subfigure}[t]{0.46\textwidth}
\begin{tikzpicture}
\begin{axis}[
xmin = 0,
xmax = 1000,
ymin = 0,
ymax = 1000,
line width = 1pt,
]
\addplot[
red,
domain=0:800,
samples=801,
line width = 1pt
]
{x+200)};
\addplot[
red,
domain=200:1000,
samples=801,
line width = 1pt
]
{x-200)};
\addplot+[blue, only marks, mark size = 0.1pt] table {mallowsSample1000n99q.dat};
\end{axis}
\end{tikzpicture}
\end{subfigure}
\begin{subfigure}[t]{0.46\textwidth}
\begin{tikzpicture}
\begin{axis}[
xmin = 0,
xmax = 1000,
ymin = 0,
ymax = 1000,
line width = 1pt,
]
\addplot[
red,
domain=0:600,
samples=601,
line width = 1pt
]
{x+400)};
\addplot[
red,
domain=400:1000,
samples=601,
line width = 1pt
]
{x-400)};
\addplot+[blue, only marks, mark size = 0.1pt] table {mallowsSample1000n995q.dat};
\end{axis}
\end{tikzpicture}
\end{subfigure}
\caption{
    Graphs of samples of the Mallows distribution $\mu_{n,q}$ with $n=1000$, $q=0.99$ (left) and $q = 0.995$ (right). The red lines are at vertical distance $\frac{2}{1-q}$ from the
    diagonal. They delimit a region containing most of the points of the permutation, see also \eqref{eq:displacements}.
} \label{fig:mallows_sample}
\end{figure}

Thus the expected
displacements are of order $o(n)$ when $\frac{1}{1-q} \ll n$ and it
is natural to ask whether this also results in shorter cycles.
Our first result determines the expected length of cycles.

\smallskip {\bf Notation:} For two quantities $x,y$, which may depend
on other parameters such as $n$ or $q$, we write $x\approx y$ if
there exist absolute constants $c,C>0$ such that $c\, y\le x \leq
C\, y$.

For a permutation $\pi\in \S_n$ and $1\le s\le n$ we let $\CC_s=\CC
_s(\pi)$ be the orbit of $s$ in $\pi$, i.e., the set of points in
the cycle of $\pi$ which contains $s$.
\begin{theorem}[Expected Cycle Length]\label{cycle_length}
 Let $n\ge 1$, $0<q<1$ and $\pi\sim\mu_{n,q}$. Then
    \[
        \E|\CC_s|\approx \min\left\{\frac{1}{(1-q)^{2}},\, n\right\}\quad \text{for all $1\le s\le n$}.
    \]
\end{theorem}
Thus the expected length of the cycle containing a given point
transitions from being $o(n)$ when $\frac{1}{1-q}\ll \sqrt{n}$ to
being $\Omega(n)$ in the complementary regime. The same is true also
for the \emph{maximal} cycle length in the permutation, see
Claim~\ref{cl:pdl_arcs_claim1}.
We say this transition marks the
emergence of \emph{macroscopic cycles} in the permutation.

Theorem~\ref{cycle_length} identifies a similarity between the
uniform distribution and the Mallows distribution in the regime that
macroscopic cycles exist, namely, that the expected cycle lengths in
both distributions are of order $n$. The two distributions are quite
different in many other respects, for instance, when $\frac{1}{1-q}
\ll n$ they are distinguished even by their typical displacements as
measured by \eqref{eq:displacements}. Our next result shows that as
far as the lengths of the \emph{long} cycles are concerned, the
similarities extend much further than what may initially be
expected: the two distributions give rise to the same limit law.

\begin{theorem}[Poisson-Dirichlet Law]\label{thm:pdl}
Suppose that the sequence $(q_n)$ satisfies
\begin{equation*}
  0<q_n<1 \quad\text{and}\quad (1-q_n)^2\cdot n\to 0.
\end{equation*}
Let $\pi\sim\mu_{n,q_n}$ and let $\ell_1\geq \ell_2 \ge \ldots$ be
the sorted lengths of cycles in $\pi$. Then, as ${n\to\infty}$,
\[
\text{$\tfrac{1}{n}\left(\ell_1,\ell_2,\ldots\right)$ converges in
distribution to the Poisson-Dirichlet law with parameter one.}
\]
In addition, for any sequence $(s_n)$ satisfying $1\leq s_n\le n$,
as $n\to\infty$,
\[
    \text{$\tfrac{1}{n}|\CC_{s_n}|$ converges in distribution to the uniform distribution on $[0,1]$.}
\]
\end{theorem}

\pgfplotsset{width=\textwidth}
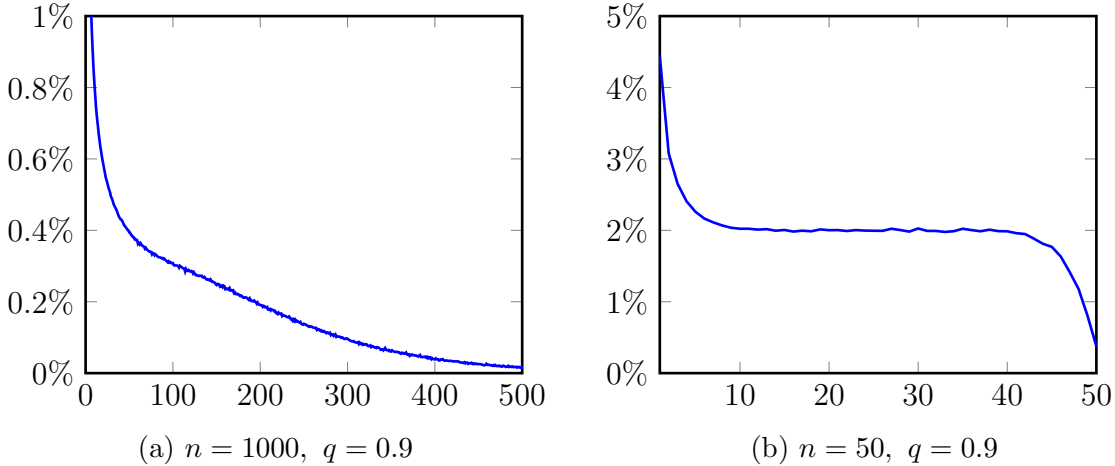
\begin{figure}
\begin{subfigure}{0.46\textwidth}
\begin{tikzpicture}
\begin{axis}[
xmin = 0,
xmax = 500,
ymin = 0,
ymax = 1,
ytick = {0,0.2,0.4,0.6,0.8,1},
yticklabels = {0\%,0.2\%,0.4\%,0.6\%,0.8\%,1\%},
line width = 1pt] \addplot[blue] table {CycleLength1000n9q.dat};

\end{axis}
\end{tikzpicture}
\caption{   $n=1000, \ q=0.9$}
\end{subfigure}
\quad
\begin{subfigure}{0.46\textwidth}
\begin{tikzpicture}
\begin{axis}[
xmin = 1,
xmax = 50,
ymin = 0,
ymax = 5,
ytick = {0,1,2,3,4,5},
yticklabels = {0\%,1\%,2\%,3\%,4\%,5\%},
line width = 1pt ]
\addplot[blue] table {CycleLength50n9q.dat};

\end{axis}
\end{tikzpicture}
\caption{   $n=50, \ q=0.9$} \label{fig:cycle_length_uni_n50}
\end{subfigure}
\caption{Distribution of the length of the cycle containing a
uniform random point. Obtained empirically with 1000000 samples.}
\label{fig:cycle_length}
\end{figure}

Our results provide further information on the cycle lengths in the
regime in which there are no macroscopic cycles. We show that the
cycle lengths are not concentrated in the sense that their standard
deviation has the same order of magnitude as their expectation.
\begin{theorem}[Variance of Cycle Length]\label{cycle_var}
    Let $n\ge 1,\, 0<q<1$ and $\pi\sim\mu_{n,q}$. Then
    \[
        \var|\CC_s|\approx \min\left\{\frac{q}{(1-q)^{4}},
        (n-1)^2\right\}\quad\text{for all $1\leq s \leq n$}.
    \]
\end{theorem}
The argument showing that the standard deviation is at least as
large as the expectation, when $q$ is bounded away from $0$, bears
something of a general nature and may be applicable to other spatial
permutation models such as the interchange model; see
Section~\ref{sec:spatial_random_permutations} and
Section~\ref{sec:discussion} for more details on these models.

Our next theorem considers the diameter of cycles, showing that the
cycles are dense in their support in the sense that their lengths
are comparable to their diameters on average.
\begin{theorem}[Expected Cycle Diameter]
    \label{cycle_diam_theorem}
    Let $n\ge 1,\, 0<q<1$ and $\pi\sim\mu_{n,q}$. Then, for all $1 \le s \le
    n$,
    \begin{equation} \label{diam_equation}
        \E[\max (\CC_s)-\min(\CC_s)]\approx \min\left\{\frac{q}{(1-q)^{2}}, n-1\right\}
    \end{equation}
and, moreover,
    \begin{align}
        \label{max_equation}    \E[\max (\CC_s)-s]&\approx \min\left\{\frac{q}{(1-q)^{2}}, n-s\right\}, \\
        \label{min_equation}    \E[s-\min (\CC_s)] &\approx \min\left\{\frac{q}{(1-q)^{2}}, s-1\right\} .
    \end{align}
\end{theorem}

Given our previous theorems, naively, one may expect that a typical
random Mallows permutation $\pi$ has about $n / \min\{(1-q)^{-2},
n\}$ cycles, as the cycle containing a given point typically has
length of order $\min\{(1-q)^{-2}, n\}$. However, such reasoning is
known to be false even for a uniformly random permutation, in which
the cycle containing a given point typically has length of order
$n$, yet there are $\log n$ cycles on average. This phenomenon
reflects the fact that while most cycles are short, most points lie
in long cycles. Our last theorem clarifies that this is also the
case for random Mallows permutations and gives the order of
magnitude of the number of cycles.
\begin{theorem}[Expected Number of Cycles]
    \label{thm:number_of_cycles}
    Let $n\ge 1,\, 0<q<1$ and $\pi\sim\mu_{n,q}$. Then
    \[
        \E[\text{number of cycles in $\pi$}]\approx (1-q)\cdot n + \log(n+1).
    \]
\end{theorem}

\subsection{Sampling Algorithm}\label{sec:initial_sampling_algorithm}

Our results are based on an exact sampling algorithm for the Mallows
distribution which goes back to the original work of Mallows
\cite{M57}. The algorithm allows us to sample a permutation $\pi\sim
\mu_{n,q}$ sequentially as follows: Given $\pi_1, \ldots \pi_{s-1}$,
the distribution of $\pi_{s}$ is distributed on the remaining
$n-s+1$ values in a geometric progression. Precisely, if the
remaining values are $j_1<j_2<\cdots<j_{n-s+1}$ then
\begin{equation}\label{eq:sampling_formula}
  \P[\pi_{s} = j_k\mid \pi_1, \ldots, \pi_{s-1}] =
  \frac{1-q}{1-q^{n-s+1}} \cdot q^{k-1}.
\end{equation}
It is simple to verify the validity of this formula by noting that,
given $\pi_1, \ldots, \pi_{s-1}$, the assignment $\pi_s = j_k$
creates exactly $k-1$ inversions between $\pi_s$ and $\pi_{s+1},
\ldots, \pi_{n}$; precisely, if $\pi_s=j_k$ then necessarily
$|\{t\mid t>s,\, \pi_t < \pi_s\}| = k-1$.

In our proofs we develop more flexible versions of the above
formula, allowing us to sample portions of the cycles of the
permutation iteratively and control the evolution of these portions,
see Section~\ref{sec:diagonal_exposure} and the beginning of
Section~\ref{section_mt}.

There exist extensions of the Mallows distribution and the above
formula (for $q<1$) to infinite permutations; one-to-one and onto
functions $\pi\colon\N\to\N$ or $\pi\colon\Z\to\Z$. The extension to
the case of $\N$ is straightforward, one simply takes the limit
$n\to\infty$ in \eqref{eq:sampling_formula} to obtain a geometric
distribution, see Gnedin and Olshanski \cite{GO10}. The extension to
a two-sided infinite permutation, when the index set is $\Z$, is
more complicated due to the fact that there is no natural initial
position to start the sampling process from.

Generating methods for the two-sided infinite case were developed in
\cite{GO12}. In one of these methods, one samples two one-sided
infinite Mallows permutations and uses a `stitching' mechanism to
merge these into a two-sided infinite permutation. We also present a
method for sampling a Mallows permutation `from an interior point',
see Section~\ref{ggm}. The method is presented for finite $n$ and
may be used also for the two-sided infinite case via an
approximation theorem from \cite[Section 7.2]{GO12}. This method may
serve as a bridge to transfer results from the finite $n$ case to
the two-sided infinite case.

\subsection{Relation with other models}
In this section we briefly describe other models for which related
results have been obtained or are conjectured.

\subsubsection{Permutons} The regime of parameters in which $n\cdot
(1-q)\to\beta$ is also of special interest as in this case there is
a limiting density to the empirical measure of the points in the
graph of a Mallows permutation. Starr~\cite{S09} obtained an
explicit formula for the limiting density as a function of $\beta$.
In modern terminology, the limiting density is called a
\emph{permuton}. Recently, Mukherjee \cite{M15} proved Poisson limit
theorems for the lengths of \emph{short} cycles for models
converging to permutons, including the Mallows model as a special
case. See also Kenyon, Kr\'al', Radin and Winkler \cite{KKRW15} for
relations with permutons with fixed pattern densities.

\subsubsection{Band Matrices}
In the study of random matrices, models of matrices with a band
structure are of interest. We elaborate on one representative model:
Let $A$ be an $n\times n$ random matrix in which, for a given band
width $0 < W \le n$, the entries $A_{i,j}$, $|i - j|<W$, are
independent and identically distributed standard Gaussian random
variables and the other entries are set to zero. Define the
symmetric band matrix $H$ by
\begin{equation*}
  H:=\frac{A + A^t}{\sqrt{2}}.
\end{equation*}
The main focus in these studies is on the eigenvalues and
eigenvectors of $H$.

In one extreme case $W=1$, meaning that the matrix $H$ is diagonal,
the eigenvectors are the standard basis vectors. The other extreme
case, when $W=n$, results in the Gaussian Orthogonal Ensemble (GOE)
distribution (up to scaling). In this case the distribution of $H$
is invariant under conjugation by orthogonal matrices, implying that
the eigenvectors of $H$ form a uniformly distributed orthonormal
basis.

It is conjectured that random matrices of this kind undergo a
localization / delocalization transition as the band width $W$
increases beyond the threshold $\sqrt{n}$. More precisely, one
expects that when $W \ll \sqrt{n}$, the eigenvectors are localized
in the sense that most of their $\ell^2$ mass lies on a set whose
size is $o(n)$ (possibly even in an interval of such size), whereas
if $W \gg \sqrt{n}$ the eigenvectors have their $\ell^2$ mass
approximately uniformly spread. Furthermore, in the second regime,
it is expected that the local eigenvalue statistics have the same
limit as in the GOE case as $n$ tends to infinity. Informally, we
may say that the local eigenvalue statistics should have the
mean-field limit in the delocalized regime. See the survey of
Spencer \cite{S11} and references within for more on these topics.

Our results prove an analogous transition for the Mallows
distribution. One may consider the permutation matrix $H_\pi$
associated with a random permutation $\pi\sim\mu_{n,q}$. By
\eqref{eq:displacements}, this matrix has an approximate band
structure in the sense that few of its non-zero entries $(H_\pi)_{s,
\pi_s}$ have $|\pi_s-s|$ greater than a constant multiple of the
band width $W = \min\{\frac{1}{1-q}, n\}$ (in fact, the probability
that $|\pi_s - s|\ge tW$ decays exponentially in $t$, see
\cite[Theorem 1.1]{BP13} and Figure~\ref{fig:mallows_sample}). Such
a matrix is orthogonal, having its eigenvalues on the unit circle.
The eigenvalues and eigenvectors of $H_\pi$ are determined by the
cycle structure of $\pi$: associated with each cycle of length
$\ell$, one has the $\ell$ eigenvalues $\exp\left(\frac{2\pi i
j}{\ell}\right)$, $0\le j\le \ell-1$, and correspondingly $\ell$
eigenvectors, supported on the coordinates of the cycle and giving
equal mass to all points of it. Thus, a localization /
delocalization transition corresponds to the emergence of cycles
whose length is of order $n$. Theorem~\ref{cycle_length} shows that
such a transition occurs as the band width increases beyond
$\sqrt{n}$, paralleling the conjecture for random band matrices.
Moreover, Theorem~\ref{thm:pdl} shows that in the delocalized
regime, the statistics of long cycles approach the Poisson-Dirichlet
distribution, the limiting statistics for uniform random
permutations, in analogy with the above prediction for the local
eigenvalue statistics.

The reader is also referred to the survey of Olshanski \cite{O11}
for other analogies between random permutations and random matrices,
discussing, in particular, analogies between random permutations
distributed according to the Ewens measure (see also
Section~\ref{sec:discussion}) and deformations of Dyson's circular
ensemble of random matrices.

\subsubsection{Card Shuffling}\label{sec:card_shuffling}
There are many natural dynamics on permutations for which the
uniform distribution is stationary. Diaconis and Shahshahani
\cite{DS81} consider the following natural card shuffling scheme:
Start with a deck of $n$ cards. At each step choose two cards
uniformly and independently and exchange their positions in the
deck. How many steps does one need to perform in order for the deck
to become almost perfectly shuffled? In a beautiful application of
representation theory to the study of Markov chains, it is proved in
\cite{DS81} that the state of the deck after $\frac{1}{2}n\log n + c
n$ such steps is close to uniform (in the total variation distance)
when $c$ is a large positive constant, and is far from uniform when
$c$ is a large negative constant. The latter bound follows from the
analysis of the coupon collector problem: when $c$ is a large
negative constant there will be many cards in the deck which have
not moved from their initial position, creating a permutation with
many fixed points. Thus, the result of \cite{DS81} may be
interpreted as saying that the number of short cycles is the main
obstacle for a permutation to become approximately uniform in this
card shuffling scheme.

Schramm \cite{S05} considered the above card shuffling scheme
further, investigating the state of the deck after $tn$ steps are
performed. The analysis in \cite{S05} proceeds by drawing an
associated graph on the vertex set $\{1,\ldots, n\}$, in which an
edge is put between $i$ and $j$ if the cards at positions $i$ and
$j$ in the deck have been exchanged. This associated graph is
distributed as an Erd\H os-R\'enyi random graph, allowing one to
deduce from the standard literature that when $t \le \frac{1}{2}$,
all cycles in the random permutation have size $o(n)$. Schramm's
work focuses on the case that $t>\frac{1}{2}$ and proves that
macroscopic cycles emerge in this regime (see also Berestycki
\cite{B11} for a later simpler argument). Moreover, confirming a
conjecture of Aldous, it is proved that the limiting joint
distribution of these macroscopic cycles obeys the same
Poisson-Dirichlet law observed for uniform permutations. Thus,
although it takes about $\frac{1}{2}n\log n$ steps for the full
permutation to become approximately uniform, it takes far fewer
steps for macroscopic cycles to start emerging and the joint
distribution of these macroscopic cycles converges very quickly to
the limiting joint distribution. A similar fact is true for the
Mallows model by our results: when $q$ increases beyond the
threshold $1 - \frac{1}{\sqrt{n}}$, although the Mallows permutation
is still far from uniform (distinguished by its displacements, say,
as in \eqref{eq:displacements}), macroscopic cycles begin to emerge
and their joint distribution converges to the Poisson-Dirichlet law.

In this context we mention that the Mallows permutation also arises
via a shuffling algorithm. As proved by Benjamini, Berger, Hoffman
and Mossel \cite{BBHM05}, it arises as the stationary distribution
of a biased card-shuffling algorithm. In this algorithm, one starts
with a deck of cards numbered $\{1,\ldots, n\}$ and at each
iteration picks uniformly a pair of adjacent cards in the deck. One
flips a coin with probability $p = \frac{1}{1+q}$ for heads and
rearranges the two cards according to the coin result, in increasing
order if heads and in decreasing order if tails. The iterations are
done independently of one another.

\subsubsection{Spatial Random
Permutations}\label{sec:spatial_random_permutations} A \emph{spatial
random permutation} is a random permutation which is biased towards
the identity in some underlying geometry. This broad idea covers
many models, among them the Mallows distribution which is biased
towards the identity in a one-dimensional geometry. In this section
we briefly describe two other models in this class for which related
results have been proved.

Let $G = (V,E)$ be a finite or infinite bounded-degree graph. The
\emph{interchange process} (also called the \emph{stirring process}
in some of the literature) gives a dynamics on permutations in
$\S_V$, one-to-one and onto functions $\pi\colon V\to V$, which is
associated to the structure of the graph. Each edge of the graph is
endowed with an independent Poisson process of rate $1$. An edge is
said to \emph{ring} at time $t$ if an event of its Poisson process
occurs at that time. Starting from the identity permutation
$\pi^0\in \S_V$, the interchange process, introduced by T\'oth
\cite{T93}, is the permutation-valued stochastic process $(\pi^t)$
obtained by performing a transposition along each edge at each time
that it rings.

The interchange process on the complete graph coincides with a
continuous time version of the Diaconis-Shashahani card shuffling
algorithm discussed in the Section~\ref{sec:card_shuffling}. Special
attention has been given to the case that the graph $G = \Z^d$,
where the interchange process is related to the magnetization of the
quantum Heisenberg ferromagnet \cite{T93}. In particular, the
following conjecture of B\'alint T\'oth has attracted significant
attention but remains unresolved: When $d=2$, for any $t>0$, all
cycles of $\pi^t$ are finite almost surely. In contrast, when $d\ge
3$ and $t$ is sufficiently large, $\pi^t$ has an infinite cycle
almost surely.

Besides the case of the complete graph, results on the existence of
long cycles in the interchange process are currently available only
for trees, by Angel \cite{A03} and Hammond \cite{H13, H15}, and for
the hypercube graph, by Koteck\'y, Mi\l o\'s and Ueltschi
\cite{KMU15}.

\pgfplotsset{width=\textwidth}
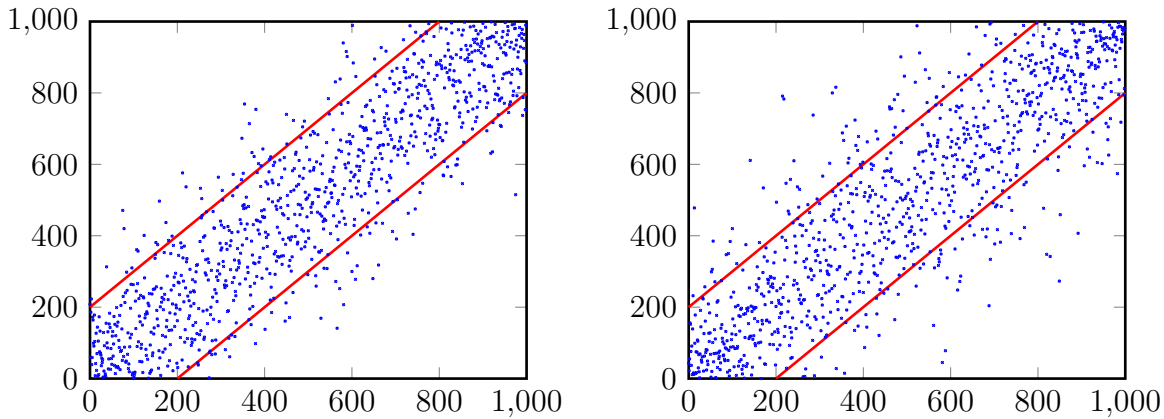
\begin{figure}
\begin{subfigure}[t]{0.46\textwidth}
\begin{tikzpicture}
\begin{axis}[
xmin = 0,
xmax = 1000,
ymin = 0,
ymax = 1000,
line width = 1pt,
]
\addplot[
red,
domain=0:800,
samples=801,
line width = 1pt
]
{x+200)};
\addplot[
red,
domain=200:1000,
samples=801,
line width = 1pt
]
{x-200)};
\addplot+[blue, only marks, mark size = 0.1pt] table {interchangeSample.dat};
\end{axis}
\end{tikzpicture}
\caption{
    Graph of a sample of the interchange process in a one-dimensional geometry with $n=1000$ and $t=10000$.
}
\end{subfigure}
\quad
\begin{subfigure}[t]{0.46\textwidth}
\begin{tikzpicture}
\begin{axis}[
xmin = 0,
xmax = 1000,
ymin = 0,
ymax = 1000,
line width = 1pt,
]
\addplot[
red,
domain=0:800,
samples=801,
line width = 1pt
]
{x+200)};
\addplot[
red,
domain=200:1000,
samples=801,
line width = 1pt
]
{x-200)};
\addplot+[blue, only marks, mark size = 0.1pt] table {mallowsSample1000n99q.dat};
\end{axis}
\end{tikzpicture}
\caption{
    Graph of a sample of the Mallows distribution $\mu_{n,q}$ with $n=1000$ and $q=0.99$.
}
\end{subfigure}
\caption{
    Comparison between the Mallows distribution and the interchange model.
}
\label{fig:interchange_vs_mallows}
\end{figure}

Recently, a quantitative study of the interchange process in a
one-dimensional geometry, $V = \{1,\ldots, n\}$ with $i$ adjacent to
$i+1$, was performed by Kozma and Sidoravicius. Here, as each
$(\pi^t_s)_{t>0}$ is a simple random walk, the typical displacement
$|\pi^t_s - s|$ is of order $\min\{\sqrt{t}, n\}$. Thus, the graph
of $\pi^t$ has a band structure similar to the graph of a Mallows
permutation, see Figure~\ref{fig:interchange_vs_mallows}, and the
two models seem graphically similar when one takes
\begin{equation}\label{eq:q_and_t}
  q = 1 - \frac{1}{1+\sqrt{t}}.
\end{equation}
In a work in preparation, Kozma and Sidoravicius prove that the
expected length of the cycle containing a given point in $\pi^t$ has
order $\min\{t+1, n\}$. This result, whose mathematical details were
completed before our work began, is analogous to our
Theorem~\ref{cycle_length} when making the assignment
\eqref{eq:q_and_t}.

A second model of spatial random permutations, related to the
Feynman-Kac representation of the ideal Bose gas in quantum
statistical mechanics, has also received significant attention, see
\cite{BU11} and references within. In this model, one samples a
random collection of points $(x_1,\ldots, x_n)$ in a finite box
$\Lambda\subset\R^d$ and a random permutation $\pi$ on these points.
The distribution is such that permutations with large displacements
$\pi(x_i) - x_i$ have lower density. In the physical context, the
emergence of macroscopic cycles in the model is related to the
phenomenon of Bose-Einstein condensation. In recent work, Betz and
Ueltschi \cite{BU11} (see also S\"ut\H o \cite{S02}) have shown that
the following phase transition takes place in the model when $d\ge
3$: Define the density of points per unit area $\rho =
\frac{n}{|\Lambda|}$. There exists a critical density $\rho_c$ such
that, with probability tending to $1$ as $n$ tends to infinity, if
the density is fixed to a value $\rho < \rho_c$ then all cycles have
length $o(n)$, whereas if it is fixed to a value $\rho > \rho_c$
then macroscopic cycles, of size proportional to $n$, emerge.
Moreover, in the second regime, the distribution of suitably
normalized cycle lengths converges in distribution to the
Poisson-Dirichlet law.

\subsection{Reader's Guide}
\hspace*{\parindent}Section~\ref{sec:prem} introduces notation and
preliminary facts used throughout the paper.

Section~\ref{section_gpac} develops flexible sampling methods for
the Mallows distribution and studies closely related random
processes: In Section~\ref{sec:generating_the_graph} we develop
tools for sampling the graph of a random Mallows permutation
sequentially. These are used in Section~\ref{sec:diagonal_exposure}
to introduce a `diagonal' exposure procedure for the graph. There,
we also define the `arc chain' of a permutation, which tracks the
number of open arcs (incomplete cycles) throughout the diagonal
exposure process, and analyze its basic properties. Concentration
bounds for the arc chain process are developed in
Section~\ref{sec:arc_chain_distribution} and used significantly in
later proofs. Section~\ref{section_pmc} provides bounds on the time
it takes an arc chain to reach zero, of use in the regime where $q$
is bounded away from $1$. Analogous bounds on return times also
appeared in the recent work \cite{BB16} of Basu and Bhatnagar where
a related Markov chain is introduced. Section~\ref{ggm} considers
the distribution of rectangular subsets of the graph of a Mallows
permutation. These provide the starting point for a method to sample
a Mallows permutation `from a mid-point', which is further extended
to a sampling method for the two-sided infinite case, when
$\pi:\Z\to\Z$.

Our main theorems are proved in Section~\ref{section_mt}:
Section~\ref{sec:num_of_cycles} is dedicated to the proof of
Theorem~\ref{thm:number_of_cycles}, regarding the number of cycles.
In Section~\ref{sec:cycle_diam} we prove
Theorem~\ref{cycle_diam_theorem} on the diameter of cycles by
providing deviation bounds for the distribution of the maximal and
minimal element of the cycle containing a given point. In
Section~\ref{sec:cycle_length} we prove Theorem~\ref{cycle_length}
regarding the length of cycles. Section~\ref{sec:variance_bounds} is
dedicated to the proof of Theorem~\ref{cycle_var} which provides
bounds on the variance of the cycle lengths. Theorem~\ref{thm:pdl}
on the Poisson-Dirichlet law is proved in Section~\ref{sec:pdl}.

We conclude in Section~\ref{sec:discussion} with a discussion and a
selection of open questions.
\section{Notation and Preliminaries}\label{sec:prem}
 \noindent $\bullet$ Throughout the rest of the paper $n$ is a positive integer whilst $q\in(0,1)$ is a real parameter.

\smallskip
  \noindent $\bullet$ For two quantities $x,y\geq 0$, which may depend on other parameters such
as $n$ or $q$, we write $x \lesssim y$ if there exists an absolute
constant $c>0$ such that $x \leq c\cdot y$. Note that $x\approx y$
is equivalent to $x\lesssim y$ and $y\lesssim x$.

\smallskip
  \noindent $\bullet$ $\N$ is the set of positive integers while $[n]:=\{i\in \N \mid i \leq n\}=\{1,2,\ldots,n\}$.

\smallskip
   \noindent $\bullet$ $\1_A$ and $\1\{A\}$ denote the indicator random variable of an event $A$.

\smallskip
\noindent $\bullet$ Throughout the paper we denote by $\xi=\xi_q$ the following
\begin{equation}\label{eq:xi_def}
    \xi:=\min\{i\in\N \mid q^i\leq \tfrac{1}{2}\}= \left\lceil  \log_q\tfrac{1}{2} \right\rceil \approx \frac{1}{1-q}.
\end{equation}

\noindent $\bullet$ In order to avoid cumbersome expressions we will use an abbreviated
notation when referring to subsets of $\Z^2$. We write, for instance
\[\{x<a,\,y<b\} \quad \text{instead of} \quad\{(x,y)\in \Z^2\mid
x<a,\,y<b\}\] and analogous expressions involving other subsets of
$\Z^2$.

\smallskip
 \noindent $\bullet$ We introduce two useful symmetries of the Mallows distribution
$\mu_{n,q}$, i.e., bijections $\S_n\leftrightarrow \S_n$ that
preserve $\mu_{n,q}$. The inverse symmetry is induced by the inversion map
\begin{equation}\label{inv_symmetry}
   \mu_{n,q}[\pi]=\mu_{n,q}[\pi^{-1}].
\end{equation}
The reversal symetry is defined via the reversal map $r\colon
s\mapsto n+1-s$ by
\begin{equation}\label{rev_symmetry}
    \mu_{n,q}[\pi]=\mu_{n,q}[r\circ \pi \circ r],
\end{equation}
where we note that $\pi_i=j$ if and only if $(r\circ \pi \circ
r)(n+1-i)=n+1-j$. The fact that the two maps $\pi\mapsto \pi^{-1}$
and $\pi\mapsto r\circ \pi \circ r$ preserve the Mallows
distribution follows simply by checking that they preserve the
number of inversions. These two symmetries will prove useful as they
also preserve the cycle structure. Specifically, if $\CC$ is a cycle
of $\pi$ then $\CC^{-1}$ and $r\circ\CC\circ r$ are cycles of
$\pi^{-1}$ and $r\circ\pi \circ r$, respectively.

\begin{section}{The Sampling Algorithm and the Arc Chain}
\label{section_gpac} In this section we present a sampling algorithm
for the Mallows distribution which will be fundamental in our
analysis. We further identify a Markov chain associated to this
sampling algorithm, termed the arc chain, and explore its basic
properties.

\subsection{Generating the Graph of a Mallows
Permutation}\label{sec:generating_the_graph}

In Section~\ref{sec:initial_sampling_algorithm} a method is
presented for sampling the values $(\pi_s)$ of a Mallows permutation
iteratively. Here we explain a related method which generates the
graph of the permutation
\begin{equation*}
  \Gamma_\pi:=\{(s,\pi_s)\mid s\ge 1\}
\end{equation*}
in an iterative manner, allowing to expose portions of the graph in
various orders.

Although our focus is on finite permutations, for clarity, we start
by discussing the case of infinite one-sided permutations
$\pi:\N\to\N$ in which the construction is simpler. In this case, as
explained in Section~\ref{sec:initial_sampling_algorithm},
\begin{equation*}
  \P[\pi_s = t\mid \pi_1, \ldots, \pi_{s-1}] = (1-q)\cdot
  q^{|[t]\smallsetminus\{\pi_1,\ldots,\pi_{s-1}\}|-1}\quad\text{for
  }\
  t\notin\{\pi_1,\ldots, \pi_{s-1}\}.
\end{equation*}
In other words, the value of $\pi_s$, conditioned on the values of
$\pi_1,\ldots, \pi_{s-1}$, has the geometric distribution with
success probability $1-q$ on the values in
$\N\smallsetminus\{\pi_1,\ldots, \pi_{s-1}\}$. This gives rise to
the following sampling method: starting with a two-dimensional
infinite array $(a_{s,t})_{s,t\ge 1}$ of independent Bernoulli
random variables, each satisfying
\begin{equation*}
\P[a_{s,t}=1] = 1 - \P[a_{s,t} = 0] = 1-q,
\end{equation*}
we may generate the permutation $\pi$ by setting
\begin{equation*}
  \pi_s := \min\{t\ge 1\mid t\notin\{\pi_1, \ldots, \pi_{s-1}\},\,
  a_{s,t} = 1\}.
\end{equation*}
Examination of this formula shows that the rule for deciding whether
the point $(s,t)$ belongs to the graph $\Gamma_\pi$ depends only on
the value of the bit $a_{s,t}$ and the portions of the graph
$\Gamma_{\pi}$ which lie strictly below $(s,t)$ or strictly to the
left of $(s,t)$,
\begin{equation}\label{eq:Gamma_pi_left_and_below}
  \Gamma_\pi\cap \{x<s,\, y=t\}\quad\text{and}\quad\Gamma_\pi\cap \{x=s,\,
  y<t\}.
\end{equation}
Precisely, $(s,t)\in \Gamma_{\pi}$ if and only if $a_{s,t} = 1$ and
the two sets in \eqref{eq:Gamma_pi_left_and_below} are empty. This
viewpoint allows for iterative generation of the graph
$\Gamma_{\pi}$ in many different manners. In the sequel we shall
focus on diagonal generation, in which we expose the portion of the
graph intersecting the square $\{x<t,\, y<t\}$ for increasing values
of $t$.

Our next lemma gives an analogous generating method for the graph of
a finite Mallows permutation, $\pi\in \S_n$, showing that many of
the essential features of the above construction are preserved.

\begin{lemma}\label{atomic_exposure}
    Let $\pi\sim\mu_{n,q}$, let $s,t \in [n]$ and set $U:=\{x<s \text{ or }
    y<t\}$. Then
    \begin{equation} \label{atomic_exposure_eq}
        \P[\pi_s=t \mid \Gamma_\pi\cap U]= \frac{1-q}{1-q^{|\Gamma_\pi \cap U^\complement|}} \cdot
    \1 \left\{ \begin{array}{rl} \Gamma_\pi\cap \{x<s,\,y=t\}=\emptyset \\  \Gamma_\pi\cap\{x=s,\,y<t\}=\emptyset \end{array} \right\}.
    \end{equation}
\end{lemma}
We point out that the right-hand side of \eqref{atomic_exposure_eq}
does not depend on the full information in $\Gamma_{\pi}\cap U$.
Indeed, to evaluate the right-hand side one only needs to know
whether the sets in \eqref{eq:Gamma_pi_left_and_below} are empty and
the size of the set $|\Gamma_\pi \cap U^\complement|$, which may be
computed, for instance, via
\begin{equation}\label{eq:number_of_points_in_corner}
    |\Gamma_\pi\cap U^\complement| = n - s - t + 2 + |\Gamma_\pi\cap \{x<s, \, y<t\}|.
\end{equation}
As one application, one may use the equality
\eqref{atomic_exposure_eq} iteratively to compute the probability
distribution of the portion of the graph $\Gamma_\pi\cap\{x\ge
i,\,y<j\}$ conditioned on the portion of the graph
$\Gamma_\pi\cap\{x< i,\,y<j\}$. The equality
\eqref{atomic_exposure_eq} shows that this probability distribution
remains the same if we condition additionally on $\Gamma_\pi\cap\{x<
i,\,y \ge j\}$. Therefore, we obtain the following conditional
independence statement: for each $i,j\in[n]$,
\begin{equation}\label{eq:conditional_independence_in_graph}
\begin{split}
  &\text{conditionally on }\Gamma_\pi\cap\{x<i,\, y<j\},\\
  &\qquad\Gamma_\pi\cap\{x\ge i,\,y<j\}\text{ and }\Gamma_\pi\cap\{x<i,\, y\ge j\}\text{ are independent}.
\end{split}
\end{equation}
\begin{proof}[Proof of Lemma~\ref{atomic_exposure}] Our proof relies upon the formula \eqref{eq:sampling_formula}. In the notation used there,
    \begin{equation}\label{eq:sampling_formula_for_proof}
      \P[\pi_{s} = j_k\mid \pi_1, \ldots, \pi_{s-1}, \pi_s\ge j_k] =
  \frac{1-q}{1-q^{n-s-k+2}}.
    \end{equation}
    We first claim that
    \begin{equation}\label{eq:atomic_exposure_proof}
        \P[\pi_s=t \mid \Gamma_\pi \cap U \cap \{x\leq s\}]= \frac{1-q}{1-q^{|\Gamma_\pi \cap U^\complement|}} \cdot
    \1_A,
    \end{equation}
    where $A:=\left\{\Gamma_\pi\cap
    \{x<s,\,y=t\}=\emptyset\right\}\cap\left\{\Gamma_\pi\cap\{x=s,\,y<t\}=\emptyset\right\}$. The equality \eqref{eq:atomic_exposure_proof}
    certainly holds on $A^\complement$ as both sides are zero.
    Now set $k := |[t-1]\smallsetminus\{\pi_1,\ldots,\pi_{s-1}\}|$ and note that, by \eqref{eq:number_of_points_in_corner}, $|\Gamma_\pi \cap U^\complement| =
    n-s-k+2$. Observe that, in the notation of \eqref{eq:sampling_formula_for_proof}, we have that $t = j_k$ on the event $A$, as the set $\{\pi_1,\ldots,
    \pi_{s-1}\}$ misses exactly $k-1$ elements out of $[t-1]$. It
    then follows from formula \eqref{eq:sampling_formula_for_proof} that
    \begin{equation*}
      \P[\pi_s=t \mid \Gamma_\pi \cap U \cap \{x\leq s\}]= \frac{1-q}{1-q^{n-s-k+2}} = \frac{1-q}{1-q^{|\Gamma_\pi \cap U^\complement|}}\quad\text{on
      $A$},
    \end{equation*}
    finishing the proof of \eqref{eq:atomic_exposure_proof}.

    Next, we observe that the argument used above to
    derive \eqref{eq:conditional_independence_in_graph} from
    \eqref{atomic_exposure_eq} may also be used to derive
    \eqref{eq:conditional_independence_in_graph} from
    \eqref{eq:atomic_exposure_proof}. Applying
    \eqref{eq:conditional_independence_in_graph} with
    $i=s+1$ and $j=t$ shows
    that $\pi_s$ is conditionally independent of $\Gamma_\pi \cap U \cap \{x >
    s\}$ conditioned on $\Gamma_\pi \cap U \cap \{x\leq s\}$.
    Thus the formula \eqref{atomic_exposure_eq} is a consequence of
    \eqref{eq:atomic_exposure_proof}.
\end{proof}

\subsection{Diagonal Exposure and the Arc Chain
Process}\label{sec:diagonal_exposure}

The main lemma of the previous section, Lemma~\ref{atomic_exposure},
provides a procedure for calculating the distribution of certain
portions of the graph $\Gamma_\pi$, of a Mallows permutation $\pi$,
given others. This gives rise to several iterative algorithms for
exposing the full graph. The proofs of our main theorems rely on a
particular method of exposing $\Gamma_{\pi}$ which will turn out to
be particularly convenient. The total portion of $\Gamma_\pi$ that will be revealed by time $t$ will consist of $\Gamma_\pi \cap \{x\le t,y \le t\}$.
Equivalently, as we pass from time $t$ to $t+1$ we reveal
\begin{equation}\label{eq:t_iteration_new_portions}
    \Gamma_\pi\cap\{x=t+1,\, y \le t\}, \quad \Gamma_\pi\cap\{x\le t,\, y = t+1\} \quad \text{and} \quad \Gamma_\pi\cap\{x = t+1,\, y = t+1\}.
\end{equation}
Formally, we define a finite filtration consisting of the sigma-algebras
\[
    \CF_0\subseteq \CF_1 \subseteq \ldots \subseteq \CF_{n-1} \subseteq \CF_{n}
\]
defined by
\begin{equation}\label{eq:F_t_definition}
    \CF_t:=\sigma(\Gamma_\pi \cap \{x \leq t,y \leq t\}).
\end{equation}
Thus, $\CF_0$ is the trivial $\sigma$-algebra and $\CF_{n}$ is the
$\sigma$-algebra generated by $\pi$. We call this exposure procedure
the \emph{diagonal exposure} of $\pi$ as the procedure exposes the
graph in the diagonal direction. Corresponding to this filtration we
introduce the notation
\[
    \P_t[A]:=\P[A \mid \CF_t] \quad \text{and} \quad \E_t[X]:=\E[X \mid
    \CF_t],
\]
for an event $A$ and a random variable $X$.

An important quantity to keep track of during the diagonal exposure
process is the number of elements of $\pi$ in the revealed portion
of the graph $\Gamma_\pi$ at each time $t$, i.e.,
$|\Gamma_{\pi}\cap\{x \leq t,\,y \leq t\}|$. Our next definition
introduces the counting process of an equivalent quantity,
$|\Gamma_{\pi}\cap\{x \leq t,\,y > t\}|$, which will appear more
frequently in our analysis. This quantity, as we elaborate upon in
Section~\ref{section_mt}, counts the number of open `arcs', i.e.,
portions of cycles that are yet to be closed, which are known using
the information in $\CF_t$.
\begin{definition}[Arc Chain]
    The \emph{arc chain} $(\kappa_t)$, $0\le t\le n$, of a permutation
    $\pi\in\S_n$ is defined by
    \begin{equation}\label{arc_chain_def}
        \kappa_t = \kappa_t(\pi) := |\{i\in[t] \mid \pi_i > t\}| = t - |\Gamma_\pi \cap \{x\leq t, y\leq t\}|,
    \end{equation}
    that is, $\kappa_t$ counts the number of $\pi_1,\ldots,\pi_{t}$ that are
greater than $t$.
\end{definition}

The arc chain is adapted to the filtration $(\CF_t)$, that is,
$\kappa_t$ is determined by $\CF_t$. One should note that $\pi$ and
$\pi^{-1}$ share the same arc chain, that is,
\begin{equation}\label{eq:kappa_t_symmetric_def}
   \kappa_t = |\{t<i\le n \mid \pi_i\le t\}|.
\end{equation}
The next proposition formalizes the fact that $(\kappa_t)$ is a
time-inhomogeneous Markov chain.
\begin{proposition}\label{mallows_chain_transition_prob}
    Let $\pi\sim\mu_{n,q}$. The arc chain $\kappa$ of $\pi$ is a time-inhomogeneous Markov chain, with respect to the filtration $(\CF_t)$, satisfying $\kappa_0=0$ and $|\kappa_{t+1}-\kappa_{t}|\leq
    1$, with transition probabilities given by
    \begin{equation}
    \begin{split}
           \label{eq:arc_chain_transition_prob}
            \P_t \left[ \kappa_{t+1}=\kappa_t-1 \right] &=
            \left(\frac{1-q^{\kappa_t}}{1-q^{n-t}}\right)^2 , \\
            \P_t \left[ \kappa_{t+1}=\kappa_t \right] &=
            \frac{q^{\kappa_t} - q^{n-t}}{1-q^{n-t}}\cdot\frac{2 - q^{\kappa_t} - q^{\kappa_t + 1}}{1-q^{n-t}}
            ,\\
            \P_t \left[ \kappa_{t+1}=\kappa_t+1 \right] &=
            \frac{q^{\kappa_t}-q^{n-t}}{1-q^{n-t}} \cdot
            \frac{q^{\kappa_t+1}-q^{n-t}}{1-q^{n-t}}.
    \end{split}
    \end{equation}
\end{proposition}
As an illustration of the usefulness of the arc chain, we note that
the probability that $\pi$ has a fixed point at position $t+1$,
given the information in $\CF_t$, has a simple expression in terms
of $\kappa_t$.
\begin{lemma}\label{fix_point_equation}
    Let $\pi\sim\mu_{n,q}$ and $\kappa$ be its arc chain. Then
    \[
        \P_t
        [\pi_{t+1}=t+1]=\frac{q^{\kappa_t}-q^{\kappa_t+1}}{1-q^{n-t}}\cdot\frac{q^{\kappa_t}-q^{n-t}}{1-q^{n-t}},\qquad
        0\le t < n.
    \]
\end{lemma}
We prove the proposition and lemma together.
\begin{proof}[Proof of Proposition~\ref{mallows_chain_transition_prob} and Lemma~\ref{fix_point_equation}]
The newly revealed portions of $\Gamma_\pi$ at time $t+1$ were
described in \eqref{eq:t_iteration_new_portions}. Denote the first
two of these portions by
\[
    X:=\Gamma_\pi\cap \{x = t + 1,  y \le t\} \quad \text{and} \quad Y:=\Gamma_\pi\cap \{x \le t,  y = t+1\}.
\]
We claim that
\begin{equation}\label{diag_exposure_eq1}
    \P_t[X\neq \emptyset]=\P_t[Y\neq \emptyset] =
    \frac{1-q^{\kappa_t}}{1-q^{n-t}}.
\end{equation}
 It is convenient to derive this directly from
\eqref{eq:sampling_formula}. Write $j_1<j_2<\cdots<j_{n-t}$ for the
values in $[n]\smallsetminus\{\pi_1, \ldots, \pi_{t}\}$ and observe
that $j_k \le t$ if and only if $k \le \kappa_t$, see
\eqref{eq:kappa_t_symmetric_def}. Therefore, it follows from
\eqref{eq:sampling_formula} that
\begin{equation*}
  \P_t[X\neq \emptyset]=\E_t\big[\P[X\neq \emptyset\mid \pi_1,\ldots,
  \pi_{t}]\big] = \E_t\left[\sum_{k=1}^{\kappa_t} \frac{1-q}{1-q^{n-t}} \cdot
  q^{k-1}\right] = \frac{1-q^{\kappa_t}}{1-q^{n-t}}.
\end{equation*}
The equality $\P_t[X\neq \emptyset]=\P_t[Y\neq \emptyset]$ follows
from the symmetry \eqref{inv_symmetry}, as $\pi$ and $\pi^{-1}$
share the same arc chain.

Now note that
\[
    \P_t[\pi_{t+1}=t+1] = \P_t[X=Y=\emptyset] \cdot \P_t[\pi_{t+1}=t+1\mid
    X=Y=\emptyset].
\]
The second factor can be computed directly from
Lemma~\ref{atomic_exposure},
\begin{equation*}
  \P_t[\pi_{t+1}=t+1\mid X=Y=\emptyset] =
  \frac{1-q}{1-q^{n-t-\kappa_t}}.
\end{equation*}
In addition, observe that $X$ and $Y$ are conditionally independent
given $\CF_t$, as follows from
\eqref{eq:conditional_independence_in_graph}. Thus the value of the
first factor may be calculated from \eqref{diag_exposure_eq1}, which
completes the proof of Lemma~\ref{fix_point_equation}.

Observe that $\kappa_{t+1}-\kappa_t = 1 -|Z|$, where $Z:= X\cup Y
\cup (\Gamma_\pi \cap \{(t+1,t+1)\})$, as follows from the
definition~\eqref{arc_chain_def} of $(\kappa_t)$. As $|Z|\le 2$, it
follows that $|\kappa_{t+1}-\kappa_t|\leq 1$. The equations in
\eqref{eq:arc_chain_transition_prob} can be verified by computations
similar to the ones used to prove Lemma~\ref{fix_point_equation}, as
\begin{align*}
    \P_t[\kappa_{t+1}=\kappa_t-1]
    &=\P_t[|Z|=2]
    =\P_t[X\neq \emptyset \text{ and }Y\neq \emptyset], \\
    \P_t[\kappa_{t+1}=\kappa_t+1]
    &=\P_t[|Z|=0]
    =\P_t[X=Y=\emptyset]\cdot\P_t[\pi_{t+1}\neq t+1\mid
    X=Y=\emptyset],
\end{align*}
and the value of $\P_t[\kappa_{t+1}=\kappa_t]$ is derived from these
using that $|\kappa_{t+1}-\kappa_{t}|\leq 1$.
\end{proof}

\subsection{The Distribution of the Arc Chain} \label{sec:arc_chain_distribution}
In this section we study the distribution of the arc chain of a
Mallows permutation at a fixed time $t$. Our main result,
Theorem~\ref{mallows_chain_bounds} below, states that the value of
the chain is unlikely to be much larger than the value of $\xi$
given in \eqref{eq:xi_def}. Figure~\ref{fig:kappa} depicts the
percentiles of the distribution of the arc chain for certain values
of $n$ and $q$ and all times~$t$. These suggest that the typical
values of the arc chain are close to $\xi$ when $t$ is bounded away
from $1$ and $n$. For such times, we establish in
Proposition~\ref{prop:arc_chain_limit} below a formula for the
limiting distribution of the arc chain when $n$ tends to infinity
with $q$ fixed.

It is convenient to refer to the arc chain as an abstract Markov
chain, without reference to an underlying Mallows permutation, as
facilitated by the following definition.

\begin{definition}
    A random sequence $(\kappa_t)$, $0\le t\le n$, is an \emph{$(n,q)$-arc chain}, denoted $\kappa\sim\AC_{n,q}$,
    if $\kappa$ is a time-inhomogeneous Markov chain with transition probabilities as in Proposition~\ref{mallows_chain_transition_prob}
    and some initial distribution $\kappa_0$ supported in $\{0,\ldots,
    n\}$.
\end{definition}
We point out that the formulas in
Proposition~\ref{mallows_chain_transition_prob} constitute valid
transition probabilities (that is, they are non-negative and sum to
1) when $0\le \kappa_t\le n-t$. Using induction on $t$, one checks
that an $(n,q)$-arc chain satisfies this condition for all $t$
almost surely.

\pgfplotsset{width=\textwidth}
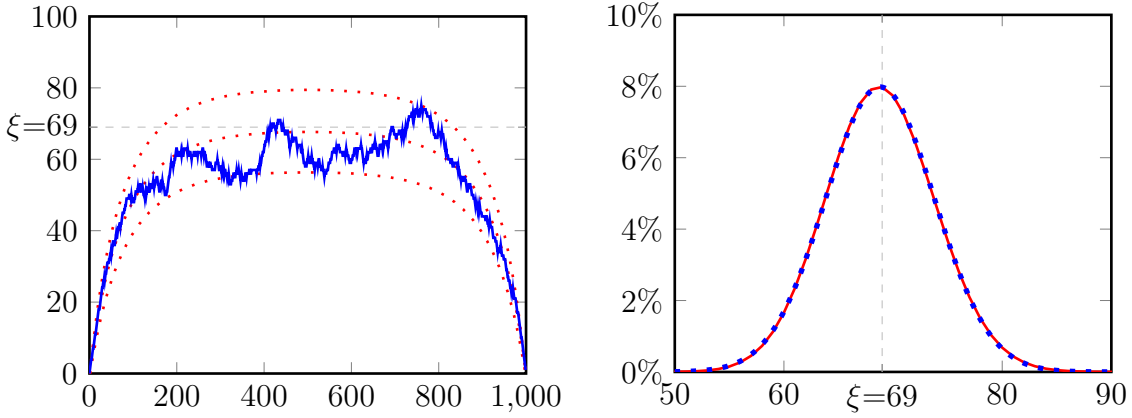
\begin{figure}
\begin{subfigure}[t]{0.46\textwidth}
\begin{tikzpicture}
\begin{axis}[
xmin = 0, xmax = 1000, ymin = 0, ymax = 100, line width = 1pt, extra
y ticks={69}, extra y tick style={grid=major, dashed}, extra y tick
labels={$\xi$=69} ] \addplot[red,loosely dotted] table
{KappaDist01p.dat}; \addplot[red,loosely dotted] table
{KappaDist50p.dat}; \addplot[red,loosely dotted] table
{KappaDist99p.dat}; \addplot[blue] table {KappaSample.dat};

\end{axis}
\end{tikzpicture}
\caption{
    The blue line depicts a single sample of an arc chain.
    The red dotted lines depict (linearly interpolated) percentiles of the distribution of the arc chain: 1st, 50th and 99th.
}
\label{fig:kappa}
\end{subfigure}
\quad
\begin{subfigure}[t]{0.46\textwidth}
\begin{tikzpicture}
\begin{axis}[
xmin = 50, xmax = 90, ymin = 0, ymax = 10, ytick = {0,2,4,6,8,10},
yticklabels = {0\%,2\%,4\%,6\%,8\%,10\%}, xtick = {50,60,80,90},
line width = 1pt, extra x ticks={69}, extra x tick
style={grid=major, dashed}, extra x tick labels={$\xi$=69} ]
\addplot[red] table {KappaDist.dat}; \addplot[ blue, domain=50:90,
samples=1001, loosely dotted, line width = 2pt ]
{100*exp(-(x-ln(2)*100+0.5)^2 / (2*5^2)) / (5 * sqrt(2*pi))};

\end{axis}
\end{tikzpicture}
\caption{
    The red line is the graph of the stationary distribution of an $(\infty,0.99)$-arc chain.
    The blue dotted line is the graph of a normal distribution.
}
\label{fig:kappa_distr}
\end{subfigure}
\caption{Distributions associated with the arc chain of a Mallows
permutation with parameters $n=1000$ and $q=0.99$.}
\end{figure}
\begin{theorem} \label{mallows_chain_bounds}
    Let $\kappa\sim\AC_{n,q}$ with $\kappa_0=0$. Then for all $0\le t \le n$ and $d\in\N$ we have
    \begin{equation} \label{mc_single}
        \P[\kappa_t > \xi+d] \leq \frac{q^{d^2+d}}{1-q^{2d}}.
    \end{equation}
\end{theorem}

The idea of proof involves the definition of a limiting
time-homogeneous Markov chain, corresponding formally to the case
that $n=\infty$, and bounding the distribution of the $(n,q)$-arc
chain by the stationary distribution of the limiting chain. Some of
the tools that we develop here will be used later in the paper as
well.

We recall that a time-homogeneous Markov chain $(\alpha_t)$ which
takes values in the non-negative integers and satisfies
$|\alpha_t-\alpha_{t+1}|\leq 1$ is called a \emph{birth-and-death
chain}.
\begin{definition}\label{def:infinity_q_arc_chain}
    A random sequence $(\kh_t)_{t\ge 0}$ is an \emph{$(\infty,q)$-arc chain}, denoted $\kh~\sim~\AC_{\infty,q}$,
    if $\kappa$ is a birth-and-death chain with transition
    probabilities given by
    \begin{equation}
    \begin{split}\label{eq:infinite_arc_chain_transition_prob}
      &\P[ \kh_{t+1} = \kh_t - 1\mid \kh_t]=(1-q^{\kh_t})^2,\\
      &\P[ \kh_{t+1} = \kh_t\mid \kh_t]=2 q^{\kh_t} - q^{2\kh_t} - q^{2\kh_t+1},\\
      &\P[ \kh_{t+1} = \kh_t + 1\mid \kh_t]=q^{2\kh_t+1}.
    \end{split}
    \end{equation}
\end{definition}
There is a formula for the stationary measure of a birth-and-death
chain. If $(\alpha_t)$ is a birth-and-death chain taking values in
$\{0,\ldots, m\}$, where $m$ may be finite or infinite, having
positive transition probabilities between consecutive integers in
$\{0,\ldots, m\}$, then $\alpha$ has a stationary measure $z$
defined by\footnote{Here and later in the paper, to avoid
introducing extra notation, we denote the transition probability of
the chain from $i-1$ to $i$ at time $t$ by $\P[\alpha_{t+1} = i \mid
\alpha_t=i-1]$, even if $\P[\alpha_t=i-1]=0$, and use similar
notation for other transition probabilities.}
\begin{equation}
        z_0:=1\quad\text{and}\quad z_s:=\prod_{i=1}^s \frac{\P[ \alpha_{t+1} = i \mid \alpha_t=i-1]}{\P[\alpha_{t+1} = i-1 \mid
        \alpha_t=i]}\quad\text{for $1\le s\le m$}. \label{stationary_measure_eq}
\end{equation}
This is straightforward to verify directly. It is also not difficult
to check that the stationary measure is unique up to scaling though
we shall not use this fact (see also
\cite[Section~2.5]{mixing_times}).

This fact allows us to find the stationary distribution of an
$(\infty,q)$-arc chain $\kh$. Put
\begin{equation}\label{bound_chain_transitions}
\begin{aligned}
    u_i &:= \P[ \kh_{t+1} = i \mid \kh_t=i-1]=q^{2i-1} > 0, &i\ge 1,\\
    v_i &:= \P[ \kh_{t+1} = i-1\mid \kh_t=i ]=(1-q^i)^2 > 0, &i\ge 1.
\end{aligned}
\end{equation}
Then the sequence $(\nu_s)$, $s\ge 0$, defined by
\begin{equation}\label{eq:bounding_chain_stat_dist}
    \nu_s := \frac{\prod_{i=1}^s \frac{u_i}{v_i}}{\sum_{j\geq 0}\prod_{i=1}^j \frac{u_i}{v_i}}
\end{equation}
defines a stationary distribution for $\kh$, where, as usual, an
empty product is interpreted as $1$. The denominator in
\eqref{eq:bounding_chain_stat_dist} is finite since $u_i\rightarrow
0$ and $v_i\rightarrow 1$ as $i\to \infty$. See
Figure~\ref{fig:kappa_distr} for a graph of $\nu$.

We study further the relation between the distributions of the
$(n,q)$-arc chain $\kappa$ and the $(\infty,q)$-arc chain $\kh$. Our
next proposition shows that in a suitable limit, in which $q$ is
fixed, the distribution of $\kappa_t$ converges to the stationary
distribution of $\kh$. This proposition will be of use in
Section~\ref{ggm}.
\begin{proposition}\label{prop:arc_chain_limit}
Let $\kappa\sim\AC_{n,q}$ with $\kappa_0=0$ and set $t=t_n$. If both
$t\to\infty$ and $n-t\to\infty$ then the law of $\kappa_t$ converges
to the stationary distribution of $\AC_{\infty,q}$, as $n$ tends to
infinity with $q$ fixed.
\end{proposition}

Our main tool for proving the above theorem and proposition is a
coupling in which the $(\infty,q)$-arc chain bounds the $(n,q)$-arc
chain at all times. We first introduce a general method for
performing such couplings.

Let $\alpha$ be a Markov chain, possibly time-inhomogeneous, taking
values in the non-negative integers and satisfying
$|\alpha_t-\alpha_{t+1}|\leq 1$. Let $(U_t)$ be a sequence of
independent random variables, each uniformly distributed on $[0,1]$.
One may couple the Markov chain $\alpha$ with the sequence $U$ as
follows. The initial distribution $\alpha_0$ is taken independent of
$U$. Then, for each $t\ge 0$, $\alpha_{t+1}:=\alpha_t + F_t^\alpha
(\alpha_t, U_t)$ where
\[
   F_t^\alpha (a,u):=\1_{\{u > 1-\P[\alpha_{t+1} = a+1\,\mid\, \alpha_t=a]\}} - \1_{\{u\leq
   \P[\alpha_{t+1} = a-1\,\mid\, \alpha_t=a]\}}.
\]
This can be understood as `$\alpha_{t+1}$ is a monotone function of
$U_t$ for a given $\alpha_t$'. We say that a set of Markov chains of
the above type is \emph{monotonically coupled} if they are all
coupled to the same sequence $U$ via the above method.

\begin{proposition}  \label{arc_chain_monotonicity}
    Let $q,\hat{q}\in(0,1)$ and $n,\hat{n}\in \N\cup\{\infty\}$ satisfy $q \leq \hat{q}$ and $n\leq \hat{n}$.
    Let an $(n,q)$-arc chain $\kappa$ be monotonically coupled with an $(\hat{n},\hat{q})$-arc chain $\kh$
    that satisfies $\kappa_0\leq \kh_0$ almost surely. Then, almost surely, $\kappa_t\le \kh_t$ for all $0\le t \le n$.
\end{proposition}

\begin{definition}[Bounding Chain]\label{bounding_chain_def}
    For an $(n,q)$-arc chain $\kappa$, a \emph{bounding chain} is any $(\infty,q)$-arc chain $\kh$ that
    is monotonically coupled with $\kappa$ and satisfies $ \kappa_0 \le \kh_0$ almost surely.
   Proposition~\ref{arc_chain_monotonicity} implies that any bounding chain satisfies $\kappa_t \leq \kh_t$, almost surely, for all $0\le t \le n$.
\end{definition}
\begin{proof}[Proof of Proposition~\ref{arc_chain_monotonicity}]
    The proof relies on the facts that $|\kappa_{t+1}-\kappa_t|\leq 1$, $|\kh_{t+1}-\kh_t|\leq 1$ and the following three
    inequalities,
\begin{align}
    &\P[\kappa_{t+1} = k+1 \mid \kappa_t = k] \leq 1 -  \P[\kh_{t+1} = k \mid \kh_t=k+1], && 0\leq k \leq \min\{n-t, \hat{n}-t-1\},\nonumber\\
    &\P[\kappa_{t+1} = k - 1 \mid
    \kappa_t=k]\ge \P[ \kh_{t+1} = k-1 \mid \kh_t=k], &&0\leq k\leq n-t, \label{eq:kappa_kh_ineqs}\\
    &\P[\kappa_{t+1} = k + 1 \mid
    \kappa_t=k] \le \P[ \kh_{t+1} = k+1 \mid \kh_t=k], &&0\leq k\leq n-t. \nonumber
\end{align}
To prove these inequalities, observe that the transition
probabilities of arc chains are given by
\begin{equation}\label{eq:arc_chain_transition_prob2}
\begin{split}
  &\P[\kappa_{t+1} = k - 1 \mid \kappa_t=k] =
\left(\frac{1-q^k}{1-q^{n-t}}\right)^2,\\
  &\P[\kappa_{t+1} = k + 1
\mid \kappa_t=k] = \left(1-\frac{1-q^{k}}{1-q^{n-t}}\right) \cdot
            \left(1-\frac{1-q^{k+1}}{1-q^{n-t}}\right).
\end{split}
\end{equation}
Thus, the last two inequalities in \eqref{eq:kappa_kh_ineqs} follow
from the fact that $\frac{1-q^a}{1-q^{b}} =
\frac{1+q+\cdots+q^{a-1}}{1+q+\cdots+q^{b-1}}$, with $1\le a\le b$,
decreases with both $b$ and $q$. The first inequality in
\eqref{eq:kappa_kh_ineqs} follows from the third inequality there
and the fact that $\P[\kh_{t+1} = k+1 \mid \kh_t = k] \leq 1 -
\P[\kh_{t+1} = k \mid \kh_t=k+1]$. This last fact follows by
substituting the formulas in \eqref{eq:arc_chain_transition_prob2},
using that $1-x^2 = (1-x)(1+x)$ and taking out the non-negative
common factor $1 - \frac{1-\hat{q}^{k+1}}{1-\hat{q}^{n-t}}$.

We proceed to prove the proposition. Suppose $t < n$ is such that
$\kappa_t\le \kh_t$ almost surely and let us show that
$\kappa_{t+1}\le\kh_{t+1}$, almost surely. Recall that $0\le
\kappa_t\le n-t$ and $0\le \kh_t\le \hat{n}-t$ almost surely, and
let us consider separately the following three cases.
    \begin{itemize}[noitemsep,topsep=3pt]
      \item The inequality is clear if $\kh_t-\kappa_t \geq 2$.
      \item If $\kappa_t = \kh_t - 1$ then $\kappa_{t+1}\le\kh_{t+1}$  follows from the first inequality in \eqref{eq:kappa_kh_ineqs}.
    \item Lastly, if $\kappa_t=\kh_t$ then $\kappa_{t+1}\le\kh_{t+1}$  is a consequence of the second and third inequality in \eqref{eq:kappa_kh_ineqs}.\qedhere
    \end{itemize}
\end{proof}
As a corollary of this proposition we deduce that an $(n,q)$-arc
chain $\kappa$ with $\kappa_0 = 0$ satisfies
\begin{equation}\label{mallows_corr_bounds}
    \P[\kappa_t\geq d]\leq \nu[d,\infty] = \sum_{i\geq d} \nu_i \quad \text{for all $t\le n$ and $d\geq 0$,}
\end{equation}
where $\nu$ is the stationary distribution of an $(\infty,q)$-arc
chain, as given by \eqref{eq:bounding_chain_stat_dist}. This follows
by letting $\kh$ be the bounding chain of $\kappa$ having $\kh_0\sim
\nu$. Then \eqref{mallows_corr_bounds} is a consequence of the facts
that $\kh_t\sim\nu$ and $\kappa_t\leq \kh_t$.

We are now ready to prove Theorem~\ref{mallows_chain_bounds} and Proposition~\ref{prop:arc_chain_limit}.
\begin{proof}[Proof of Theorem~\ref{mallows_chain_bounds}]
    By \eqref{mallows_corr_bounds}, it suffices to prove that the stationary distribution $\nu$ of an $(\infty,q)$-arc chain satisfies
   \begin{equation}\label{eq:mallows_chain_bounds_eq1}
        \nu[\xi + d + 1,\infty] \leq
        \frac{q^{d^2+d}}{1-q^{2d}},\quad d\ge 1.
   \end{equation}
    Let $u_i$ and $v_i$ be as in \eqref{bound_chain_transitions}  and set $w_i:=u_i/v_i$. Formula~\eqref{eq:bounding_chain_stat_dist}
    implies that
    \begin{equation}\label{eq:mallows_chain_bounds_eq2}
        \nu_{\xi+s+1} = \nu_{\xi+1}\cdot \prod_{j = 1}^{s} w_{\xi + j + 1} \leq \prod_{j = 1}^{s} w_{\xi+ j + 1}, \quad s\ge 1.
    \end{equation}
    Observe that $u_{\xi+1}\leq \tfrac{1}{4}$ and $v_{\xi+1}\geq \tfrac{1}{4}$, by the definition \eqref{eq:xi_def} of $\xi$, yielding $w_{\xi+1}\leq 1$.
    One may verify that $w_{i+1}\leq q^2\cdot w_i$, which yields that $w_{\xi+ j +1}\leq  q^{2j}\cdot w_{\xi+1} \leq q^{2j}$.
        By substituting this in \eqref{eq:mallows_chain_bounds_eq2} we
        conclude that
    \[
                \nu_{\xi+s+1}\le q^{s^2+s}.
     \]
    Summing this inequality over $s\geq d$ yields
    \eqref{eq:mallows_chain_bounds_eq1}, by bounding the sum with a geometric progression with quotient $q^{2d}$.
\end{proof}

    \begin{proof}[Proof of Proposition~\ref{prop:arc_chain_limit}]
Recall the \emph{convergence theorem for finite-state Markov chains}:
    if a finite-state time-homogeneous Markov chain $(a_t)$ is aperiodic and irreducible then it has a stationary
    distribution and the distribution of $a_t$ converges to this stationary distribution as
    $t\to\infty$.

       Equation \eqref{mallows_corr_bounds} states that $\nu$ dominates $\kappa_t$.
    We will construct distributions $\tilde\nu_h$ that are asymptotically dominated
by $\kappa_t$ as $t\to\infty$. Then, with the limit of $\kappa_t$ sandwiched between $\tilde\nu_h$
and $\nu$, we will show that $\tilde\nu_h$ approaches $\nu$ as $h\to\infty$.
        Let $h$ be some fixed positive integer and assume without loss of generality that $n-t\geq h$.
        Let $\tilde \kappa$ be the birth-and-death chain having $\kappa_0=0$ and transition probabilities determined by
    \[
        \P[\kt_{t+1} =\kt_t - 1  \mid \kt_t ]=\left( \frac{1-q^{\kt_t}}{1-q^h} \right)^2  \ \text{and} \ \
        \P[\kt_{t+1} =\kt_t + 1  \mid \kt_t]=\frac{q^{\kt_t}-q^h}{1-q^h}  \frac{q^{\kt_t+1}-q^h}{1-q^h}.
  \]
        Observe that $0\le\kt_t\le h$ for all $t$, almost surely.
        Let $\tilde \kappa$ be monotonically coupled with $\kappa$.
        It is not hard to check that the pair $(\tilde \kappa, \kappa)$ satisfies the analogous inequalities of $\eqref{eq:kappa_kh_ineqs}$ for $t \le
        n-h$, which implies, by following the proof of
        Proposition~\ref{arc_chain_monotonicity}, that $\tilde \kappa_t\leq \kappa_t$ for all $t\leq n-h$, almost surely.

        By applying the convergence theorem for finite-state Markov chains we obtain that  $\tilde \kappa$ has stationary distribution $\tilde \nu = \tilde \nu_h$
       and that $\tilde \kappa_t$ converges to $\tilde \nu$ in distribution. Since $\kappa_t$ dominates $\tilde \kappa_t$ for all $t\leq n-h$, we obtain
        \[
            \tilde\nu[d,\infty] \leq \liminf \P[\kappa_t\geq d] \leq \limsup \P[\kappa_t\geq d] \leq\nu[d,\infty]
    \quad \text{for $d\geq 0$,}
        \]
        where the limits are taken for $n,t\to\infty$ with the restriction $n-t\geq
        h$. It remains to verify, using
        \eqref{stationary_measure_eq} and the fact that the ratios $\frac{\P[\kt_{t+1} =i  \mid \kt_t=i-1 ]}{\P[\kt_{t+1} =i - 1  \mid \kt_t=i]}$ increase with $h$, that $\tilde{\nu}_s\to\nu_s$
        as $h\to\infty$, for all $s$. Thus, $\tilde\nu[d,\infty]$ converges to $\nu[d,\infty]$ as $h\to\infty$, completing the proof of the proposition.
    \end{proof}

\begin{subsection}{The Hitting Time of Zero}\label{section_pmc}
The times in which the arc chain is at zero can be thought of as cut
points for the graph of the permutation in the sense that if
$\kappa_t=0$ then $\Gamma_{\pi}\subseteq \{x\le t, y\le t\}\cup\{x>
t, y>t\}$. This leads one to consider the evolution of the arc chain
as performing a sequence of excursions away from zero; a point of
view which will be useful for us in the regime that $q$ is bounded
away from $1$ since, as we now prove, the excursions tend to be
relatively short in this regime.

The recent work of Basu and Bhatnagar \cite{BB16} uses a similar
viewpoint in their analysis of the longest monotone subsequences in
a random Mallows permutation. There, a Markov chain related to our
$(\infty,q)$-arc chain is considered. While the two chains differ,
they share the same visit times to zero and the work \cite{BB16}
contains an analysis of the distribution of the return times to
zero, related to our discussion here.

The following theorem will be instrumental in the proofs of the
upper bounds of Theorem~\ref{cycle_length}, Theorem~\ref{cycle_var}
and Theorem~\ref{cycle_diam_theorem} in this regime.

\begin{theorem}\label{mallows_chain_return_time}
   Let $\kappa$ be an $(n,q)$-arc chain with $\kappa_0 = 0$ and let $0\leq s \leq n$.
     For any $\eps>0$ there exists a constant $c_\eps>0$ such that
    \[
        T_s:=\min\{t \ge s\mid \kappa_t=0\} \quad \text{satisfies}\quad\E[(T_s-s)^2]\leq  c_\eps\cdot q, \quad \text{for all $q\in(0,1-\eps)$.}
    \]
\end{theorem}
This theorem is a consequence of the following two statements.
\begin{proposition}\label{prop:arc_chain_T_bounds}
    There exists a monotone non-decreasing function $f\colon(0,1)\to[0,\infty)$ such that the following holds. Let $k\geq 0$ and let $\kh$ be an $(\infty,q)$-arc chain with $\kh_0 = k$. Then
    \[
         T:=\min \{t\geq 0\mid \kh_t = 0\} \quad \text{satisfies}\quad \E[T^2]\leq f(q) \cdot  k^2, \quad \text{for all $k\geq 0$}.
    \]
\end{proposition}
Recall the definition of $\xi$ from \eqref{eq:xi_def}.
\begin{lemma}\label{lem:stationary_distribution_bounds}
    Let $\nu$ be the stationary distribution of the $(\infty,q)$-arc chain. One has
    \begin{align}
        &\nu_x\leq 2^{2\xi-x}, &\text{for $q\in(0,1)$,} \label{eq:stationary_dist_low_bounds_1}\\
    &\nu_x\leq \frac{4q}{2^x},  &\text{for $x>0$ and $q\in(0,\tfrac{1}{4})$.}  \label{eq:stationary_dist_low_bounds_2}
    \end{align}
\end{lemma}

\begin{proof}[Proof of Theorem \ref{mallows_chain_return_time}]
    Let $\kh$ be a bounding chain of $\kappa$ with $\kh_0$ having the stationary distribution of the $(\infty,q)$-arc chain.
     Define $ \hat T_s$ by
    \[
         \hat T_s:= \min\{t\ge s \mid \kh_t = 0\}.
    \]
        Proposition~\ref{arc_chain_monotonicity} implies that $T_s \leq \hat T_s$. Hence it suffices to prove that there exists $c_\eps>0$ such that
    \[
        \E[(\hat T_s-s)^2]\leq c_\eps \quad \text{for $q\in(0,1-\eps)$,}\quad \text{and}\quad
        \E[(\hat T_s-s)^2]\lesssim q \quad\text{for $q\in(0,\tfrac{1}{4})$.}
    \]
     As $\kh$ is a time-homogeneous Markov chain, Proposition~\ref{prop:arc_chain_T_bounds}
     implies that
    \[
        \E[(\hat T_s-s)^2 \mid \kh_s] \leq f(q) \cdot \kh_s^2.
    \]
    By taking expectations we obtain
    \begin{equation}\label{tau_eq_3}
        \E[(\hat T_s-s)^2] \leq f(q)\cdot \sum_{k > 0} k^2\cdot \nu_k
    \end{equation}
        Using the geometric bounds on $\nu_s$ provided in \eqref{eq:stationary_dist_low_bounds_1}
        and the fact that $f$ is a monotone non-decreasing function,
        it follows that the right-hand side in \eqref{tau_eq_3} is uniformly bounded for $q\in(0,1-\eps)$.

        Now consider the case $q\leq\frac{1}{4}$. By applying \eqref{eq:stationary_dist_low_bounds_2} to \eqref{tau_eq_3} we
        conclude that
        \[
                \E[(\hat T_s -s)^2]\leq f(\tfrac{1}{4}) \cdot \sum_{k>0}\nu_k\cdot k^2 \leq 4 q \cdot  f(\tfrac{1}{4}) \cdot \sum_{k>0} \frac{k^2}{2^k} \lesssim q. \qedhere
        \]
\end{proof}

Now we need only prove Lemma~\ref{lem:stationary_distribution_bounds} and Proposition~\ref{prop:arc_chain_T_bounds}.
\begin{proof}[Proof of Lemma~\ref{lem:stationary_distribution_bounds}]
  Let $u_i$ and $v_i$ be as in \eqref{bound_chain_transitions}.
    As $u_i$ is monotone decreasing in $i$ and $v_i$ is monotone increasing in $i$ it follows that for $i> 2\xi$ one has
        \[
            \frac{u_i}{v_i} \leq \frac{u_{2\xi + 1}}{v_{\xi}}\leq q^{\xi}\cdot \frac{u_{\xi+1}}{v_{\xi}} \leq  q^{\xi} \leq
            \frac{1}{2},
        \]
        where we have used that $u_{\xi+1} \leq \frac{1}{4}$ and $v_{\xi}\geq \frac{1}{4}$. Thus, using \eqref{stationary_measure_eq}, we obtain
    \begin{equation}\label{pmc_eq_1}
        \nu_x = \nu_{2\xi}\cdot \prod_{i=2\xi+1}^{x}\frac{u_i}{v_i} \leq \nu_{2\xi}\cdot 2^{2\xi-x}\leq 2^{2\xi-x} \quad \text{for $x>2\xi$}.
    \end{equation}
    This completes the proof of \eqref{eq:stationary_dist_low_bounds_1}. We proceed with the proof of \eqref{eq:stationary_dist_low_bounds_2} and assume $q\leq \frac{1}{4}$.
    In this case one may verify that $\frac{u_s}{v_s}\leq \frac{1}{2}$ for $s>1$, while $\frac{u_1}{v_1}\leq 2 q$.
   Similarly as we obtained \eqref{pmc_eq_1} we conclude
    \[
        \nu_x\leq \frac{q}{2^{x-2}}\nu_0\leq  \frac{q}{2^{x-2}}, \quad \text{for $x>0$.} \qedhere
    \]
\end{proof}
We proceed to prove Proposition~\ref{prop:arc_chain_T_bounds}. Let
$\kh$ be an $(\infty,q)$-arc chain. We shall write $\E^k$ to denote
the expectation under the measure where $\kh_0 = k$. We also define
the stopping times,
\begin{equation*}
  \tau_i = \min\{t\ge 0\mid \kh_t \le i\},\quad i\ge 0.
\end{equation*}
\begin{claim}\label{cl:tau_bound} If $k>4\xi$ then $\E^k[3^{\tau_{4\xi}}] \leq
9^{k-4\xi}$ and if $1\le k\le 4\xi$ then
$\E^k[\lambda_{k,q}^{\tau_{k-1}}]
    <\infty$ for some $\lambda_{k,q}>1$.
\end{claim}
\begin{proof}
    Suppose first that $k>4\xi$. Consider the random sequence $X_t:=3^{2\kh_t-8\xi+t}$ and note that $X_0 = 9^{k-4\xi}$. Let $u_i$ and $v_i$ be as in \eqref{bound_chain_transitions} with $v_0:=0$.           The sequence $(X_t)$ satisfies
    \[
        \E^k[X_{t+1}\mid \kh_t = i] = (27\cdot u_{i+1} +3 \cdot (1-u_{i+1}-v_i) + \tfrac{1}{3} v_i) \cdot X_t.
    \]
    One may verify that for all $i > 4\xi$ one has $27\cdot u_{i+1} +3 \cdot (1-u_{i+1}-v_i) + \tfrac{1}{3} v_i \leq \frac{3}{4}$.
    This implies that \[\E^k[X_{t+1}\mid \kh_t=i] \leq \tfrac{3}{4} X_t\quad \text{for all $i > 4\xi$.}\]
    Denoting $t\wedge\tau_{4\xi} = \min\{t,\tau_{4\xi}\}$ and using the facts that $\tau_{4\xi}$ is a stopping time for $(X_t)$ and $\kh_t>4\xi$ when $\tau_{4\xi} > t$, it follows that
    $\E^k[X_{t\wedge\tau_{4\xi}}]$ is non-increasing in $t$. As $3^{t\wedge \tau_{4\xi}}\le X_{t\wedge\tau_{4\xi}}$ we conclude by the monotone convergence theorem
    that
    \begin{equation*}
      \E^k [3^{\tau_{4\xi}}] = \lim_{t\to\infty} \E^k [3^{t\wedge\tau_{4\xi}}] \le
      \lim_{t\to\infty} \E^k[X_{t\wedge\tau_{4\xi}}]\le X_0 = 9^{k-4\xi}.
    \end{equation*}
    Now suppose that $1\le k\le 4\xi$. By induction and using the previous case we may assume that there exists
    some $\lambda_{k+1,q}>1$ for which
    \begin{equation}\label{eq:k_plus_one_lambda}
      \E^{k+1}[\lambda_{k+1,q}^{\tau_k}]<\infty.
    \end{equation}
    Let $t\ge 1$ and $1<\lambda<\lambda_{k+1,q}$. By conditioning on
    the first step of the Markov chain we have
    \begin{align}
      \E^k[\lambda^{\tau_{k-1}\wedge t}] &=
         v_k  \lambda +
            (1-u_{k+1}-v_k)  \E^k [\lambda^{\tau_{k-1}\wedge t} \mid \kh_1 = k] +
            u_{k+1}  \E^k [\lambda^{\tau_{k-1}\wedge t} \mid \kh_1 = k + 1] \nonumber
              \\
        &= v_k\lambda +
            (1-u_{k+1}-v_k)  \E^k[\lambda^{(1+\tau_{k-1})\wedge t}]
            +u_{k+1}  \E^{k+1}[\lambda^{(1+\tau_{k-1})\wedge t}].\label{eq:lambda_tau_bound}
    \end{align}
    Now observe that $\E^k[\lambda^{(1+\tau_{k-1})\wedge  t}]\le \lambda\cdot\E^k[\lambda^{\tau_{k-1}\wedge t}]$. In addition,
    \[
        \E^{k+1}[\lambda^{(1+\tau_{k-1})\wedge t}] \leq
        \lambda \cdot \E^{k+1}[\lambda ^{\tau_k+(\tau_{k-1}-\tau_k)\wedge t}] =
        \lambda\cdot\E^{k+1}[\lambda^{\tau_k}]\cdot\E^k[\lambda^{\tau_{k-1}\wedge t}]
    \]
    where in the last equality we used the strong Markov property
    and the fact that $\tau_k$ is almost surely finite under the measure where $\kh_0 = k+1$ by \eqref{eq:k_plus_one_lambda}.
    Substituting these two bounds into \eqref{eq:lambda_tau_bound} and rearranging the terms we conclude that
    \begin{equation*}
      \big(1 - (1-u_{k+1}-v_k) \cdot\lambda - u_{k+1} \cdot\lambda\cdot\E^{k+1}[\lambda^{\tau_k}] \big)\cdot \E^k[\lambda^{\tau_{k-1}\wedge t}] \leq v_k\cdot\lambda.
    \end{equation*}
    Thus, using \eqref{eq:k_plus_one_lambda}, we may pick $\lambda>1$
    sufficiently small to make the coefficient of $\E^k[\lambda^{\tau_{k-1}\wedge
    t}]$ positive. With this choice, we conclude that $\E^k[\lambda^{\tau_{k-1}\wedge
    t}]$ is bounded uniformly in $t$. Taking the limit $t\to\infty$
    finishes the proof.
\end{proof}
\begin{proof}[Proof of Proposition~\ref{prop:arc_chain_T_bounds}]
The proposition is trivial for $k=0$ so we assume that $k\ge 1$.
Observe that
\begin{equation*}
  T = T_{4\xi} + \sum_{i=0}^{4\xi - 1} (\tau_i - \tau_{i+1}).
\end{equation*}
The Cauchy-Schwartz inequality then implies that
\begin{equation*}
  T^2 \le (4\xi + 1)\Big(T_{4\xi}^2 + \sum_{i=0}^{4\xi - 1} (\tau_i -
  \tau_{i+1})^2 \Big).
\end{equation*}
If $i\ge k$ then $\tau_i =
  \tau_{i+1}$ whereas if $i<k$ then the strong Markov property
  and Claim~\ref{cl:tau_bound} imply that
  \begin{equation*}
    \E^k[(\tau_i -
  \tau_{i+1})^2] = \E^{i+1} [\tau_i^2] < \infty.
  \end{equation*}
  Similarly, if $4\xi\ge k$ then $T_{4\xi} = 0$ while if $4\xi<k$
  it follows from Claim~\ref{cl:tau_bound} and the fact that $\log_3^2 x$ is
  concave for $x\ge e$ that
  \begin{equation*}
    \E^k[\tau_{4\xi}^2] = \E^k[\log_3^2(3^{\tau_{4\xi}})]\le \log_3^2 \E^k[3^{\tau_{4\xi}}]\le g(q) k^2
  \end{equation*}
  for some $g(q)>0$. Combining all of the above facts we conclude
  that
  \begin{equation*}
    \E^k[T^2] \le g(q)(4\xi + 1)k^2 + (4\xi+1)\cdot4\xi\cdot\max_{0\le i\le 4\xi-1}\E^k[(\tau_i -
  \tau_{i+1})^2]\le h(q) k^2
  \end{equation*}
  for some $h(q)>0$. This bound implies a similar bound in which $h$
  is replaced by a monotone non-decreasing function $f$ as
  $\E^k[T^2]$ is a non-decreasing function of $q$ by
  Proposition~\ref{arc_chain_monotonicity}.
\end{proof}

\end{subsection}

\begin{subsection}{Induced Mallows Permutations and a Stitching Process}\label{ggm}
    Our discussion so far was based on the results of Section~\ref{sec:generating_the_graph}, describing the
    distribution of a portion of the graph $\Gamma_\pi$ of a
    Mallows permutation $\pi$ conditioned on the parts of the
    graph `to the left and below this portion'. In our proof of
    Theorem~\ref{thm:pdl}, pertaining to the Poisson-Dirichlet limit
    law, we will need to understand the distribution of portions of
    $\Gamma_\pi$ under more general conditioning events. Our first result in this section
    discusses the distribution of $\Gamma_\pi$ restricted to a rectangle, given the complementary part of $\Gamma_\pi$.
    As it turns out, in this case the relative ordering of the points of
    $\Gamma_\pi$ is itself distributed via a Mallows distribution.
    This is formulated precisely below.

    Given a finite set of points $\Gamma\subseteq\R^2$, no two of
    which have equal $x$ or equal $y$ coordinate, we define the
    \emph{relative order of $\Gamma$} as a permutation $\lambda$
    characterized by the following properties:
    \begin{equation}\label{eq:relative_order_permutation}
    \text{if $\Gamma = \{(x_1,y_1),\ldots,
    (x_k, y_k)\}$ with $x_1<\cdots < x_k$ then $\lambda\in\S_k$
    and $\lambda_i := |\Gamma \cap \{y \le y_i\}|$.}
    \end{equation}
    The name relative order stems from the fact that for each pair $i,j\in [k]$, one has $\lambda_i<\lambda_j$ if and
    only if $y_i < y_j$.
    \begin{lemma} \label{lem:ggm_subgraph}
            Let $R := \left\{x_1\leq x \leq x_2, \, y_1\leq y\leq y_2\right\}\subseteq [n]^2$ be
            a discrete rectangle.
            Let $\pi\sim\mu_{n,q}$ and $\Gamma_\pi$ be the graph of $\pi$.
            When $\Gamma_\pi \cap R$ is non-empty the relative order \eqref{eq:relative_order_permutation} of
    $\Gamma_\pi \cap R$, conditioned on $\Gamma_\pi\smallsetminus R$,
            has the Mallows distribution $\mu_{m,q}$ with $m := |\Gamma_\pi \cap R|$.
    \end{lemma}
    \noindent The lemma generalizes the more familiar special case where $R$ is a vertical rectangle $\{x_1\leq x \leq x_2\}$ (see, for instance, \cite[Corollary 2.7]{BP13} for a
    proof of this special case). The original paper of Mallows \cite{M57} contained a discussion of related
    facts. In fact, the above lemma can be deduced from the special case, though we will provide a direct proof below.

    So far we have discussed methods for sampling the graph of a Mallows
    permutation iteratively `from beginning to end'. The above lemma gives rise to a method for sampling the
    graph `from a mid-point'. Indeed, one can consider,
    say, the top and bottom parts of the graph,
    \begin{equation}\label{eq:ggm_def_AB}
        A:=\Gamma_\pi \cap \{y\leq s\}\quad \text{and}\quad B:=\Gamma_\pi \cap \{y >
        s\}
    \end{equation}
    for a given $0\le s\le n$. Then, due to the lemma, the relative orders of $A$ and $B$ are independent
    and have distributions $\mu_{s,q}$ and $\mu_{n-s,q}$,
    respectively. We now provide a `stitching' procedure for
    determining the full permutation $\pi$ from the relative orders.
    Indeed, given the relative orders, one may determine the full permutation
    from the projections of $A$ and $B$ on the $x$-axis. Defining the process $(\chi_t)$, $0 \leq t \leq
    n$, by
    \begin{equation}\label{eq:chi_def}
       \chi_t := |\Gamma_\pi \cap \{x > t,y \le s\}|
    \end{equation}
    we note that the projection of $A$ on the $x$-axis is exactly the set of
    descents of $\chi$, $\left\{t\in[n]\mid
    \chi_{t}=\chi_{t-1}-1 \right\}$, and the projection of $B$ on the $x$-axis is the complementary set.
    Thus, the following lemma provides a computational procedure for determining these projections.
    \begin{lemma} \label{lem:ggm_subsequence}
        The process $\chi$ defined in \eqref{eq:chi_def} is a time-inhomogeneous Markov chain with transition
        probabilities given by
            \begin{equation}\label{ggm_eq1}
                    \P [\chi_{t+1} =\chi_t-1\mid \pi_1,\ldots,\pi_t]=\frac{1-q^{\chi_t}}{1-q^{n-t}} \quad \text{and}
                    \quad \P [\chi_{t+1} = \chi_t\mid \pi_1,\ldots,\pi_t]=\frac{q^{\chi_t}-q^{n-t}}{1-q^{n-t}}.
            \end{equation}
    \end{lemma}
    We remark that the formulas \eqref{ggm_eq1} provide more than the
    transition probabilities of $\chi$; namely, that these
    probabilities remain the same even when conditioning on $\pi_1,\ldots, \pi_t$.

    Putting together the above two lemmas we obtain the following
    method for sampling a Mallows permutation $\pi$. Let $0\le s\le
    n$ and let $A$ and $B$ be as in \eqref{eq:ggm_def_AB}. Observe
    that, according to Lemma~\ref{lem:ggm_subgraph}, the relative order of $A$ is
    independent of $\chi$, as $\chi$ is determined by
    $B$. Similarly, the relative order of $B$ is independent of both
    $\chi$ and $A$. Thus we may sample $\pi$ by independently sampling
    $\chi$, the relative order of $A$ and the relative order of $B$. The sampling of $\chi$ can be performed using the transition probabilities given in
    Lemma~\ref{lem:ggm_subsequence}. The usual sampling algorithm \eqref{eq:sampling_formula} may then be used to sample the
    relative orders of $A$ and $B$.

    {\bf The Infinite Case:} Gnedin and Olshanski defined an infinite two-sided Mallows
    distribution as the unique
    $q$-exchangeable measure on one-to-one and onto $\pi\colon
    \Z\to\Z$, see \cite{GO12} for details. A method for sampling from this distribution was given in \cite{GO12}.
    Although the infinite two-sided Mallows distribution is not used in our work, we note here that the above sampling
    algorithm may be extended to produce another sampling method for it.

    We do not define the infinite two-sided Mallows distribution and shall rely only on the fact, proved in \cite[Proposition 7.6]{GO12},
    that this distribution is the limit of the distributions of finite
    Mallows permutations in a suitable sense. To give precise
    meaning to this let us extend the definition of the finite
    Mallows permutation to arbitrary finite, non-empty, intervals
    $I\subset\Z$ by saying that $\pi\sim\mu_{I,q}$ if $\pi:I\to I$
    is a bijection and $\P[\pi]$ is proportional to $q^{\inv(\pi)}$
    as in \eqref{eq:mu_n_q_def}. This is the same as saying that $P_I^{-1}\circ\pi\circ
    P_I\sim\mu_{|I|,q}$ where $P_I$ is the unique
    increasing bijection from $[|I|]$ to $I$. We view bijections
    $\pi:\Z\to\Z$ as elements of $\Z^\Z$ with the product topology
    and identify each bijection $\pi:I\to
    I$ with a bijection $\pi:\Z\to\Z$ by setting $\pi(i)=i$ for
    $i\notin\Z$.
    \begin{fact}[\protect{\cite[Proposition
    7.6]{GO12}}]\label{fact:approximation_prop}
      Let $(I_n)$ be an arbitrary sequence of finite, non-empty, intervals increasing
      to $\Z$ and let $\pi^{(n)}\sim\mu_{I_n,q}$. Then $\pi^{(n)}$ converges
      in distribution to the infinite two-sided
      Mallows distribution.
    \end{fact}
    We augment this with the following useful tightness property.
    \begin{claim}\label{cl:tightness_for_chi}
      Let $(I_n)$ be an arbitrary sequence of finite, non-empty, intervals increasing
      to $\Z$ and let $\pi^{(n)}\sim\mu_{I_n,q}$. Then
      \begin{equation*}
        \lim_{t\to \infty}\sup_n \P[\Gamma_{\pi^{(n)}}\cap \{x> t,
        y\le0\}\neq \emptyset] = 0.
      \end{equation*}
    \end{claim}
    Here and later, similarly to before, given $\pi:I\to I$ with $I\subseteq\Z$, we write
    \begin{equation*}
      \Gamma_{\pi} := \{(s,\pi_s)\mid s\in I\}.
    \end{equation*}
    \begin{proof}[Proof of Claim~\ref{cl:tightness_for_chi}]
      The claim follows either from
    Theorem~\ref{mallows_chain_return_time} or from tail bounds on the displacement of
    elements as in, say, \cite[Theorem 1.1]{BP13}. Let us argue from
    Theorem~\ref{mallows_chain_return_time}. Write $I_n =
    \{-a,-a+1,\ldots, b\}$ and assume that $a,b\ge 0$ as otherwise, deterministically, $\Gamma_{\pi^{(n)}}$ does not intersect the quadrant $\{x > 0, y\le 0\}$.
    Let $\tilde{\pi}^{(n)}\sim\mu_{b+a+1, q}$.
    By the definition of $\mu_{I_n,q}$, we see that
    \begin{equation}\label{eq:translated_interval}
      \P[\Gamma_{\pi^{(n)}}\cap \{x> t,y\le0\}\neq \emptyset] = \P[\Gamma_{\tilde{\pi}^{(n)}}\cap \{x > t + a,y\le a\}\neq
      \emptyset].
    \end{equation}
    Let $\kappa$ be the arc chain of $\tilde{\pi}^{(n)}$ and define
    $T:=\min\{s \ge a\mid \kappa_s=0\}$. By
    \eqref{eq:kappa_t_symmetric_def}, assuming also $t\ge 0$,
    \begin{equation*}
      \{\Gamma_{\tilde{\pi}^{(n)}}\cap \{x > t + a,y\le a\}\neq
      \emptyset\}\subseteq \{T>t+a\}.
    \end{equation*}
    Thus, Theorem~\ref{mallows_chain_return_time} and Markov's
    inequality imply that the probabilities in
    \eqref{eq:translated_interval} are at most $c_q / t^2$ for some $c_q>0$ depending only on $q$, from
    which the claim follows.
    \end{proof}

    We require the notion of an infinite \emph{one-sided} Mallows distribution. We recall that one may define an
    infinite one-sided Mallows permutation $\pi:\N\to\N$ with parameter $0<q<1$ via the
    formula \eqref{eq:sampling_formula} with the formal substitution
    $n=\infty$ and this yields a convenient sampling algorithm. This construction may be slightly generalized:
    For a countably infinite subset $I\subset\Z$ with either a minimal or maximal element let $P_I$ be the unique monotone bijection $P_I:\N\to I$
    (increasing if $I$ has a
    minimal element and decreasing if it has a maximal element).
    Given two countably infinite $I,J\subset\Z$, each with a minimal
    or maximal element, one defines the infinite one-sided Mallows distribution from $I$ to $J$ with parameter $q$ as the measure on bijections
    $\pi:I\to J$ satisfying that
    $P_J^{-1}\circ \pi\circ P_I$ has the Mallows distribution from $\N$ to $\N$ with parameter $q$.

    We now describe the sampling algorithm for the infinite two-sided case. Let $\pi\colon \Z\to\Z$ have the infinite two-sided Mallows
    distribution with parameter $q$. Let
    \[
        A:=\Gamma_{\pi} \cap \{y\leq 0\}\quad \text{and}\quad B:=\Gamma_{\pi} \cap \{y > 0\},
    \]
    be the `bottom' and `top' portions of the graph of $\pi$,
    similarly to the finite case. Let $A_{\x}$ and $B_{\x}$ be the projections onto the $x$-axis of $A$ and $B$, respectively. It is known that,
    almost surely, $A_{\x}$ has a maximal element and $B_{\x}$
    has a minimal element, and also that
    \begin{equation}\label{eq:conditional_one_sided_Mallows}
    \begin{split}
      &\text{conditioned on $A_{\x}$ and $B_{\x}$, the restrictions $\pi|_{A_{\x}}$ and $\pi|_{B_{\x}}$ are independent} \\
      &\text{\indent and have the infinite one-sided Mallows distributions with parameter $q$.}
    \end{split}
    \end{equation}
    These facts were also noted and used in one of the sampling algorithms presented in
    \cite{GO12}. With a bit of work, they also follow from
    Fact~\ref{fact:approximation_prop}: the facts on the maximal and
    minimal elements follow using Claim~\ref{cl:tightness_for_chi}, together with the reversal symmetry~\eqref{rev_symmetry}, and the fact
    \eqref{eq:conditional_one_sided_Mallows} follows from
    Lemma~\ref{lem:ggm_subgraph}. Thus, to complete the
    description of our sampling method for $\pi$ it suffices to give
    an algorithm for sampling the projections $A_{\x}$ and $B_{\x}$,
    i.e., a method to `stitch' the one-sided infinite bijections to
    a two-sided infinite bijection.

    Define the process $\chi$ by
    \[
        \chi_t := |\Gamma_{\pi}\cap \{x>t, y\leq 0\}|,\quad t\in \Z
    \]
    and note that $A_{\x}$ is exactly the set of descents of $\chi$, i.e., $A_{\x} = \left\{t\in\Z \mid \chi_{t}=\chi_{t-1}-1
    \right\}$. The distribution of $(\chi_t)_{t\geq 0}$ is described by the following two
    facts, whose proof we postpone:
    \begin{align}
        &\text{$\chi_0$ is distributed as the stationary distribution of the $(\infty,q)$-arc chain, see \eqref{eq:bounding_chain_stat_dist}.} \label{eq:chi_0_distribution} \\
        &\text{$\chi$ is a time-homogeneous Markov chain with transition probabilities given by} \label{eq:chi_markov_chain} \\
        &\text{\indent} \P [\chi_{t+1} =\chi_t-1\mid \text{$\pi_i$ for $i\leq t$}]=1-q^{\chi_t} \quad \text{and}
        \quad \P [\chi_{t+1} = \chi_t\mid \text{$\pi_i$ for $i\leq t$}]=q^{\chi_t}. \nonumber
    \end{align}
    Thus, we may easily sample $A_{\x}\cap \N$. To finish, we need
    only sample $B_{\x} \smallsetminus \N$ conditioned on $A_{\x}\cap
    \N$ and $\chi_0$, as $B_{\x} \smallsetminus \N$ together with $A_{\x}\cap
    \N$ determine both $A_{\x}$ and $B_{\x}$. To this end we rely on
    the following facts, whose proof is again postponed:
    \begin{equation}\label{eq:Ax_Bx_symmetric}
    \begin{split}
        &\text{Given $\chi_0$, $A_{\x}\cap \N$ and $B_{\x}\smallsetminus \N$ are conditionally independent and have the same}\\
        &\text{\indent distribution up to reflection. Precisely, given $\chi_0$, $B_{\x}\smallsetminus \N \stackrel{d}{=} -(A_{\x}\cap
        \N)+1$.}
    \end{split}
    \end{equation}
    In conclusion, one may sample $A_{\x}$ and $B_{\x}$ as follows: First sample $\chi_0$ from the distribution
    \eqref{eq:bounding_chain_stat_dist}. Make two independent
    samples of $(\chi_t)_{t\ge 0}$, with the same given $\chi_0$,
    via the Markov chain transition probabilities in \eqref{eq:chi_markov_chain}. Then take $A_{\x}\cap
    \N$ to be the set of descents of the first copy of $(\chi_t)_{t\ge
    0}$ and reconstruct $B_{\x}$ by taking $1-(B_{\x}\smallsetminus \N)$ to be the set of descents of the second copy of $(\chi_t)_{t\ge
    0}$. The sets $A_{\x}$ and $B_{\x}$ are determined from $A_{\x}\cap
    \N$ and $B_{\x}\smallsetminus \N$. The full permutation $\pi$
    may now be reconstructed using the property
    \eqref{eq:conditional_one_sided_Mallows} and the sampling
    algorithm for infinite one-sided Mallows permutations.

    We now return to prove \eqref{eq:chi_0_distribution}, \eqref{eq:chi_markov_chain} and
    \eqref{eq:Ax_Bx_symmetric}. Define the discrete intervals
    \begin{equation*}
    I_n = \{-n+1, -n,\ldots, n\}
    \end{equation*}
    and let $\pi^{(n)}\sim\mu_{I_n,q}$. Define the processes $\chi^{(n)}$ by
    \[
        \chi^{(n)}_t := |\Gamma_{\pi^{(n)}}\cap \{x>t, y\leq 0\}|,\quad t\in
        \Z.
    \]
    Then Fact~\ref{fact:approximation_prop} together with
    Claim~\ref{cl:tightness_for_chi} imply that
    \begin{equation} \label{eq:chi_convergence}
     \text{$\chi^{(n)}$ converges in distribution to $\chi$.}
\end{equation}
    Let us elaborate on the proof of this fact. Observe that
    \begin{equation*}
      \chi^{(n)}_t - \chi^{(n)}_s = |\{i\mid t < i \le s,
      \pi^{(n)}_i\leq 0\}|\quad\text{and}\quad \chi_t - \chi_s = |\{i\mid t < i \le s,
      \pi_i\leq 0\}|,\quad t\le s,
    \end{equation*}
    so that these differences depend only on the value of the permutations
    at finitely many indices. Thus, Fact~\ref{fact:approximation_prop} implies that ${(\chi^{(n)}_t - \chi^{(n)}_s)_{t\leq
    s}}$ converges in distribution to ${(\chi_t - \chi_s)_{t\leq s}}$.
    This may be upgraded to \eqref{eq:chi_convergence} by using the fact that ${\P[\chi_s\neq 0] \stackrel{s\to\infty}{\to}
    0}$ (since $A_\x$ has a maximal element, almost surely) and ${\sup_{n} \P[\chi^{(n)}_s\neq 0] \stackrel{s\to\infty}{\to}
    0}$ by Claim~\ref{cl:tightness_for_chi}.

    Observe that $\chi^{(n)}_0$ has the distribution of the arc chain of $\pi^{(n)}$ at $0$.
    By using \eqref{eq:chi_convergence}, property \eqref{eq:chi_0_distribution} follows from Proposition~\ref{prop:arc_chain_limit}
    and property \eqref{eq:chi_markov_chain} follows from
    Lemma~\ref{lem:ggm_subsequence}. Finally, property \eqref{eq:Ax_Bx_symmetric} is a
    consequence of the reversal symmetry \eqref{rev_symmetry} and Lemma~\ref{lem:ggm_subsequence} applied to $\pi^{(n)}$.

    \medskip
    {\bf Proofs of Lemma~\ref{lem:ggm_subgraph} and Lemma~\ref{lem:ggm_subsequence}:} To complete this section we need only prove these two lemmas.
    \begin{proof}[Proof of Lemma~\ref{lem:ggm_subgraph}]
    Throughout the proof we condition on $\Gamma_\pi \smallsetminus R$ and assume that $|\Gamma_\pi \cap R|>0$. Let $\rho$ be the relative order of $\Gamma_\pi \cap R$.
    Observe that the permutation $\rho$ uniquely determines
    $\pi$ and $\rho$ may assume, with positive probability, any
    value in $\S_{|\Gamma_\pi \cap R|}$.
    Hence, the distribution of $\rho$ is proportional to $q^{\inv(\pi)}$ by
    the definition of the Mallows distribution
    \eqref{eq:mu_n_q_def};
    while we need to prove that the distribution of $\rho$ is proportional to~$q^{\inv(\rho)}$.
    Therefore, it suffices to verify that
    \begin{equation}\label{eq:ggm_subgraph_main}
        \inv(\pi)-\inv(\rho)\quad\text{is determined by $\Gamma_\pi\smallsetminus R$.}
    \end{equation}
       We say that two points $(v_{\x},v_{\y}),(w_{\x},w_{\y})\in \R^2$ form an inversion if $(v_{\x}-w_{\x})\cdot (v_{\y} -w_{\y}) < 0$.
        For two finite subsets $V,W\subset \R^2$ we define
        \[
            \inv (V,W) := |\{(v,w)\in V\times W \mid \text{$(v,w)$ forms an inversion}\}| \quad \text{and} \quad
            \inv (V) := \tfrac{1}{2}\inv(V,V).
        \]
        The definitions are chosen so that $\inv(\sigma) = \inv(\Gamma_\sigma)$ for any permutation $\sigma$. Consider the following
        equality,
        \[
            \inv (\pi) = \inv(\Gamma_\pi\smallsetminus R) + \inv(\Gamma_\pi\smallsetminus R,\Gamma_\pi\cap R) + \inv(\Gamma_\pi\cap R).
        \]
        Observe that $\inv(\Gamma_\pi\smallsetminus R)$ is determined by $\Gamma_\pi \smallsetminus R$ and that
        $\inv(\rho) = \inv(\Gamma_\pi\cap R)$. Thus we need only prove
        that
        \begin{equation}\label{eq:break_into_u}
          \inv(\Gamma_\pi\smallsetminus R,\Gamma_\pi\cap R) = \sum_{u\in \Gamma_\pi \smallsetminus
          R} \inv(\{u\},\Gamma_\pi\cap R)\quad\text{is determined by $\Gamma_\pi\smallsetminus
          R$}.
        \end{equation}

        Let $u=(u_{\x},u_{\y})\in \Gamma_\pi \smallsetminus R$. Since $u\in R^\complement$, we know that at least one of four inequalities occur: $u_{\x}<x_1, u_{\x} > x_2, u_{\y}<y_1, u_{\y}>y_2$.
        Assume $u_{\x} > x_2$. We have
        \begin{align*}
            \inv(\{u\},\Gamma_\pi\cap R) &= |\{ (x,y)\in \Gamma_\pi\cap R \mid y>u_{\y}\}| \\
                    &= n - u_{\y} -  |\{ (x,y)\in \Gamma_\pi\smallsetminus R \mid
            y>u_{\y}\}|.
        \end{align*}
    Thus $\inv\left(\{u\},\Gamma_\pi\cap
    R\right)$ is determined by $\Gamma_\pi\smallsetminus R$ and $u$. Applying similar reasoning in the other three cases shows
    that \eqref{eq:break_into_u} and hence \eqref{eq:ggm_subgraph_main} holds.
    \end{proof}

    \begin{proof}[Proof of Lemma~\ref{lem:ggm_subsequence}]
       As $0\leq \chi_{t}-\chi_{t+1} \leq 1$,
        it suffices to establish the formula for $\P [\chi_{t+1} =\chi_t-1\mid \pi_1,\ldots,\pi_t]$.
    It is convenient to use the formula \eqref{eq:sampling_formula}. Let $j_1<\ldots<j_{n-t}$ be the elements of $[n]\smallsetminus \{\pi_1,\ldots,\pi_{t}\}$.
    Observe that $j_k \leq s$ if and only if $k \le \chi_t$. Hence by \eqref{eq:sampling_formula} and the definition of $\chi_t$ we have
        \[
            \P[\chi_{t+1}=\chi_t-1\mid \pi_1,\ldots,\pi_t] = \sum_{k=1}^{\chi_t}\P[\pi_{t+1} = j_k \mid \pi_1,\ldots,\pi_t] = \frac{1 - q^{\chi_t }}{1-q^{n-t}}. \qedhere
        \]
    \end{proof}
\end{subsection}
\end{section}

\begin{section}{Main Theorems}\label{section_mt}
We start by introducing several definitions which we will need for
proving our main theorems. A non-empty subset $\Fa$ of $[n]$ is
called an \emph{arc} of the permutation $\pi\in \S_n$ if its
elements can be ordered so that $\Fa = \{a_1, \ldots, a_{|\Fa|}\}$
with $\pi(a_{i}) = a_{i+1}$ for $1\le i<|\Fa|$. We say the arc has
\emph{length} $|\Fa|$ and say the arc is \emph{closed} if it forms a
cycle, that is, if also $\pi(a_{|\Fa|}) = a_1$. A non-closed arc is
called \emph{open}. When the arc is open the above ordering is
unique, in which case we call $a_1$ and $a_{|\Fa|}$ the \emph{tail}
and \emph{head} of the arc $\Fa$, respectively, and denote them by
$\tail(\Fa)$ and $\head(\Fa)$.

In our proofs of the main theorems we will rely upon the diagonal
exposure process introduced in Section~\ref{sec:diagonal_exposure}.
We recall that by time $t$ of this process, we expose the
portion of the graph $\Gamma_\pi$ contained in $\{x\le t,y \le t\}$.
This information allows us to determine all arcs which are contained
in $[t]$ and, moreover, to tell whether each such arc is open or
closed. This motivates the following definitions.

Let $\pi\in\S_n$ and $0 \le t \le n$. We say that an arc $\Fa$ of $\pi$
is $[t]$-maximal (with respect to inclusion) if $\Fa\subseteq[t]$
and if every arc $\Fb\subseteq[t]$ which contains $\Fa$ is in fact
equal to $\Fa$. Denote
\begin{align*}
    &\CA_t (\pi):=\{\Fa\mid \text{$\Fa$ is $[t]$-maximal}\},\\
    &\CO_t(\pi):=\{\Fa\mid \text{$\Fa$ is $[t]$-maximal and open}\}.
\end{align*}
Recalling the definition of the arc chain $\kappa$ from
\eqref{arc_chain_def} we observe that
\begin{equation}\label{eq:arc_chain_rel_arcs}
  \kappa_t(\pi)=|\CO_t(\pi)|,
\end{equation}
which is the origin of the name `arc chain'. We note further that
for each $1\le s\le t$ there exists a unique $[t]$-maximal arc containing
$s$ and we denote this arc by $\arc_s^t(\pi)$. When the permutation
$\pi$ is clear from the context we shall abbreviate $\arc_s^t(\pi),
\CA_t (\pi)$ and $\CO_t(\pi)$ to $\arc_s^t, \CA_t$ and $\CO_t$.

\begin{figure}
\begin{subfigure}[tl]{0.42\textwidth}
\bigskip\smallskip
\begin{tikzpicture}
\begin{axis}[
title = {Graph of (18726)(3)(45)},
xmin = 0.5,
xmax = 8.5,
ymin = 0.5,
ymax = 8.5,
];
\addplot+[  only marks,
        color = blue,
            mark size=2.9pt] coordinates {(1,8) (2,6) (3,3) (4,5) (5,4) (6,1) (7,2) (8,7)};
\addplot[color = red, line width = 0.5pt, style = dashed] coordinates {(1.5, 0) (1.5, 1.5) (0, 1.5)};
\addplot[color = red, line width = 0.5pt, style = dashed] coordinates {(2.5, 0) (2.5, 2.5) (0, 2.5)};
\addplot[color = red, line width = 0.5pt, style = dashed] coordinates {(3.5, 0) (3.5, 3.5) (0, 3.5)};
\addplot[color = red, line width = 0.5pt, style = dashed] coordinates {(4.5, 0) (4.5, 4.5) (0, 4.5)};
\addplot[color = red, line width = 0.5pt, style = dashed] coordinates {(5.5, 0) (5.5, 5.5) (0, 5.5)};
\addplot[color = red, line width = 0.5pt, style = dashed] coordinates {(6.5, 0) (6.5, 6.5) (0, 6.5)};
\addplot[color = red, line width = 0.5pt, style = dashed] coordinates {(7.5, 0) (7.5, 7.5) (0, 7.5)};
\end{axis}
\end{tikzpicture}
\end{subfigure}
\begin{subtable}[tr]{0.5\textwidth}
  \begin{tabular}[tr]{ c | c | c | c }
    \multicolumn{1}{c}{$t$} & \multicolumn{1}{c}{$\kappa_t$} & \multicolumn{1}{c}{$\CO_t$} & \multicolumn{1}{c}{$\CA_t \smallsetminus \CO_t$} \\
    \noalign{\medskip}\hline
     1 & 1 & $\{1\}$ & $\emptyset$ \\
     2 & 2 & $\{1\}, \{2\}$ & $\emptyset$ \\
     3 & 2 & $\{1\}, \{2\}$ & $\{3\}$ \\
     4 & 3 & $\{1\}, \{2\}, \{4\}$ & $\{3\}$ \\
     5 & 2 & $\{1\}, \{2\}$ & $\{3\}, \{4,5\}$ \\
    6 & 1 & $\{1,2,6\}$ & $\{3\}, \{4,5\}$ \\
    7 & 1 & $\{1,2,6,7\}$ & $\{3\}, \{4,5\}$ \\
    8 & 0 & $\emptyset$ & $\{3\}, \{4,5\}, \{1,2,6,7,8\}$ \\
  \end{tabular}
\end{subtable}
\caption{Graph of the permutation $(18726)(3)(45)$ alongside its arc
chain process $(\kappa_t)$, its maximal open arcs process $(\CO_t)$
and its maximal closed arcs process $(\CA_t\smallsetminus \CO_t)$.}
\label{fig:graph_with_table}
\end{figure}
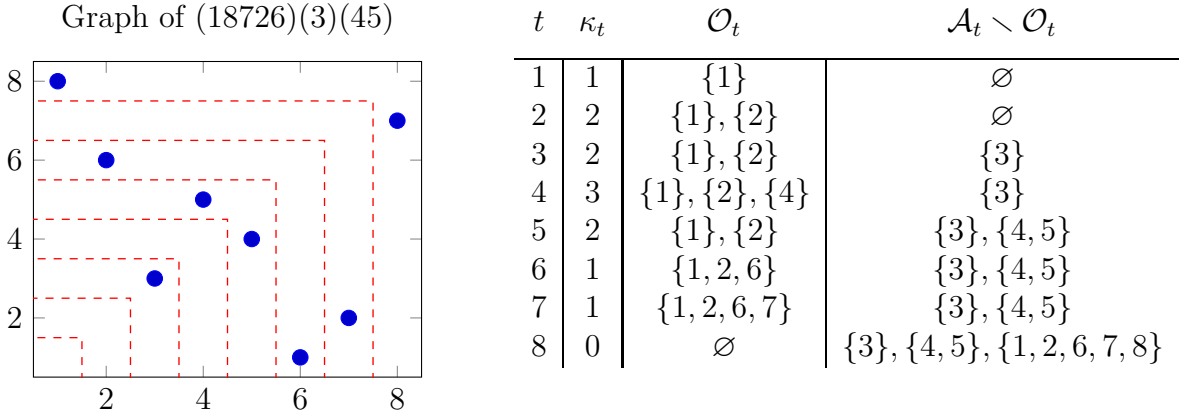

Let us describe how $\CA_t$ evolves during the diagonal exposure
process, i.e., the relationship between $\CA_t$ and $\CA_{t+1}$; see
Figure~\ref{fig:graph_with_table} for an example. The newly exposed
portions of the graph $\Gamma_\pi$ at time $t+1$ were described in
\eqref{eq:t_iteration_new_portions}. Thus, the set of
$[t+1]$-maximal arcs is formed from the set of $[t]$-maximal arcs by
having the element $t+1$ either: (i) form an, open or closed, arc by
itself; (ii) extend an open arc to a new, open or closed, longer
arc; or (iii) merge two open arcs into a longer open arc. These
three possibilities are considered below according to their effect
on the number of open arcs.
\begin{itemize}[topsep=2pt,itemsep=1pt]
    \item
        If $\kappa_{t+1}=\kappa_t+1$ then necessarily $\{t+1\}$ is a $[t+1]$-maximal open arc and $\CA_{t+1}$ equals $\CA_t$ with $\{t+1\}$
        added.
    \item
        If $\kappa_{t+1}=\kappa_t$ then either $t+1$ is a fixed point of $\pi$
        or $t+1$ extends an open arc in $\CA_t$ to a longer open
        arc in $\CA_{t+1}$, either as the head or as the tail of the arc.
    \item
        If $\kappa_{t+1}=\kappa_t-1$ then either two open arcs were merged via $t+1$ or an open arc was extended by $t+1$ to a closed
        arc.
\end{itemize}

We now consider the probabilities for the process $(\CA_t)$ to
evolve according to the above possibilities when $\pi$ is a Mallows
random permutation. As mentioned above, we note that $\CA_t$ and
$\CO_t$ are measurable with respect to $\CF_t$, i.e., $\CA$ and
$\CO$ are adapted to the diagonal exposure filtration.

The probability, conditioned on $\CF_t$, that the element $t+1$
forms a $[t+1]$-maximal arc by itself was already calculated in
Proposition~\ref{mallows_chain_transition_prob} (for the case that
it forms an open arc, or equivalently that $\kappa_{t+1} = \kappa_t
+ 1$) and Lemma~\ref{fix_point_equation} (for the case that it is a
fixed point).

The other options, in which the element $t+1$ either extends an
existing arc or merges two arcs, are determined from the basic
events $\{\pi_{t+1}^{-1}=\head(\Fa)\}$ and $\{\pi_{t+1}=\tail(\Fa)\}$ for an
arc $\Fa\in\CO_t$. Indeed, for the merging event one checks in a
straightforward manner that if $\Fa,\Fb\in\CO_t$ are distinct arcs
then
\[
    \{\Fa\cup \{t+1\} \cup \Fb \in\CO_{t+1} \text{ and } \pi_{t+1}\in \Fb\} \quad \iff \quad \{\pi_{t+1}^{-1}=\head(\Fa) \text{ and  } \pi_{t+1}=\tail(\Fb)\}
\]
and for the event that $t+1$ extends the open arc $\Fa \in \CO_t$ to a
closed arc we have
\[
    \{\Fa\cup \{t+1\} \in\CA_{t+1}\smallsetminus\CO_{t+1}\} \quad \iff \quad \{\pi_{t+1}^{-1}=\head(\Fa) \text{ and  } \pi_{t+1}=\tail(\Fa)\}.
\]
The event that $t+1$ extends an open arc to a longer
open arc is the complement of the other possibilities. Thus the
probabilities of these events may be derived from the following
lemma.
\begin{lemma}\label{lemma_U_arc_eq}
    Let $\pi\sim\mu_{n,q}$, let $0\le t < n$ and let $\Fa,\Fb \in \CO_t$  be two, not necessarily distinct, arcs.
    The events $\{\pi_{t+1}^{-1}=\head(\Fa)\}$ and $\{\pi_{t+1}=\tail(\Fb)\}$ are conditionally independent given $\CF_t$.
    Furthermore,
    \begin{equation}\label{arc_tail_head_eq}
        \P_t[\pi_{t+1}^{-1}=\head(\Fa)]= q^{\indh_\Fa} \cdot \frac{1-q}{1-q^{n-t}} \quad \text{and} \quad
        \P_t[\pi_{t+1}=\tail(\Fb)]=q^{\indt_\Fb} \cdot \frac{1-q}{1-q^{n-t}},
    \end{equation}
    where
    \[
        \indh_\Fa=|\{\Fc\in\CO_t\mid \head(\Fc)< \head(\Fa)\}| \quad \text{and} \quad \indt_\Fb=|\{\Fc\in\CO_t\mid \tail(\Fc)< \tail(\Fb)\}| .
    \]
\end{lemma}
A consequence of the lemma is that for $q\geq \frac{1}{2}$ and any
two open arcs $\Fa,\Fb \in \CO_t$ one has
  \[
        \P_t[\pi_{t+1}^{-1}=\head(\Fa) \text{ and } \pi_{t+1}=\tail(\Fb)]\approx \left(\frac{1-q}{1-q^{n-t}}\right)^2 \quad \text{on $\{\kappa_t\leq
        2\xi\}$},
  \]
which eventually leads to the appearance of the length scale
$\tfrac{1}{(1-q)^2}$ in our theorems.
\begin{proof}[Proof of Lemma~\ref{lemma_U_arc_eq}]
The fact that the events $\{\pi_{t+1}^{-1}=\head(\Fa)\}$ and
$\{\pi_{t+1}=\tail(\Fb)\}$ are conditionally independent given $\CF_t$
    is a direct consequence of \eqref{eq:conditional_independence_in_graph}.

It is not difficult to verify that
\begin{equation*}
  \indt_\Fb = |\{1\le j<\tail(\Fb)\mid j\notin \{\pi_1,\ldots,\pi_{t}\}\}|.
\end{equation*}
Thus, the sampling formula \eqref{eq:sampling_formula} yields that
\begin{align*}
  \P_t[\pi_{t+1}=\tail(\Fb)] =\E_t \P[  \pi_{t+1}=\tail(\Fb) \mid \pi_1,\ldots, \pi_{t}]
            =\E_t \left[q^{\indt_{\Fb}}\cdot\frac{1-q}{1-q^{n-t}} \right]
  &= q^{\indt_{\Fb}}\cdot \frac{1-q}{1-q^{n-t}}
\end{align*}
as we wanted to show. The formula for
$\P_t[\pi_{t+1}^{-1}=\head(\Fa)]$ follows from the formula for
$\P_t[\pi_{t+1}=\tail(\Fb)]$ applied to the inverse permutation
$\pi^{-1}$, as we have the inverse symmetry \eqref{inv_symmetry} and
the fact that $\CF_t$ and $\CO_t$ are invariant under this symmetry.
\end{proof}

\begin{subsection}{Expected Number of Cycles}\label{sec:num_of_cycles}
In this section we prove Theorem~\ref{thm:number_of_cycles}. Let
$\pi\sim\mu_{n,q}$ and $\kappa$ be the arc chain of $\pi$. Recall
from the Introduction that $\CC_s$ is the set of points in the cycle
of $\pi$ which contains $s$.

Our proof is based on the fact that the number of cycles in $\pi$
equals the number of points $s\in [n]$ that satisfy $s=\max
(\CC_s)$, i.e.,
\begin{equation}\label{eq:num_of_cycles_sum}
    \text{number of cycles in $\pi$} \ = |\{s\in [n]\mid s = \max(\CC_s)\}| = \sum_{s\in [n]} \1_{\{s=\max(\CC_s)\}}.
\end{equation}
The following proposition provides an estimate for the conditional
probability of the event $\{s=\max (\CC_s)\}$ given $\CF_{s-1}$.
\begin{proposition} \label{upper_bound_arc_s}
    For $1\le s \le n$ one has
    \[
        q^{\kappa_{s-1}}\cdot\frac{1-q}{1-q^{n-s+1}} \leq \P_{s-1}[s=\max (\CC_s)]\leq\frac{1-q}{1-q^{n-s+1}} .
    \]
\end{proposition}

\begin{proof}
    Given $\pi_1,\ldots,\pi_{s-1}$ there is a
    unique element $t\in[n]$, distinct from $\pi_1,\ldots,\pi_{s-1}$, such that $s = \max(\CC_s)$ if and only if $\pi_s = t$
    (indeed, there is a unique $k\ge 1$ and unique distinct $t_1, \ldots, t_k \in [s]$ satisfying
    $t_k = s$, $\pi_{t_i} = t_{i+1}$ for $1\le i\le k-1$ and
    $t_1\notin\{\pi_1, \ldots, \pi_{s-1}\}$ whence we set $t = t_1$).
    We shall derive the proposition from this fact and formula \eqref{eq:sampling_formula}.
    Write $j_1<\cdots<j_{n-s+1}$ for the elements of $[n]\smallsetminus
    \{\pi_1\ldots,\pi_{s-1}\}$ and let $k$ be such that $t = j_k$.
    Then
    \[
         \P_{s-1}[s=\max (\CC_s)]   =\E_{s-1} \P[\pi_s=t\mid \pi_1,\ldots,\pi_{s-1}]
                        = \E_{s-1}\left[ q^{k-1} \frac{1-q}{1-q^{n-s+1}} \right].
    \]
    The proposition follows as $0\leq k-1 \leq \kappa_{s-1}$.
\end{proof}
We augment this proposition with the following simple estimate on
$q^{\kappa_{s-1}}$.
\begin{claim} \label{cl:expected_q_to_kappa}
    For $1\le s \le n$ one has
    \[
        \E[q^{\kappa_{s-1}}]\approx 1.
    \]
\end{claim}
\begin{proof}
Clearly, $q^{\kappa_{s-1}}\le 1$ so we only need to prove the lower
bound. We consider two cases. If $q\geq \frac{1}{8}$ then, by
Theorem~\ref{mallows_chain_bounds},
\begin{equation*}
  \E[q^{\kappa_{s-1}}]\ge q^{2\xi}\P[\kappa_{s-1}\le 2\xi] \ge
  \frac{1}{4}\left(1 - \tfrac{q^{\xi^2 + \xi}}{1-q^{2\xi}}\right)
  \gtrsim 1.
\end{equation*}
Whereas if $q\le \frac{1}{8}$ then, by \eqref{mallows_corr_bounds}
and Lemma~\ref{lem:stationary_distribution_bounds},
\begin{equation*}
  \E[q^{\kappa_{s-1}}] \ge \P[\kappa_{s-1} = 0] \gtrsim 1.\qedhere
\end{equation*}
\end{proof}
Putting together \eqref{eq:num_of_cycles_sum} and the previous two
estimates shows that
\begin{equation*}
  \E[\text{number of cycles in $\pi$}]\approx \sum_{s=1}^n \frac{1-q}{1-q^{n-s+1}}.
\end{equation*}
Theorem~\ref{thm:number_of_cycles} is a direct consequence of this
fact, together with the observation that
\begin{equation}\label{eq:q_k_approx}
    \frac{1-q}{1- q^k}\approx 1-q +\frac{1}{k},\quad k\ge 1.
\end{equation}
To verify \eqref{eq:q_k_approx} we consider two cases. Recall $\xi$
from \eqref{eq:xi_def}.

    If $k\geq \xi$ then $q^{k}\leq \frac{1}{2}$ and $\frac{1}{k}\lesssim 1- q$; therefore $\frac{1-q}{1-q^k}\approx 1 - q \approx 1- q + \frac{1}{k}$.

    If $k< \xi$ then  $q^{k} > \frac{1}{2}$ and $\frac{1}{k}\gtrsim 1- q$; thus  $\frac{1-q}{1-q^k} = \frac{1}{1+q+\cdots+q^{k-1}}\approx \frac{1}{k} \approx 1- q + \frac{1}{k}$.
\end{subsection}

\begin{subsection}{Expected Cycle Diameter}\label{sec:cycle_diam}
In this section we prove Theorem~\ref{cycle_diam_theorem}. Some of
the tools developed here will also be used in proving the rest of
our main theorems.

We prove the lower and upper bounds on the quantity $\E[\max
(\CC_s)-s]$ as given in \eqref{max_equation}. The other bounds in
the theorem follow: The bounds on $\E[s - \min (\CC_s)]$ given in
\eqref{min_equation} are equivalent to those of \eqref{max_equation}
via the reversal symmetry \eqref{rev_symmetry}. Put together, the
bounds in \eqref{max_equation} and \eqref{min_equation} yield the
bounds on $\E[\max (\CC_s)-\min(\CC_s)]$ given in
\eqref{diam_equation}.

Throughout this section we let $\pi\sim\mu_{n,q}$ and $\kappa$ be the arc chain of $\pi$.

\smallskip \noindent {\bf Lower bound:} Let us begin with the proof
of the lower bound of \eqref{max_equation}, i.e.,
\begin{equation}\label{max_equation_low}
     \E[\max(\CC_s)-s]\gtrsim \min\{q\cdot\xi^2, \ n-s\},
\end{equation}
where we recall the definition of $\xi$ from \eqref{eq:xi_def}.
First we are going to provide upper bounds for closing the cycle of
$s$ at time $s\leq t <n$. Observe that for $s\leq t <n$, conditioned
on the event $\{\text{$\arc_s^t$ is open}\}$, one has,
\[
    \arc^{t+1}_s \text{ is closed} \quad \iff \quad \pi_{t+1}^{-1}=\head(\arc^t_s)\ \text{ and  }\ \pi_{t+1}=\tail(\arc^t_s).
\]
Hence the following equation is an immediate
corollary of Lemma \ref{lemma_U_arc_eq}:
\begin{equation}\label{upper_bound_arc_t}
        \text{for $s\leq t <n$,}\quad \P_t  \left[ \arc^{t+1}_s \text{ is closed} \right] \leq \left( \frac{1-q}{1-q^{n-t}} \right)^2  \quad \text{on $\{\arc^t_s \in \CO_t\}$.}
\end{equation}
It is worth noting that the probability of closing the cycle of $s$
at time $t=s$ has a larger estimate, as given in
Proposition~\ref{upper_bound_arc_s}, as the event $s=\max (\CC_s)$
occurs either when $s$ is a fixed point or when any one of the open
arcs closes at time $t=s$.

Combining Proposition \ref{upper_bound_arc_s} with \eqref{upper_bound_arc_t} we derive by induction the following corollary.
\begin{corollary} \label{diam_lower_corr}
    For $0<s\le t\leq n$ one has\footnote{In the case when $s=t$ the empty product is assumed to be $1$.}
    \[
        \P_{s-1}[\arc_s^t \text{ is open}]\geq \frac{q-q^{n-s+1}}{1-q^{n-s+1}}
        \prod_{s \le i < t} \left( 1-\left(\frac{1-q}{1-q^{n-i}} \right)^2 \right).
    \]
\end{corollary}

Now we state the main proposition.
\begin{proposition}\label{cycle_diam_prop_low}
   Let $1\leq s < n$ and set $r:=\min\{s+\xi^2, n-1\}$. Then for $q\geq \frac{1}{2}$ we have
    \[
        \P_{s-1}[\arc_s^r\text{ is open}]\gtrsim 1.
    \]
\end{proposition}

Observe that equation \eqref{max_equation_low} in the case $q\geq\frac{1}{2}$ is an immediate corollary of this proposition.
One may also verify that Proposition~\ref{upper_bound_arc_s} yields \eqref{max_equation_low} in the case $q\leq\frac{1}{2}$.
Hence it suffices to prove Proposition~\ref{cycle_diam_prop_low}.

We use the following calculus fact: for a sequence $(x_i)$ with
$x_i\in[0,1)$ one has
\begin{equation} \label{prod_bound_sum}
    \text{if $\sum_i x_i\lesssim 1$ and $1-x_i\gtrsim 1$ for all $i$, then $\prod_i (1-x_i)\gtrsim
    1$.}
\end{equation}
\begin{proof}[Proof of Proposition \ref{cycle_diam_prop_low}]
    Due to Corollary \ref{diam_lower_corr} and the fact that $\frac{q-q^{n-s+1}}{1-q^{n-s+1}}\geq \frac{q-q^2}{1-q^2}\geq \frac{1}{3}$, it suffices to
    verify that
    \[
        \prod_{s\le i< r} \left( 1-\left(\frac{1-q}{1-q^{n-i}} \right)^2 \right) \gtrsim 1.
    \]
    Due to \eqref{prod_bound_sum} we need only verify that one has
    \[
       1-\left(\frac{1-q}{1-q^{n-i}} \right)^2\gtrsim 1, \quad \text{for $s\le i<r$, and}\quad\sum_{s \le i<r} \left(\frac{1-q}{1-q^{n-i}} \right)^2 \lesssim 1.
    \]
    The first inequality follows from the fact that $q\geq \frac{1}{2}$ and $i<n-1$. For the second inequality, observe that
    \[
         \sum_{s \le i<r} \left(\frac{1-q}{1-q^{n-i}} \right)^2 \leq \sum_{t=1}^{r-s} \left(\frac{1-q}{1-q^t} \right)^2\leq \sum_{t=1}^{\xi^2} \left(\frac{1-q}{1-q^t} \right)^2
    \lesssim \sum_{t=1}^{\xi^2} \left((1-q)^2 + \tfrac{1}{t^2}\right) \lesssim 1,
    \]
    where we applied the estimate $\left(\frac{1-q}{1-q^t} \right)^2\lesssim (1-q)^2 + \frac{1}{t^2}$ which follows from \eqref{eq:q_k_approx}.
\end{proof}

\noindent {\bf Upper bound:} We proceed to establish the upper bound
\begin{equation}\label{max_equation_upp}
    \E[\max(\CC_s)-s]\lesssim \min\{q\cdot\xi^2, \ n-s\}.
\end{equation}
We are going to state analogues to \eqref{upper_bound_arc_t} and
Corollary~\ref{diam_lower_corr} for the upper bound.
Lemma~\ref{lemma_U_arc_eq} implies that
\begin{equation} \label{arc_lower_bound}
    \text{for $s\le t < n$,} \quad
        \P_{t} \left[\arc_{s}^{t+1}\text{ is closed}\right]
        \geq \left( \frac{q^{\kappa_t}-q^{\kappa_t+1}}{1-q^{n-t}} \right)^2.
\end{equation}
\begin{proposition}\label{diam_upper_corr}
    For $1\leq s\leq t \leq n$ and $d\geq 0$ we have
    \begin{equation} \label{diam_upper_corr_eq}
        \P_{s}\left[\arc_s^t \text{ is open and $\kappa_i\leq d$ for all $i\in[s,t]$}\right]\leq
        \exp\left(-(t-s)\cdot (1-q)^2\cdot q^{2d}\right).
    \end{equation}
\end{proposition}
\begin{proof}For $s< r \leq n$ we set
    \[
        E_r := \{ \arc_s^r \text{ is open and $\kappa_i\leq d$ for all $i\in[s,r]$} \}.
    \]
    Via \eqref{arc_lower_bound} one may verify that for such $r$,
    \begin{align*}
         \P_s[E_r] &\le \P_s[E_{r-1}\cap \{\text{$\arc_s^r$ is
         open}\}] \le \P_s[E_{r-1}]\cdot\left(1- \left( \frac{q^{d}-q^{d+1}}{1-q^{n-r+1}} \right)^2\right)\\
    &\leq \P_s[E_{r-1}]\cdot\exp\left(-(1-q)^2\cdot q^{2d}\right).
  \end{align*}
 By applying this inequality for all $s< r \le t$ we obtain inequality \eqref{diam_upper_corr_eq}.
\end{proof}

\begin{proposition}\label{diam_upper_prop_main}
    For $1\le s \le n$ one has
    \[
        \E[(\max (\CC_s)-s)^2]\lesssim \min\{q\cdot \xi^4,(n-s)^2\}.
    \]
\end{proposition}
This proposition implies \eqref{max_equation_upp} for $q \geq
\frac{1}{2}$ using the fact that $\E[X^2]\geq \E[X]^2$ for any
integrable random variable $X$. For $q \leq \frac{1}{2}$ we use the
fact that $\max (\CC_s)-s\leq (\max (\CC_s)-s)^2$ to again deduce
\eqref{max_equation_upp} from the same proposition.

Proposition~\ref{diam_upper_prop_main} is stronger than what we need
here as it bounds the second moment of $\max (\CC_s)-s$, but this
extra strength will be used in the proof of Theorem~\ref{cycle_var}.
\begin{proof}
    The inequality $\max (\CC_s)-s\leq n-s$ holds true by definition.
    So we need only verify $\E[(\max (\CC_s)-s)^2]\lesssim q\cdot \xi^4$.
    In the regime when $q$ is bounded away from~$1$, the inequality is a direct consequence of Theorem~\ref{mallows_chain_return_time} due to:
    \[
         \kappa_t=0 \quad \implies \quad \text{$\max(\CC_i)\le t$ for all $i\leq
         t$},
    \]
    as $\kappa_t$ counts the number of open arcs at time $t$, see
    \eqref{eq:arc_chain_rel_arcs}.
    Thus to complete the proof of proposition it suffices to verify that
     \begin{equation}\label{eq:diam_up_bound}
        \E[(\max(\CC_s)-s)^2]\lesssim \xi^4, \quad \text{for $q$ sufficiently close to $1$}.
     \end{equation}
    We assume without loss of generality that $n> s + 2\xi^2$. Our starting point is the inequality
   \begin{equation}\label{eq:max_def}
        \E[(\max(\CC_s)-s)^2] \leq 4\xi^4 + 3\sum_{t= s+2\xi^2}^{n-1} (t-s)\cdot \P[\max(\CC_s) >
        t].
   \end{equation}
   Define
     \[
        d_t:=\left\lceil\tfrac{1}{4}\log_{1/q}\left(\tfrac{t-s}{\xi^2}\right)\right\rceil \quad \text{and} \quad
        E_t:=\{\kappa_i\leq d_t+\xi \text{ for all $i\in[s,t]$}\}.
    \]
    Now write
    \begin{equation*}
      \sum_{t= s+2\xi^2}^{n-1} (t-s) \P[\max(\CC_s) >
        t] \le \sum_{t= s+2\xi^2}^{n-1} (t-s) \P[E_t^\complement] + \sum_{t= s+2\xi^2}^{n-1} (t-s) \P[\{\max(\CC_s) >
        t\}\cap E_t].
    \end{equation*}
    We estimate the two sums separately. For the first sum, using
    Theorem~\ref{mallows_chain_bounds} and the observation that $d_t\ge 12$ for $t\ge s+2\xi^2$
    as we assumed that $q$ is sufficiently close to $1$, we obtain
    \begin{equation*}
      \sum_{t= s+2\xi^2}^{n-1} (t-s) \P[E_t^\complement] \lesssim \sum_{t= s+2\xi^2}^{n-1} (t-s) q^{12d_t}\le \sum_{t= s+2\xi^2}^{n-1}
      (t-s)\left(\tfrac{\xi^2}{t-s}\right)^3\lesssim \xi^4.
    \end{equation*}
    For the second sum, Proposition~\ref{diam_upper_corr} implies
    that
    \begin{align*}
     \sum_{t= s+2\xi^2}^{n-1} (t-s) \P[\{\max(\CC_s) >
        t\}\cap E_t] &\le  \sum_{t= s+2\xi^2}^{n-1} (t-s) \exp\left(-(t-s)(1-q)^2
        q^{2d_t+2\xi}\right) \\ &\leq \sum_{t \geq s+2\xi^2} (t-s) \exp\left(-(t-s)(1-q)^2
        \tfrac{q^{2+2\xi}\xi}{\sqrt{t-s}}\right)\\
        &\leq \sum_{t \geq s+2\xi^2} (t-s) \exp\left(-c\sqrt{t-s}(1-q)\right)\lesssim \xi^4,
    \end{align*}
    for a positive absolute constant $c>0$, where the last estimate is not difficult to check directly.
\end{proof}

\end{subsection}

\begin{subsection}{Expected Cycle Length}\label{sec:cycle_length}
    In this section we prove Theorem~\ref{cycle_length}. As $|\CC_s|\leq\diam(\CC_s)+1$
    the upper bound for $\E|\CC_s|$ follows immediately from the upper bound on the diameter of $\CC_s$ proved in
    Theorem~\ref{cycle_diam_theorem}. Thus we need only prove the lower bound, namely that
         \begin{equation}\label{eq:cycle_len_lower_bound}
                \E|\CC_s|\gtrsim \min\{\xi^2,n\}.
    \end{equation}
    Since $|\CC_s|\geq 1$ we may and will restrict, for the rest of the section, to the regime
    where $q$ is sufficiently close to $1$ and $n$ is sufficiently
    large.

    Our starting point is the formula
    \[
            \E|\CC_s| = \sum_{t\in[n]} \P[t\in\CC_s].
        \]
    The lower bound \eqref{eq:cycle_len_lower_bound} is an immediate consequence of this formula combined with the next lemma.
    \begin{lemma} \label{lm:al_proof_cycle_length}
        Let $s,t\in [n]$ satisfy $|t-s| \leq \xi^2$ then
        \[
            \P[t\in\CC_s] \gtrsim 1.
        \]
    \end{lemma}
    \begin{proof}
    It suffices to consider the case $s<t$. Define the events $(V_r)$, $t\le r\le n$ by
     \[
        V_t :=\{\arc_s^{t}=\arc_t^{t}\} \quad \text{and} \quad
        V_r:=\left\{ \arc_s^{r}=\arc_t^{r} \text{ and }\ \arc_s^{r-1}\neq \arc_t^{r-1}\right\}
        \text{ for $r>t$.}
    \]
    By definition, $V_r$ occurs if and only if the arc of $s$ and the arc of $t$ merge exactly at time~$r$.
    Hence $\{t\in\CC_s\}$ is the disjoint union $\bigsqcup_{t\leq r\leq n}V_r$,
    yielding that
    \[
        \P[t\in\CC_s]=\sum_{t\leq r\leq n}\P[V_r].
    \]
    We shall prove the following estimates,
    \begin{equation}\label{eq:cycle_length_lemma_eq_main}
    \P[V_t]\gtrsim \frac{1-q}{1-q^{n-t+1}} \quad \text{and} \quad
        \P [V_r]\gtrsim \left(\frac{1-q}{1-q^{n-r+1}} \right)^2\quad \text{for $t<r < \min \{t+ \xi^2, n-2\}$.}
    \end{equation}
    The lemma follows easily from these, as if $t+\xi^2 < n-2$ one
    may sum the estimates for $\P [V_r]$ and otherwise it suffices
    to consider only the estimate for $\P [V_t]$ or $\P [V_{n-3}]$.
    Let us prove these estimates. The following inequality is an immediate consequence of
    Lemma~\ref{lemma_U_arc_eq},
    \begin{equation}\label{eq:V_t_estimate}
               \P_{t-1}[V_t] \geq q^{2\xi}\cdot\tfrac{1-q}{1-q^{n-t+1}}\cdot\1_{A\cap\{\kappa_{t-1}\le 2\xi\}}, \quad \text{where $A:=\{\text{$\arc_s^{t-1}$ is open}\}$}.
    \end{equation}
    Thus the estimate for $\P[V_t]$ follows by noting that $\P[A\cap\{\kappa_{t-1}\le
    2\xi\}]\gtrsim 1$ for $q$ sufficiently close to $1$. Indeed, $\P[A]\gtrsim 1$ by
    Proposition~\ref{cycle_diam_prop_low} and $\P[\kappa_{t-1}\le
    2\xi]$ tends to~$1$ as $q$ tends to
    $1$, uniformly in $t$ and $n$, by Theorem~\ref{mallows_chain_bounds}.

    We proceed to estimate $\P [V_r]$ for $t<r < \min \{t+ \xi^2,
    n-2\}$. Define
    \begin{equation*}
      B_i:=\{\text{$\arc_s^{i}$ and $\arc_t^{i}$ are open and
      distinct}\},\quad t\le i\le n.
    \end{equation*}
    Similarly to \eqref{eq:V_t_estimate}, we have
    \[
      \P_{r-1}[V_r] \geq q^{4\xi}\cdot\left(\tfrac{1-q}{1-q^{n-t+1}}\right)^2\cdot\1_{B_{r-1}\cap\{\kappa_{r-1}\le 2\xi\}}
    \]
    and, as before, it suffices to show that $\P[B_{r-1}]\gtrsim 1$. Lemma~\ref{lemma_U_arc_eq} implies that for $i>t$,
    \begin{align*}
      \P[B_i] &= \E \left[\1_{B_{i-1}}\cdot \left(1 -
      \P_{i-1}\left[
        \begin{array}{l}
        \pi_i\in\{\tail(\arc_s^{i-1}),\tail(\arc_t^{i-1})\}, \\
      \pi_i^{-1}\in\{\head(\arc_s^{i-1}),\head(\arc_t^{i-1})\} \end{array}
      \right] \right)\right]\\ &\ge
      \P[B_{i-1}] \cdot \left(1 -
      4\left(\tfrac{1-q}{1-q^{n-i+1}}\right)^2\right).
    \end{align*}
    Furthermore, for $i=t$ we obtain using Proposition~\ref{upper_bound_arc_s}
    and Lemma~\ref{lemma_U_arc_eq} that
    \begin{align*}
      \P[B_t] &= \P[A \cap \{\pi_t \neq \tail(\arc_s^{t-1}), \, \pi^{-1}_t \neq \head(\arc_s^{t-1}), \, t\neq
      \max(\CC_t)\}]\\
                &\geq \P[A]\cdot \left(1 - 3\cdot \tfrac{1-q}{1-q^{n-t+1}}\right).
    \end{align*}
    Combining these inequalities with the estimate
    $\P[A]\gtrsim 1$ proved previously shows that
    \begin{equation*}
      \P[B_{r-1}] \gtrsim \left(1 - 3\cdot \tfrac{1-q}{1-q^{n-t+1}}\right)\prod_{t < i < r} \left(1 -  4\left(\tfrac{1-q}{1-q^{n-i+1}}\right)^2\right)
    \end{equation*}
    By estimating the product as in the proof of
    Proposition~\ref{cycle_diam_prop_low} we conclude that
    $\P[B_{r-1}]\gtrsim 1$, as we wanted to show.
    \end{proof}
\end{subsection}

\begin{subsection}{Variance of Cycle Length}\label{sec:variance_bounds} In this section we prove
Theorem~\ref{cycle_var}. Throughout we let $\pi\sim\mu_{n,q}$. We
need to show that for every $s\in[n]$ we have
\begin{equation*}
    \var|\CC_s|\approx \min\left\{q\cdot \xi^4, (n-1)^2\right\}.
\end{equation*}
 We
divide the proof into 3 cases:
\begin{enumerate}[topsep=2pt,itemsep=1pt]
    \item The lower bound for $q\geq \frac{1}{8}$ follows from a general non-concentration argument together with Theorem~\ref{cycle_length}.
    \item The lower bound for $q\leq \frac{1}{8}$ is a corollary of Lemma~\ref{fix_point_equation}.
    \item The upper bound is a direct corollary of Proposition~\ref{diam_upper_prop_main}.
\end{enumerate}

\smallskip
{\bf Case 1:} Assume that $q\ge \frac{1}{8}$. Assume that $n>2$ as
the case $n=1$ is trivial. Further, using the reversal symmetry
\eqref{rev_symmetry}, assume that $s<n$.

We consider the following equivalence relation on $\S_n$: we say
that $\sigma_1\sim \sigma_2$ if
\[
     \sigma_1 \in\{\sigma_2,\ \tau\circ \sigma_2,\ \sigma_2\circ \tau,\ \tau\circ \sigma_2\circ \tau \} \quad \text{where $\tau:=(s,s+1)$.}
\]
Let $X$ be the random equivalence class of $\pi$ in this equivalence
relation. We shall prove that
\begin{equation}\label{eq:var_goal_cs1}
    \var(|\CC_s|\mid X) \gtrsim \E[|\CC_s|^2\mid X].
\end{equation}
The proof of the lower bound in this case follows from the
inequality together with Theorem~\ref{cycle_length} as
\[
    \var(|\CC_s|)\geq \E[ \var(|\CC_s|\mid X)] \gtrsim  \E[|\CC_s|^2] \geq \E[|\CC_s|]^2.
\]
Let us proceed with the proof of \eqref{eq:var_goal_cs1}. Composing
a permutation with an adjacent transposition, like $\tau = (s,s+1)$,
changes the number of inversions in the permutation exactly by $1$.
It follows that any two permutations $\sigma_1 \sim \sigma_2$
satisfy $|\inv(\sigma_1)-\inv(\sigma_2)|\leq 2$, whence
\eqref{eq:mu_n_q_def} implies that
\[
    \P[\pi=\sigma_1] \approx \P[\pi=\sigma_2].
\]
As $X$ is an equivalence class of $\S_n$ of size at most $4$, we
conclude that
\begin{equation}\label{eq:conditional_distr_pi_X}
    \P[\pi = \sigma \mid X]\approx 1, \quad \text{for all $\sigma \in X$.}
\end{equation}
The equivalence class $X$ necessarily contains a permutation $\rho$
satisfying $\CC_s(\rho) \neq \CC_{s+1}(\rho)$. Choosing such a
$\rho$, one checks that for each $\sigma\in X$, $|\CC_s(\sigma)|$ is
either $|\CC_s(\rho)|, |\CC_{s+1}(\rho)|$ or
$|\CC_s(\rho)|+|\CC_{s+1}(\rho)|$ and each of these values occurs
for some $\sigma\in X$. Thus \eqref{eq:var_goal_cs1} is a
consequence of \eqref{eq:conditional_distr_pi_X}.

\begin{remark}
  The above argument is a general argument for showing non-concentration of cycle lengths, i.e., that
  \begin{equation*}
    \var(|\CC_s|)\gtrsim\E[|\CC_s|^2].
  \end{equation*}
  It may be applied to other random permutation models satisfying
  the following assumption. There exists some $1\le t\le n$,
  different from $s$, for which
  \begin{equation*}
    \P[\sigma] \approx \P[\sigma \circ(s,t)]\approx
    \P[(s,t)\circ\sigma]\quad \text{for all $\sigma\in \S_n$}.
  \end{equation*}
\end{remark}

{\bf Case 2:} Assume now that $q\le \frac{1}{8}$. We need to prove
that
\[
    \var|\CC_s|\gtrsim q \quad \text{for $n>1$}.
\]
    It suffices to verify that
    \[
        1-\P[|\CC_s|=1] \gtrsim q \quad \text{and} \quad \P[|\CC_s|=1]\gtrsim 1.
    \]
    Let $\kappa$ be the arc chain of $\pi$.
    Due to the reversal symmetry \eqref{rev_symmetry} we may assume that $n-s\geq 1$.
    Lemma~\ref{fix_point_equation} implies that
    \[
        \P[|\CC_s|=1]   =\P[\pi_s=s]
                    =\E\left[\frac{q^{\kappa_{s-1}}-q^{\kappa_{s-1}+1}}{1-q^{n-s+1}}\cdot\frac{q^{\kappa_{s-1}}-q^{n-s+1}}{1-q^{n-s+1}}\right]
                    \leq\frac{1-q}{1-q^2}= \frac{1}{1+q}.
    \]
    Hence $1-\P[|\CC_s|=1]\geq\frac{q}{1+q}\gtrsim q$. Now we shall verify that $\P[|\CC_s|=1]\gtrsim 1$. Observe that
    \[
        \P[\pi_s=s] \geq \P[\pi_s=s\mid \kappa_{s-1}=0]\cdot \P[\kappa_{s-1}=0]\geq (1-q)\cdot \P[\kappa_{s-1}=0] \gtrsim \P[\kappa_{s-1}=0]
    \]
    by Lemma~\ref{fix_point_equation} and our assumption that $q\leq \frac{1}{8}$.
    The fact that $\P[\kappa_{s-1} = 0]\gtrsim 1$ follows from \eqref{mallows_corr_bounds} and Lemma~\ref{lem:stationary_distribution_bounds}.

\smallskip
{\bf Case 3:} We prove the upper bound for the variance of cycle length.
    Since $|\CC_s|\geq 1$ it follows that
    \begin{equation}\label{var_eq_1}
        \var|\CC_s|\leq \E[(|\CC_s|-1)^2]\leq\E[\diam(\CC_s)^2]\leq 2 \E[(\max(\CC_s)-s)^2+(s-\min (\CC_s))^2].
    \end{equation}
    Proposition~\ref{diam_upper_prop_main}, via the reversal symmetry \eqref{rev_symmetry}, states that
    \[
    \E[(\max(\CC_s)-s)^2]\lesssim
        \min\{q\xi^4,(n-s)^2\}\quad \text{and} \quad \E[(s-\min (\CC_s))^2]\lesssim \min\{q\xi^4,(s-1)^2\},
    \]
    which yields the upper bound via \eqref{var_eq_1}.
\end{subsection}

\begin{subsection}{Poisson-Dirichlet Law}\label{sec:pdl}
    In this section we prove Theorem~\ref{thm:pdl}. We need to prove two facts: that the normalized length $\frac{1}{n}|\CC_{s_n}|$ converges in
    distribution to the uniform distribution on $[0,1]$ for any sequence $(s_n)$ with $s_n\in[n]$ and that the distribution of the sorted and normalized cycle lengths converges to the Poisson-Dirichlet law.
    The proofs of these two facts are similar and we shall focus on
    the proof for the Poisson-Dirichlet law. At the end of the section
    we point out the needed modifications to obtain the limiting distribution of $\frac{1}{n}|\CC_{s_n}|$.

    \medskip{\bf The Poisson-Dirichlet law in a space of multisets of reals:}
    Denote by $\CD$ the space of sorted sequences
    $(\alpha_i)$, $i\ge 1$, $\alpha_1\ge\alpha_2\ge\ldots$ of non-negative reals
    with finite sum. The space $\CD$ is endowed with the product topology, the
    topology inherited from $\R^{\N}$. We shall consider also the $\ell^2$ metric on sequences in $\CD$ and take note of the fact that convergence in the $\ell^2$ metric implies convergence in the product topology.

    The Poisson-Dirichlet law with parameter one, denoted by $\CPD$, is a distribution on~$\CD$, supported on sequences
    with sum $1$. We will not need the precise definition of $\CPD$, instead relying only
    on its relation with uniform random permutations, and the reader
    is referred, e.g., to \cite{F10} for further background.
    Specifically, we shall use that if $\sigma$ is a uniformly random
    permutation in $\S_k$ and $\ell_1\ge\ell_2\ge\ldots$ are the
    sorted lengths of cycles in $\sigma$ then
    \begin{equation}\label{eq:pdl_convergence_eq1}
      \tfrac{1}{k}(\ell_1, \ell_2,\ldots)\stackrel{d}{\to}\CPD\quad\text{as
      $k\to\infty$},
    \end{equation}
    where a finite sequence is viewed as an element of $\CD$ by adding to it a
    trailing sequence of zeros.

    It is convenient to work with an alternative, equivalent, description of the space $\CD$. A sequence in $\CD$ may be equivalently
    described by a multiset of non-negative reals with finite sum (summing elements according to their multiplicities), in which the
    multiplicity of each number is the number of its occurrences in
    the sequence. Note that $0$ is the only number possibly having an infinite
    multiplicity in this representation. We denote the space of such multisets by
    $\CM$ and make the identification of $\CD$ and $\CM$ in the sequel, putting the induced topology on
    $\CM$, i.e., the push-forward of the product
    topology via the identification map. We denote by $d$ the metric on $\CM$ obtained as the push-forward of the $\ell^2$ metric on $\CD$. The metric $d$ has the following explicit expression:     
    given two multisets $X,Y\in \CM$,
    \begin{equation}\label{eq:multiset_metrics}
       d(X,Y) := \min_\varphi \sqrt{\sum_{x\in X} (\varphi(x)-x)^2},
    \end{equation}
    where the minimum is taken over all bijections $\varphi\colon X \to
    Y$ and it is understood that bijections may assign different images to multiple occurrences of the same element.
    To see that $d$ coincides with the push-forward of the $\ell^2$ metric on $\CD$ we note that the minimum in \eqref{eq:multiset_metrics} is obtained by taking $\varphi$ to be a monotone non-decreasing mapping.

    We will also consider multisets of finite cardinality of non-negative
    reals as members of $\CM$ by adding an infinite amount of zeros to the multiset.

    \medskip{\bf Convergence Criterion:}
    Here we state a convergence criterion to the Poisson-Dirichlet
    distribution which generalizes \eqref{eq:pdl_convergence_eq1}.
    We start with some definitions.

        For a permutation $\sigma\colon S \to S$ and a weight function $w\colon S \to \R$, where $S$ is a non-empty finite set,
    we define the weighted length function $\CL(\sigma,w)$ to be the multiset
        \begin{equation}\label{eq:cycle_weighted_length_multiset}
                \CL(\sigma,w) := \left\{\sum_{i\in\CC^1} w_i,\ldots,\sum_{i\in\CC^m} w_i \right\},
        \end{equation}
        where $\CC^1,\ldots,\CC^m$ is a decomposition of $S$ into orbits of $\sigma$.
        For instance, $\CL(\sigma,\1)$, where $\1$ denotes the constant $1$ function, is the multiset of cycle lengths of $\sigma$.

    The convergence in \eqref{eq:pdl_convergence_eq1} may be stated equivalently
    as follows: if $\sigma$ is uniform in $\S_k$ then $\CL(\sigma,\tfrac{1}{k}\1) \to \CPD$ as $k\to\infty$.
    The following proposition allows us to generalize this fact by `adding weights'.

    \begin{proposition}\label{prop:pdl_convergence_prop1}
    Let $S$ be a non-empty finite set, let $w\colon S\to[0,\infty)$ satisfy $\sum w_s\le 1$ and let $\sigma\colon S\to S$ be a uniformly random permutation.
    Then
    \[
        \E\big[d^2\big(\CL(\sigma,\tfrac{1}{|S|}\1), \CL(\sigma, w)\big)\big]\le
        \tfrac{1}{2}\Big(\big(1-\sum w_s\big)^2 +\sum
        (\tfrac{1}{|S|}-w_s)^2
        \Big).
    \]
    \end{proposition}
    In particular, if $w=w(k)\colon [k]\to[0,\infty)$ is a sequence of
    tuples satisfying $\sum_{i=1}^k w_i\le 1$ and the limit relations $\sum_{i=1}^k w_{i} \to 1$ and $\max_{i\in [k]} w_i \to
    0$ as $k\to\infty$, then $\CL(\sigma,\tfrac{1}{k}\1)$ and $\CL(\sigma,
    w)$ share the same limit distribution, which is the $\CPD$
    distribution by \eqref{eq:pdl_convergence_eq1}.

    \begin{proof}[Proof of Proposition~\ref{prop:pdl_convergence_prop1}]
    Using the notation of \eqref{eq:cycle_weighted_length_multiset}, the following inequality follows from
    \eqref{eq:multiset_metrics},
    \begin{equation}\label{eq:pdl_convergence_peq1}
        d^2(\CL(\sigma,\tfrac{1}{|S|}\1), \CL(\sigma, w))\leq \sum_r \Big(\sum_{s\in\CC^r}( \tfrac{1}{|S|}-w_s)\Big)^2 =
        \sum_{s\in S}\sum_{t\in S}(\tfrac{1}{|S|}-w_s)(\tfrac{1}{|S|}-w_t)\cdot \1\{t\in \CC_s(\sigma)\}.
    \end{equation}
    For
    each $s\in S$, as
    $\P[t\in\CC_s(\sigma)] = \frac{1}{2}\1_{s=t}+\frac{1}{2}$,
    \begin{equation}\label{eq:pdl_convergence_peq2}
        \E\Big[ \sum_{t\in S}(\tfrac{1}{|S|}-w_s)(\tfrac{1}{|S|}-w_t)\cdot \1\{t\in \CC_s(\sigma)\}\Big]=
        \tfrac{1}{2}\Big((\tfrac{1}{|S|}-w_s)\cdot\big(1-\sum_{t\in S} w_t\big) +(\tfrac{1}{|S|}-w_s)^2\Big).
    \end{equation}
    The proposition follows by taking expectation in \eqref{eq:pdl_convergence_peq1} and substituting \eqref{eq:pdl_convergence_peq2}.
    \end{proof}

    {\bf Proof of the Poisson-Dirichlet Law:} Let $q=q_n$ satisfy
    \begin{equation}\label{eq:delocalized_regime}
    (1-q)^2\cdot n\to 0\quad\text{as $n\to\infty$}.
    \end{equation}
    Let $m=m_n\in\N$ be a sequence that converges to infinity sufficiently slowly so that
    \begin{equation}\label{eq:pdl_main_eq1}
        \frac{m}{n}\to 0 \quad \text{and} \quad q^{m^2} \to 1 \quad\text{as $n\to\infty$}.
    \end{equation}
    For instance, we may take $m = \lfloor n^{1/4}\rfloor$ or $m=\lceil \log(n+1)\rceil$. Let $\pi\sim\mu_{n,q}$. We shall analyze $\pi$ conditioned on $\Gamma_\pi \cap U$, where
    \[
        U:=\{x \leq n-m \text{ or }y > n-m\}.
    \]
    The plan is to to show that the lengths of the long cycles can be coupled closely with the lengths of cycles in a uniform permutation.
    To do it we first show that, despite the fact that $\Gamma_\pi \cap U$ is almost the whole graph of $\pi$, with high probability,
    it provides little to no information on the lengths of the long cycles of $\pi$.
    The lengths of the long cycles, given $\Gamma_\pi \cap U$,
    are mostly decided by $\Gamma_\pi \smallsetminus U$, the remaining portion of the graph.
    We shall then conclude by utilizing the fact that the relative order (see \eqref{eq:relative_order_permutation}) of $\Gamma_\pi \smallsetminus U$
    is very close to a uniformly distributed permutation.

    Our proof requires us to define permutations over a finite random set $\CO_U$
    and analyze the multiset $\CL(\cdot,w)$ of such permutations for various weight tuples $w$.
    The set $\CO_U$ may be empty, though this case does not impact on our analysis as its probability tends to~$0$.
    To avoid treating it in a special manner, as much as possible, we define the set
    of permutations over the empty set to consist
    of a single element denoted $\textup{id}_\emptyset$. This
    element has no cycles and, in particular, $\CL(\textup{id}_\emptyset,w) = \emptyset$ for all
    $w$.

    Let us begin the proof. We introduce additional definitions to discuss arcs which are determined by $\Gamma_\pi \cap U$.
    We say that an arc $\Fa$ of $\pi$ \emph{belongs} to $U$ if one may order its elements $\Fa=\{a_1,\ldots,a_{|\Fa|}\}$
    so that $a_{i+1}=\pi (a_i)$ and $(a_i,a_{i+1}) \in U$ for all $1\le i<|\Fa|$.
    If, in addition, $a_1=\pi(a_{|\Fa|})$ and $(a_{|\Fa|},a_1)\in U$ we say that $\Fa$ is \emph{relatively closed} and otherwise deem it \emph{relatively open}.
    When the arc is relatively open then the above ordering is unique, in which case we call the elements $a_1$ and $a_{|\Fa|}$
    the \emph{tail} and \emph{head} of the arc $\Fa$, respectively, and denote them by $\tail(\Fa)$ and $\head(\Fa)$.
    We say that an arc $\Fa$ is \emph{$U$-maximal} if $\Fa$ belongs to $U$ and there are no other arcs that belong to $U$ and contain $\Fa$.
    Let $\CA_U$ be the set of $U$-maximal arcs and let $\CO_U$ be the set of relatively open $U$-maximal arcs. One should note that $\CA_U$ and $\CO_U$ are determined by
    $\Gamma_\pi \cap U$ and that $\sum_{\Fa\in\CA_U} |\Fa| = n$.
    The last equality follows from the fact that every element of $[n]$ belongs to exactly one of the arcs of $\CA_U$ (possibly to an arc containing only this element).

    We proceed to discuss the way that the cycles of $\pi$ are
    formed from the arcs in $\CA_U$ and the portion of the graph $\Gamma_\pi \smallsetminus
    U$. Each point $(s,t)\in\Gamma_\pi\smallsetminus U$
    satisfies $s = \head(\Fa)$ and $t = \tail(\Fb)$ for some $\Fa,
    \Fb\in\CO_U$. Conversely, for each $\Fa\in\CO_U$ there exist points $(s,t), (s',t')\in \Gamma_\pi\smallsetminus U$ satisfying $s = \head(\Fa)$ and $t'=\tail(\Fa)$.
    Thus we may define a permutation
     $\tau:\CO_U\to\CO_U$ by setting
    \begin{equation*}
      \tau(\Fa) = \Fb\quad\text{if and only if}\quad (\head(\Fa),
      \tail(\Fb))\in \Gamma_\pi.
    \end{equation*}
    It is straightforward to check that each cycle of $\pi$ is then either a cycle in $\CA_U\smallsetminus\CO_U$, or formed
    by merging the open arcs in $\CO_U$ which are in the same orbit
    of $\tau$. In particular,
    \begin{equation*}
      \CL(\pi,\tfrac{1}{n}\1) = \Big\{\tfrac{1}{n}|\Fa|\,\big|\,
      \Fa\in\CA_U\smallsetminus\CO_U\Big\}\cup\Big\{\tfrac{1}{n}\sum_{\Fa\in\CC}
      |\Fa|\,\big|\, \text{$\CC$ is a cycle of $\tau$}\Big\},
    \end{equation*}
    where the equality and union are in the sense of multisets. Write $|\cdot|$ for the length map on arcs, $\Fa\mapsto|\Fa|$. Recalling \eqref{eq:multiset_metrics}, we conclude that
    \begin{equation}\label{eq:pi_tau_relation}
      d \left(\CL(\pi,\tfrac{1}{n}\1), \CL(\tau, \tfrac{1}{n}|\cdot|)\right) \leq \sqrt{\sum_{\Fa\in\CA_U\smallsetminus\CO_U} \frac{|\Fa|^2}{n^2}}\leq \sum_{\Fa\in\CA_U\smallsetminus\CO_U} \frac{|\Fa|}{n} = 1 - \sum_{\Fa\in \CO_U} \frac{|\Fa|}{n}.
    \end{equation}
    The Poisson-Dirichlet law is a consequence of this inequality and the following lemma, which states the properties of the Mallows model that we require.
    \begin{lemma} \label{lem:pdl_arcs} Let $\pi\sim\mu_{n,q}$ and suppose $n\to\infty$ with \eqref{eq:delocalized_regime} and
    \eqref{eq:pdl_main_eq1} holding. Then:
\begin{flalign}\label{eq:pdl_arcs}
       \text{(i)} && \tfrac{1}{n}\max_{\Fa\in \CO_U} |\Fa| \to 0 \quad\text{and}
        \quad \tfrac{1}{n}\sum_{\Fa\in \CO_U} |\Fa| \to
        1\quad\text{in probability}.&&
        \end{flalign}
      (ii) There exists a coupling of $\tau$ and a permutation $\sigma\colon \CO_U\to \CO_U$ such
        that
        \begin{equation*}
          \P[\sigma\neq \tau]\to 0\quad\text{as $n\to\infty$}
        \end{equation*}
        and, conditioned on $\Gamma_\pi\cap U$, $\sigma$ has the uniform distribution on permutations of $\CO_U$.
    \end{lemma}
    To obtain the Poisson-Dirichlet limit law, let $\sigma$ be the
    random permutation given by part (ii) of Lemma~\ref{lem:pdl_arcs}. Observe
    that
    \begin{multline*}
      d \left(\CL(\pi,\tfrac{1}{n}\1), \CL(\sigma, \tfrac{1}{|\CO_U|}\1)\right)\le\\
      \underbrace{d \left(\CL(\pi,\tfrac{1}{n}\1), \CL(\tau, \tfrac{1}{n}|\cdot|)\right)}_{=:\textrm{I}}+
      \underbrace{d \left(\CL(\tau, \tfrac{1}{n}|\cdot|), \CL(\sigma, \tfrac{1}{n}|\cdot|)\right)}_{=:\textrm{II}}+
      \underbrace{d \left(\CL(\sigma, \tfrac{1}{n}|\cdot|), \CL(\sigma,
      \tfrac{1}{|\CO_U|}\1)\right)}_{=:\textrm{III}}.
    \end{multline*}
    We estimate the expectation of each of the last three summands separately. By
    \eqref{eq:pi_tau_relation} and part (i) of Lemma~\ref{lem:pdl_arcs},
    \begin{equation*}
      \E[\textrm{I}] \le \E\left[1 - \sum_{\Fa\in \CO_U} \frac{|\Fa|}{n}\right] \to
      0\quad\text{as $n\to\infty$}.
    \end{equation*}
    Note that as $d \left(\CL(\alpha,w), \CL(\beta, w)\right)\le 2\sum_i
    |w_i|$ for any two permutations $\alpha,\beta$ on a finite set $S$ and weight tuple
    $w:S\to\R$, we have $\textrm{II}\le 2\cdot \1_{\sigma\neq
    \tau}$. Thus, by part (ii) of Lemma~\ref{lem:pdl_arcs},
    \begin{equation*}
      \E[\textrm{II}]\le 2\cdot \P[\sigma\neq \tau]\to 0\quad\text{as $n\to\infty$}.
    \end{equation*}
    Lastly, $\textrm{III} = 0$ when $\CO_U = \emptyset$ and by Proposition~\ref{prop:pdl_convergence_prop1},
    \begin{equation*}
      \E[(\textrm{III})^2\,|\,\Gamma_\pi\cap U] \le \tfrac{1}{2}\Big(\big(1-\tfrac{1}{n}\sum_{\Fa\in\CO_U} |\Fa|\big)^2
      +\sum_{\Fa\in\CO_U}
        \big(\tfrac{1}{|\CO_U|}-\tfrac{|\Fa|}{n}\big)^2
        \Big)\quad \text{on $\{\CO_U \neq \emptyset\}$}.
    \end{equation*}
    Thus, applying part (i) of Lemma~\ref{lem:pdl_arcs} and observing that it
    implies, in particular, that
    $|\CO_U|\to\infty$ in probability as $n\to\infty$, we obtain
    \begin{equation*}
      \E[\textrm{III}]\le \sqrt{\E[(\textrm{III})^2]} \to 0\quad\text{as $n\to\infty$}.
    \end{equation*}
    We conclude that
    \begin{equation*}
      \E\Big[d \left(\CL(\pi,\tfrac{1}{n}\1), \CL(\sigma,
      \tfrac{1}{|\CO_U|}\1)\right)\Big]\to 0\quad\text{as
      $n\to\infty$}.
    \end{equation*}
    The limiting distribution of $\CL(\sigma,
      \tfrac{1}{|\CO_U|}\1)$ is the $\CPD$ distribution as,
      conditioned on $\Gamma_\pi\cap U$, $\sigma$ is a uniform
      permutation on $\CO_U$ by part (ii) of  Lemma~\ref{lem:pdl_arcs}, using \eqref{eq:pdl_convergence_eq1} and relying again on the fact that $|\CO_U|\to\infty$ in probability as
      $n\to\infty$. Thus $\CL(\pi,\tfrac{1}{n}\1)$ converges also to the $\CPD$ distribution, as we wanted to prove.

   \smallskip\noindent {\bf Proof of part~(ii) of Lemma~\ref{lem:pdl_arcs}:} This is a corollary of Lemma~\ref{lem:ggm_subgraph}.
    On the event $\{\CO_U=\emptyset\}$ we simply set $\sigma = \tau$.
    On the complementary event $\{\CO_U\neq\emptyset\}$ we do as follows.
    Set $k:=|\CO_U|$. Let $\tilde{\tau}$ be the relative order, see \eqref{eq:relative_order_permutation}, of $\Gamma_\pi \smallsetminus U$.
    By our construction,
   \begin{equation*}
     \tau = \alpha\circ\tilde{\tau}\circ\beta
   \end{equation*}
   for two bijections, $\alpha:[k]\to\CO_U$ and $\beta:\CO_U\to[k]$, which are determined by $\Gamma_\pi\cap
   U$. Explicitly, this follows by viewing $\tau\colon \CO_U \to \CO_U$ as the composition of $5$ maps:
    \[
        \CO_U\stackrel{\head}{\to} \head(\CO_U) \stackrel{\textup{monotone}}{\to} [k] \stackrel{\tilde \tau}{\to} [k]
        \stackrel{\textup{monotone}}{\to} \tail(\CO_U) \stackrel{\tail^{-1}}{\to} \CO_U,
    \]
 where $A \stackrel{\textup{monotone}}{\to} B$ stands for the unique monotone increasing bijection from $A$ to $B$,
    provided that $A$ and $B$ are subsets of $\N$ of the same size.

    It thus suffices to couple $\tilde{\tau}$ with a permutation $\tilde{\sigma}:[k]\to[k]$ in a way that
   \begin{equation*}
     \P[\tilde{\tau}\neq\tilde{\sigma}\mid k\geq 1]\to 0\quad\text{as $n\to\infty$}
   \end{equation*}
   and, conditioned on $\Gamma_\pi\cap U$, $\tilde{\sigma}$ has the uniform distribution on $\S_k$, as we may then take $\sigma := \alpha\circ\tilde{\sigma}\circ\beta$.

   By Lemma~\ref{lem:ggm_subgraph}, conditioned on $\Gamma_\pi\cap U$, we have that $\tilde{\tau}\sim \mu_{k,q}$. Hence the following claim suffices to finish the proof, using our assumption \eqref{eq:pdl_main_eq1} and the fact that $k\leq m$.
    \begin{claim}
        Let $k\in \N$, $0<q\le 1$ and $\rho\sim\mu_{k,q}$. Then $\rho$ may be coupled with a uniform random permutation $\lambda$ in $\S_k$
        such that
        \[
            \P[\rho\neq\lambda]\leq 1 - q^{k^2}.
        \]
    \end{claim}
    \begin{proof}
    We recall (see, e.g., \cite[Proposition 4.7]{mixing_times}) that the total variation distance of two probability distributions $\mu$ and $\nu$ on a finite set $S$ is given by
    \[
        \textup{TV}(\mu,\nu) := \frac{1}{2}\sum_{s\in S} \big|\mu[s]-\nu[s]\big| =  1 - \sum_{s\in S}\min\{\mu[s], \nu[s]\},
    \]
    and that there exists a coupling of the two distributions, i.e., random variables $X,Y$ with $X$ distributed $\mu$ and $Y$ distributed $\nu$, so that $\P[X\neq Y] = \textup{TV}(\mu,\nu)$ (and, moreover, there is no coupling achieving a smaller value for $\P[X\neq Y]$). Thus it suffices to show that the total variation distance of $\mu_{k,q}$ and the uniform distribution on $\S_k$ is at most $1 - q^{k^2}$.

        Let $Z_{k,q}$ be as in the definition \eqref{eq:mu_n_q_def} of the Mallows permutation and note that $Z_{k,q} \leq k!$.
        Since $\inv(\sigma) \leq k^2$ for all $\sigma\in\S_k$, we obtain from \eqref{eq:mu_n_q_def} that $\mu_{k,q}[\sigma] \geq \frac{q^{k^2}}{k!}$ for all $\sigma\in\S_k$. Thus the required total variation distance is at most $1 -  \sum_{\sigma\in\S_k}\min\{\frac{q^{k^2}}{k!},\frac{1}{k!}\} =1 - q^{k^2}$.

    \end{proof}

   {\noindent \bf Proof of part~(i) of Lemma~\ref{lem:pdl_arcs}:} The claim is derived from the following proposition regarding diagonal arcs.
    \begin{proposition}\label{prop:pdl_prop_diag_arcs}
        For $1\leq s \leq r \leq n$ one has
        \begin{align}
    \E|\{1\le i\leq r \mid \text{$\arc_i^r$ is closed}\}|&\lesssim 1 +(1-q)^2\cdot n^2 + \tfrac{1}{n-r+1}\cdot n, \label{eq:pdl_prop_arc_eq2} \\
            \E|\arc_s^r|&\lesssim 1 +(1-q)^2\cdot n^2 + \tfrac{1}{n-r+1}\cdot n. \label{eq:pdl_prop_arc_eq1}
        \end{align}
    \end{proposition}
    \noindent  Since each open arc of $\CA_{n-m}$ extends to an arc of $\CO_U$,
    inequality \eqref{eq:pdl_prop_arc_eq2}, applied with $r=n-m$, implies the second limit in \eqref{eq:pdl_arcs} by our assumptions that $m$ tends to infinity with $n$ and that \eqref{eq:delocalized_regime} and
    \eqref{eq:pdl_main_eq1} hold.

    To derive the first limit in \eqref{eq:pdl_arcs} we use the
    following general claim.
    \begin{claim}\label{cl:pdl_arcs_claim1}
        Let $I$ be a non-empty finite set. Let $A_1,\ldots,A_k$, where $k\in \N$ is random, be pairwise disjoint random subsets of $I$ with union $I$.
        For $i\in I$ let $\ell_i$ be the size of the $A_j$ to which $i$ belongs. Then
        \[
             \left(\tfrac{1}{|I|}\E [\max_{i\in I}\ell_i]\right)^3 \lesssim \tfrac{1}{|I|} \max_{i\in I} \E[\ell_i].
        \]
    \end{claim}
    Now, inequality \eqref{eq:pdl_prop_arc_eq1} yields that $\tfrac{1}{n}\E|\arc_s^{n-m}|\to 0$ uniformly over all $s\in[n-m]$ by our assumptions that $m$ tends to infinity with $n$ and that \eqref{eq:delocalized_regime} and
    \eqref{eq:pdl_main_eq1} hold.
    Claim~\ref{cl:pdl_arcs_claim1}, applied with $I = [n-m]$ and $A_1,\ldots,A_k$ being the diagonal arcs of
    $\CA_{n-m}$, allows one to deduce that in fact
    \begin{equation}\label{eq:pdl_max_diag_conv}
        \tfrac{1}{n}\E\Big[\max_{\Fa\in \CA_{n-m}}|\Fa|\Big] = \tfrac{1}{n}\E\Big[\max_{s\in[n-m]}|\arc_s^{n-m}|\Big]\to 0.
    \end{equation}
    Since restricting an arc of $\CO_U$ to $[n-m]$ yields an open arc of
    $\CA_{n-m}$ which is shorter by at most $m$ elements,
    the first limit in \eqref{eq:pdl_arcs} follows from \eqref{eq:pdl_max_diag_conv} using our assumption that $\frac{m}{n}\to 0$.
    \begin{proof}[Proof of Claim~\ref{cl:pdl_arcs_claim1}]
        Set $L = \max_{i\in I}\ell_i$ and $\alpha:=\frac{1}{|I|}\E[L]$. Due to $0\leq L\leq |I|$, Markov's inequality, for $|I|-L$,
        implies that
        \begin{equation}\label{eq:pdl_arcs_claim1_eq1}
            \P[L\geq \tfrac{1}{2} \alpha |I| ] \gtrsim \alpha.
        \end{equation}
        The sum of the $\ell_i$ satisfies
        \begin{equation}\label{eq:pdl_arcs_claim1_eq2}
            \sum_{i\in I} \ell_i \geq \1\{L \geq \tfrac{1}{2}\alpha |I|\}\cdot\sum_{i\in I} \ell_i\geq
            \frac{\alpha^2 |I|^2}{4}\cdot \1\{L \geq \tfrac{1}{2} \alpha |I|\}.
        \end{equation}
        The second inequality is due to the fact that on the event $\{L \geq \tfrac{1}{2} \alpha |I|\}$
        there are at least $\tfrac{1}{2} \alpha |I|$ values of $i\in I$ for which $\ell_i\geq \frac{1}{2} \alpha |I|$.
        By taking expectation in \eqref{eq:pdl_arcs_claim1_eq2} and substituting \eqref{eq:pdl_arcs_claim1_eq1} we obtain
        \[
            |I|\cdot\max_{i\in I}\E[\ell_i]\geq  \sum_{i\in I}\E\left[\ell_i\right]\geq \frac{\alpha^2 |I|^2}{4}\P[L\geq \tfrac{1}{2}\alpha  |I|]\gtrsim \alpha^3 |I|^2.
            \qedhere
        \]
    \end{proof}

    \begin{proof}[Proof of Proposition~\ref{prop:pdl_prop_diag_arcs}] Fix $s,r$ satisfying $1\leq s \leq r \leq n$.
     Define
    \[
        M_t:=|\{1\le i\leq t \mid \text{$\arc_i^t$ is closed}\}|\quad \text{and for $t\geq s$ let} \quad N_t:=|\arc_s^t|.
    \]
    Our intention is to provide bounds for $\E[N_s]$, $\E[N_{t+1}-N_t]$ and $\E[M_{t+1}-M_t]$, and to derive the proposition from these bounds.
    Observe that
    \[
        N_{s}=1 + \sum_{\Fa\in\CO_{s-1}} |\Fa| \cdot
        \1\left\{ \begin{array}{l} \pi_s=\tail(\Fa) \text{ or } \\ \pi^{-1}_s=\head(\Fa)\end{array}\right\},
    \]
      as $\arc_s^s$ consists of $s$ and up to two arcs of $\CO_{s-1}$ that have merged with $\{s\}$ via either their head or their tail.
    In light of this equality, Lemma~\ref{lemma_U_arc_eq} and the fact that $\sum_{\Fa\in \CO_{s-1}} |\Fa| \leq n$ imply the following bound,
    \begin{equation}\label{eq:N_s_bound}
        \E[N_{s}]\lesssim 1+ n \cdot \frac{1-q}{1-q^{n-s+1}}.
    \end{equation}
    To bound $N_{t+1}-N_t$ and $M_{t+1}-M_t$ we define the events
    \[
      \CM_{t+1}(\Fa,\Fb):= \left\{\begin{array}{l} \pi_{t+1}=\tail(\Fa) \\  \pi^{-1}_{t+1}=\head(\Fb)\end{array}\right\} \bigcup
                                \left\{ \begin{array}{l} \pi_{t+1}=\tail(\Fb) \\  \pi^{-1}_{t+1}=\head(\Fa)\end{array}\right\},\quad \Fa,\Fb\in\CO_t,
    \]
        denoting the merging of $\Fa$ and $\Fb$ when the two arcs are distinct and the closure of $\Fa$ when they are equal.
    Note that $\P_t[\CM_{t+1}(\Fa,\Fb)]\leq (\tfrac{1-q}{1-q^{n-t}})^2$ for all $\Fa,\Fb\in\CO_t$, as Lemma~\ref{lemma_U_arc_eq} implies.

    Observe that for $t \geq s$ one has $N_{t+1} - N_t = 0$ if $\arc_s^t$ is closed and otherwise one has
    \[
        N_{t+1} - N_t = \1\left\{ \begin{array}{l} \pi_{t+1}=\tail(\arc_s^t) \text{ or}\\ \pi^{-1}_{t+1}=\head(\arc_s^t)\end{array}\right\} +
        \sum_{\Fa\in\CO_t\smallsetminus\{\arc_s^t\}} |\Fa| \cdot \1_{\CM_{t+1} (\Fa, \arc_s^t)}\quad \text{on $\{\arc_s^t\in \CO_t\}$},
    \]
    as $N_{t+1}- N_t>0$ only when $\{t+1\}$ has merged with $\arc_s^t$,
    in which case $N_{t+1}- N_t = |\Fa| + 1$ if $\{t+1\}$ has also merged
    with another arc $\Fa\in \CO_t\smallsetminus\{\arc_s^t\}$ and otherwise $N_{t+1}- N_t = 1$.

    For the difference $M_{t+1} - M_t$, observe that
    \begin{align*}
        M_{t+1}- M_t &=\1\left\{t+1=\pi_{t+1}\right\}+\sum_{\Fa\in\CO_t } (|\Fa|+1)\cdot \1_{\CM_{t+1}(\Fa,\Fa)}\\
            &=\1\left\{t+1=\max(\CC_{t+1})\right\}+\sum_{\Fa\in\CO_t }  |\Fa| \cdot \1_{\CM_{t+1}(\Fa,\Fa)} ,
    \end{align*}
    as $M_{t+1}- M_t>0$ only when $t+1$ closes a cycle, in which case $ M_{t+1}- M_t = |\Fa| + 1$
    when $t+1$ closes the arc $\Fa\in\CO_t$ and $ M_{t+1}- M_t = 1$
    when $t+1$ forms a fixed point.

    Note that $\sum_{\Fa\in\CO_{t}}|\Fa|\leq n$ and apply Lemma~\ref{lemma_U_arc_eq} and Proposition~\ref{upper_bound_arc_s}
    to the above formulas for $N_{t+1} - N_t$ and $M_{t+1} - M_t$ to obtain
    \begin{align}
        \E[M_{t+1} - M_t] &\lesssim \frac{1-q}{1-q^{n-t}} + n \cdot \frac{(1-q)^2}{(1-q^{n-t})^2} \quad \text{for $t\geq 0$,}\label{eq:N_t_bound}\\
        \E[N_{t+1} - N_t] &\lesssim \frac{1-q}{1-q^{n-t}} + n\cdot \frac{(1-q)^2}{(1-q^{n-t})^2} \quad \text{for $t\geq s$.}\label{eq:M_t_bound}
    \end{align}

    By using the bounds \eqref{eq:N_s_bound}, \eqref{eq:N_t_bound} and \eqref{eq:M_t_bound} one may show that both $\E[N_r]$ and $\E[M_r]$ are bounded, up to multiplication by a positive absolute constant, by
       \[
          1 + n\cdot \frac{1-q}{1-q^{n-r+1}} + \sum_{t=0}^{r-1} \frac{1-q}{1-q^{n-t}}+n\cdot\sum_{t=0}^{r-1} \frac{ (1-q)^2}{(1-q^{n-t})^2},
       \]
   The bounds \eqref{eq:pdl_prop_arc_eq2} and \eqref{eq:pdl_prop_arc_eq1}
   follow by using \eqref{eq:q_k_approx}.
    \end{proof}
    {\bf Limiting Distribution of $\mathbf{\frac{1}{n}|\CC_{s_n}|}$:} Our proof of the fact that $\frac{1}{n}|\CC_{s_n}|$ converges in distribution to $U[0,1]$, the uniform distribution on $[0,1]$, is very similar to our proof of the Poisson-Dirichlet limit law.
    Therefore, let us only elaborate on the main differences.

    The proof is based on the following simple fact:
    For $k\in \N$ let $\rho\in \S_k$ be a uniformly random permutation and let $i_k\in [k]$ be arbitrary.
    Then,
    \begin{equation}\label{eq:uniform_cycle_limit}
        \tfrac{1}{k}|\CC_{i_k} (\rho)| \to U[0,1]\quad\text{in distribution as $k\to\infty$}.
    \end{equation}

    We use the same notation as in the proof of the Poisson-Dirichlet limit law, e.g., $\sigma$ is the random permutation over $\CO_U$
    from part (ii) of Lemma~\ref{lem:pdl_arcs}. For $\Fa \in \CO_U$ we denote by $\CC_\Fa (\sigma)$ the orbit of
    $\sigma$ and similarly with $\tau$.

Let $\Fa$ be the arc of $\CA_U$ that contains $s_n$. We first claim
that
\begin{equation}\label{eq:Fa_open_U}
\P[\Fa\in\CO_U]\to 1\quad\text{as $n\to\infty$}.
\end{equation}
Indeed, this is a consequence of Corollary~\ref{diam_lower_corr},
used with the reversal symmetry \eqref{rev_symmetry} if $s_n > n-m$,
and our assumptions that $m$ tends to infinity with $n$ and that
\eqref{eq:delocalized_regime} and \eqref{eq:pdl_main_eq1} hold.

In addition, we recall that $|\CO_U|\to\infty$ in probability as
$n\to\infty$ as a consequence of \eqref{eq:pdl_arcs}. Now, the
limits \eqref{eq:uniform_cycle_limit} and \eqref{eq:Fa_open_U} and
the fact that, given $|\CO_U|$, the distribution of $\sigma$ is
uniform imply that
\[
  \tfrac{1}{|\CO_U|}|\CC_\Fa (\sigma)| \ \to \ U[0,1]\quad\text{in distribution as $n\to\infty$},
\]
where it is understood that in the case when $\Fa\notin \CO_U$ we
set $\tfrac{1}{|\CO_U|}|\CC_\Fa (\sigma)|:=0$. It thus suffices to
show that
\[
        \left|\tfrac{1}{n}|\CC_{s_n}(\pi)| - \tfrac{1}{|\CO_U|}|\CC_\Fa (\sigma)| \right| \to  0 \quad \text{in probability as $n\to \infty$.}
\]
To see this, we write, similarly as in the proof of the
Poisson-Dirichlet limit law, interpreting $\CC_\Fa (\tau)$ and
$\CC_\Fa (\sigma)$ as empty when $\Fa\notin\CO_U$,
\begin{multline*}
  \left|\tfrac{1}{n}|\CC_{s_n}(\pi)| - \tfrac{1}{|\CO_U|}|\CC_\Fa (\sigma)|
  \right| \le\\
  \Big|\tfrac{1}{n}|\CC_{s_n}(\pi)| - \sum_{\Fb\in \CC_\Fa (\tau)} \tfrac{|\Fb|}{n}\Big|
  + \Big|\sum_{\Fb\in \CC_\Fa (\tau)} \tfrac{|\Fb|}{n} - \sum_{\Fb\in \CC_\Fa (\sigma)}\tfrac{|\Fb|}{n}\Big|
  + \Big|\sum_{\Fb\in \CC_\Fa (\sigma)}
  \tfrac{|\Fb|}{n} -  \tfrac{1}{|\CO_U|}|\CC_\Fa (\sigma)|\Big|.
\end{multline*}
The first of the terms on the right-hand side is small, in
probability, due to \eqref{eq:Fa_open_U}. The second term is small
by part (ii) of Lemma~\ref{lem:pdl_arcs}. The last term can be
bounded in a similar manner as in the proof of
Proposition~\ref{prop:pdl_convergence_prop1} and shown to be small
by part (i) of Lemma~\ref{lem:pdl_arcs}.
\end{subsection}

\end{section}

\begin{section}{Discussion and Open Questions}\label{sec:discussion}
In this work we study the Mallows model for random permutations,
providing estimates for the typical length and diameter of cycles.
We observe that macroscopic cycles emerge in the parameter range
$\frac{1}{(1-q)^2}\gg n$. In this regime we prove further that the
joint distribution of the lengths of long cycles in the permutation
converges to the Poisson-Dirichlet distribution. In this section we
discuss several further questions on the Mallows model as well as
questions pertaining to other related models of random permutations.

\medskip
\textbf{The limiting distributions of the cycle length and cycle
diameter.} Let $\pi$ have the Mallows distribution with parameters
$n$ and $q$. Recall that $\CC_s$ stands for the cycle in $\pi$
containing the point $1\le s\le n$, so that $|\CC_s|$ and
$\max(\CC_s) - \min(\CC_s)$ are the length and diameter of $\CC_s$
respectively. What can be said about the limiting distributions of
these quantities when $q\to 1$ and $n\to \infty$? To avoid boundary
effects, we restrict to the case that there is some $\alpha\in
(0,1)$ for which $s = s_n$ satisfies $\frac{s}{n}\to \alpha$. We
consider three cases.

\noindent{\bf Macroscopic cycles:} Suppose that
  $n(1-q)^2\to 0$. In this regime, as shown in Theorem~\ref{thm:pdl}, the
  normalized cycle length $\frac{1}{n}|\CC_s|$ converges in
  distribution to the uniform distribution on $[0,1]$. Figures~\ref{fig:cycle_length_uni_n50} and
  \ref{fig:cycle_length_uni_n1000} suggest that, in fact, a stronger
  convergence takes place. For any sequence $1\le k_n\le n$ bounded away from $1$ and
  $n$ in a suitable manner, one has $\P[|\CC_s| = k_n]\cdot n \to 1$.

  Using Corollary~\ref{diam_lower_corr}, it may
  additionally be shown that the cycle spans the full interval, in
  the sense of the following convergence in distribution,
  \begin{equation*}
    \frac{\min(\CC_s)}{n}
    \stackrel{d}{\to}0 \quad\text{and}\quad \frac{\max(\CC_s)}{n} \stackrel{d}{\to}1.
  \end{equation*}

\noindent{\bf Microscopic cycles:} Suppose that
  $n(1-q)^2\to \infty$. It appears from simulations (see Figures \ref{fig:cycle_length} and \ref{fig:cycle_length_2})
  that the limiting distribution of the normalized cycle length $(1-q)^2|\CC_s|$ exists in
  this regime, but it is unclear what its form is. This
  limiting distribution, if it indeed exists, cannot be concentrated
  on a single point due to the lower bound on the variance of
  $(1-q)^2|\CC_s|$ given in Theorem~\ref{cycle_var}, used together
  with the bounds in Proposition~\ref{diam_upper_corr} and
  Theorem~\ref{mallows_chain_bounds}.

  The normalized cycle diameter $(1-q)^2(\max(\CC_s) - \min(\CC_s))$ seems simpler to analyze. Our
  results imply that there exist absolute constants $0<c_1\le
  c_2<\infty$ such that for any fixed $x\ge 0$,
  \begin{align*}
    \liminf \P \left[(1-q)^2(\max(\CC_s) - s)\ge x \right] &\ge e^{-c_2  x}, \\
     \limsup \P \left[(1-q)^2(\max(\CC_s) - s)\ge x \right] &\le e^{-c_1  x}.
  \end{align*}
  The lower bound follows from Corollary~\ref{diam_lower_corr} and
  the upper bound follows by again considering
  Proposition~\ref{diam_upper_corr} together with
  Theorem~\ref{mallows_chain_bounds}. We conjecture that, in fact,
  there exists a single absolute constant $c>0$ for which
  \begin{equation*}
    \lim \P \left[(1-q)^2(\max(\CC_s) - s)\ge x \right] = e^{-c  x},
    \quad x\ge 0.
  \end{equation*}
  That is, that the limiting distribution of $(1-q)^2(\max(\CC_s) -
  s)$ is exponential.
  By symmetry, the same is conjectured for $(1-q)^2(s -
  \min(\CC_s))$. Furthermore, we conjecture that the dependence
  between $\max(\CC_s)$ and $\min(\CC_s)$ disappears in this limit,
  so that the diameter $(1-q)^2(\max(\CC_s) -
  \min(\CC_s))$ converges in distribution to the sum of two
  independent, identically distributed, exponential random
  variables.

\noindent{\bf Intermediate regime:} Suppose that
  $n(1-q)^2\to \beta\in(0,\infty)$. We expect the limiting
  distributions to still exist in this regime and interpolate in a continuous manner the previous
  two cases. For the cycle diameter, this interpolation may possibly be
  achieved by truncation, as certainly $\max(\CC_s) - s\le n-s$.
  Recalling that $\frac{s}{n}\to \alpha$, we conjecture, for instance, that the limiting
  distribution of $(1-q)^2(\max(\CC_s) - s)$ is equal to that of
  $\min\{X, \beta(1-\alpha) \}$ where $X$ is the exponential
  random variable conjectured to give the limiting distribution in the previous regime.

\pgfplotsset{width=\textwidth}
\begin{figure}[h]
\begin{subfigure}{0.46\textwidth}
\begin{tikzpicture}
\begin{axis}[
xmin = 2,
xmax = 125,
ymin = 0,
ymax = 4,
ytick = {0, 1, 2, 3, 4},
yticklabels = {0\%, 1\%, 2\%, 3\%, 4\%},
line width = 1pt
]
\addplot[blue] table {CycleLength250n8q.dat};

\end{axis}
\end{tikzpicture}
\caption{   $n=250, \ q=0.8$}
\end{subfigure}
\quad
\begin{subfigure}{0.46\textwidth}
\begin{tikzpicture}
\begin{axis}[
xmin = 1,
xmax = 1000,
ymin = 0,
ymax = 0.25,
ytick = {0, 0.05, 0.10, 0.15, 0.20, 0.25},
yticklabels = {0\%, 0.05\%, 0.10\%, 0.15\%, 0.20\%, 0.25\%},
line width = 1pt,
yticklabel style={
/pgf/number format/precision=2,
/pgf/number format/fixed,
/pgf/number format/fixed zerofill,
}
]
\addplot[blue] table {CycleLength1000n99q.dat};

\end{axis}
\end{tikzpicture}
\caption{   $n=1000, \ q=0.99$}
\label{fig:cycle_length_uni_n1000}
\end{subfigure}
\caption{Distribution of the length of the cycle containing a
uniform random point. Obtained empirically with 1000000 samples.}
\label{fig:cycle_length_2}
\end{figure}
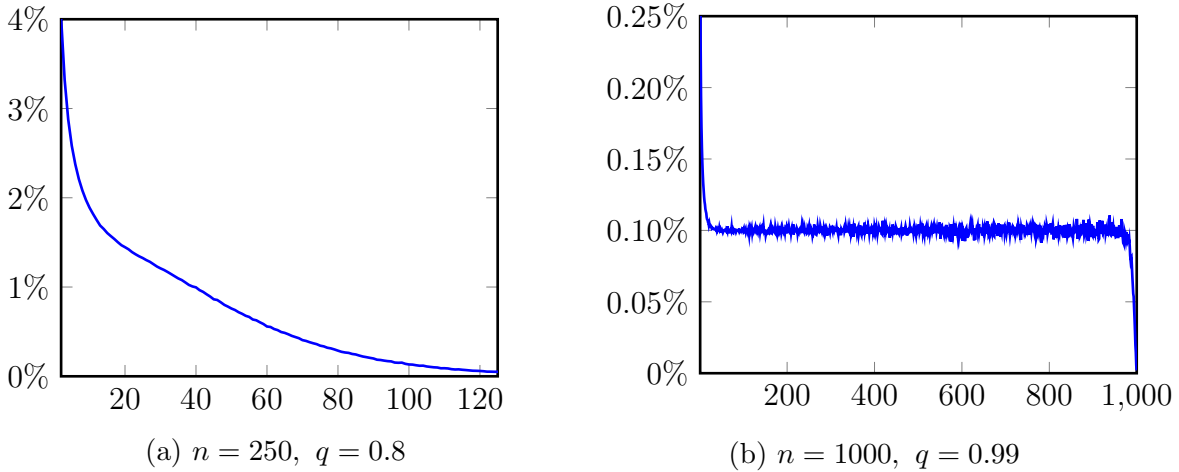

\medskip
   \textbf{Extensions of the parameter range.} As discussed in the
   introduction, there exist extensions of the Mallows distribution to infinite permutations; one-to-one and onto functions
$\pi\colon\N\to\N$ or $\pi\colon\Z\to\Z$. We expect the analogues of
our main theorems regarding the expected length, variance of the
length and expected diameter of cycles, Theorem~\ref{cycle_length},
Theorem~\ref{cycle_var} and equation \eqref{diam_equation} of
Theorem~\ref{cycle_diam_theorem}, to continue to hold for these
models, plugging formally $n=\infty$ and taking $s\in\N$ or $s\in\Z$
according to the case. Such results may follow from our methods,
using the sampling mechanism described in Section~\ref{ggm} for the
case $\pi\colon\Z\to\Z$, but we do not develop this further. The
approximation theorems of Gnedin and Olshanski \cite[Section
7.2]{GO12} may also prove useful in this context.

\pgfplotsset{width=\textwidth}
\begin{figure}
\begin{subfigure}[t]{0.46\textwidth}
\begin{tikzpicture}
\begin{axis}[
xmin = 0,
xmax = 1000,
ymin = 0,
ymax = 1000,
line width = 1pt,
]
\addplot[
red,
domain=0:900,
samples=901,
line width = 1pt
]
{900-x)};
\addplot[
red,
domain=100:1000,
samples=901,
line width = 1pt
]
{1100-x)};
\addplot+[blue, only marks, mark size = 0.1pt] table {mallowsSample1000n102q.dat};
\end{axis}
\end{tikzpicture}
\end{subfigure}
\begin{subfigure}[t]{0.46\textwidth}
\begin{tikzpicture}
\begin{axis}[
xmin = 0,
xmax = 1000,
ymin = 0,
ymax = 1000,
line width = 1pt,
]
\addplot[
red,
domain=0:859,
samples=860,
line width = 1pt
]
{x+100*sqrt(2))};
\addplot[
red,
domain=141:1000,
samples=860,
line width = 1pt
]
{x-100*sqrt(2))};
\addplot+[blue, only marks, mark size = 0.1pt] table {mallowsSample1000n102qX2.dat};
\end{axis}
\end{tikzpicture}
\end{subfigure}
\caption{
    On the left is a graph of a sample of the Mallows distribution $\mu_{n,q}$ with $n=1000$ and $q=1.02$. On the
    right is the graph of the composition of the same permutation
    with itself.
    The red lines are at vertical distance $\frac{2}{1-q}$ (left) and $\frac{2\sqrt{2}}{1-q}$ (right) from the diagonal. They delimit a region containing most of the points of the permutation.
} \label{fig:mallows_sample_x2}
\end{figure}
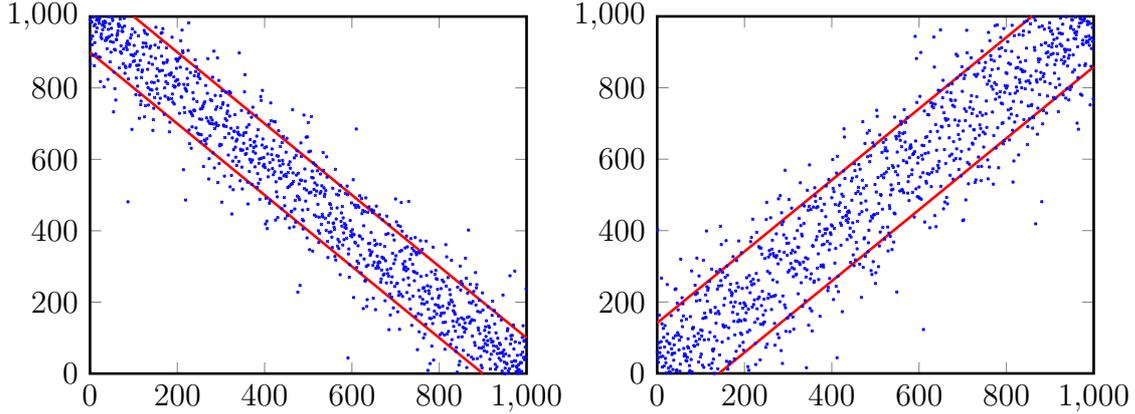

The Mallows distribution with parameters $n$ and $q$ is defined for
the case that $q>1$ via the same formula \eqref{eq:mu_n_q_def}. The
distributions with parameters $q$ and $\frac{1}{q}$ are related: If
$\pi\sim\mu_{n,q}$ then $\pi\circ r$, with $r_s = n-s+1$, is
distributed $\mu_{n,1/q}$. This operation corresponds to reflecting
the graph of the permutation $\pi$ across the line $x =
\frac{n+1}{2}$. Cycles are significantly affected by this operation
as, when $\pi\sim\mu_{n,q}$ with $q>1$, $\pi_i$ typically lies
around $n+1-i$, as follows from \eqref{eq:displacements} and the
above relation. Thus, for instance, the diameter of the cycle
containing $1$ in $\pi$ may be close to $n$ for all $q>1$. Still, we
would expect that for all $q>1$ the expected cycle lengths are still
of order $\min\left\{\frac{1}{(1-1/q)^{2}},\, n\right\}$ as in
Theorem~\ref{cycle_length} and further expect a Poisson-Dirichlet
limit law for the long cycles when $n\to\infty$, $q\to 1+$ and
$n(q-1)^2\to 0$ as in Theorem~\ref{thm:pdl}. This is suggested by
the fact that composing $\pi$ with itself leads to a permutation
whose graph is qualitatively similar to that of the Mallows model
with $q<1$, see Figure~\ref{fig:mallows_sample_x2}.

Our study of the cycle lengths of the Mallows permutation with
parameter $q<1$ was based on the diagonal exposure process.
Possibly, a similar process may be used to study the case $q>1$ by
exposing the graph of the permutation `from the center outwards'.
That is, exposing after $t$ iterations the portion of the graph
contained in a square of side length $2t$ around the mid-point
$\left(\frac{n+1}{2}, \frac{n+1}{2}\right)$. The ideas in
Section~\ref{ggm} may be useful in making such an approach rigorous.

\medskip
   \textbf{Band models.}
   We expect analogues of our results to hold for other natural models of
   random permutations whose graph typically has a `band structure'. For
   instance, for the interchange model on the one-dimensional graph $\{1,\ldots,
   n\}$ with nearest-neighbor edges. In this case, as briefly discussed in
   Section~\ref{sec:spatial_random_permutations}, the analogous result for the expected cycle
   length has been proved by Kozma and Sidoravicius. Another natural
   model is the band-Poisson model. Here, for an integer $n\ge 1$ and real $0<w\le n$, one considers a Poisson
   point process with intensity $\frac{1}{w}$ in the continuous band given by
   \begin{equation*}
     \{(x,y)\in[0,n]^2\mid |x-y|\le w\},
   \end{equation*}
   where the parameter $w$ controls the width of the band. With this definition,
   each vertical strip of width $1$ in $[0,n]^2$ contains on average a constant number of points.
   Each realization of the process gives rise to a permutation by taking
   the relative order of the points as in~\eqref{eq:relative_order_permutation}. We expect the
   analogues of our results to hold for this model with $w$ standing for $\frac{q}{1-q}$.

\medskip
\textbf{Higher dimensions and general graphs.} As discussed in
Section~\ref{sec:spatial_random_permutations}, the study of cycles
of spatial random permutations, random permutations biased towards
the identity in an underlying geometry, is of great interest. The
special case in which the geometry is that of $\R^d$ or $\Z^d$ is
particularly significant with relations to models of statistical
physics. In this context, our work pertains to the case $d=1$. With
other geometries in mind, we note here that a Mallows model may be
defined on any finite connected graph $G = (V,E)$ and parameter
$0<q\le 1$ by letting the probability of a permutation $\pi:V\to V$
be given by
\begin{equation*}
  \P_{G,q}[\pi] = \frac{1}{Z_{G,q}} q^{d(\pi, \textup{Id})},
\end{equation*}
where $Z_{G,q}$ is a normalization constant and $d(\pi,
\textup{Id})$ stands for the minimal number of adjacent
transpositions required to change $\pi$ to the identity permutation
$\textup{Id}\colon v \mapsto v$. Here, an adjacent transposition is
a transposition of the endpoints of an edge of $G$. For instance,
any transposition is allowed on the complete graph on $n$ vertices
$K_n$, whence the model coincides with the well-studied Ewens model~\cite{E72}, with parameter $\theta = \frac{1}{q}$.
This follows from the fact that
\begin{equation*}
  \P_{K_n, q}[\pi] = \frac{q^n}{Z_{K_n, q}}
  q^{-\mathcal{N}(\pi)}
\end{equation*}
with $\mathcal{N}(\pi)$ denoting the number of cycles (including
fixed points) in $\pi$. Our analysis of the Mallows model is based
on the exact sampling algorithm given by
\eqref{eq:sampling_formula}. Unfortunately, we are not aware of
corresponding algorithms for general graphs (though in the specific
case of the Ewens model an algorithm is given by the so-called
Chinese restaurant process). Nonetheless, it is of interest to
obtain results on the length of long cycles for general graphs $G$,
with the case that $G$ is a box in $Z^d$, $d\ge 2$, having special
significance.

\paragraph{Acknowledgment.} We thank Nayantara Bhatnagar, Gady Kozma, Grigori Olshanski and Sasha Sodin for useful
discussions. We are also grateful to an anonymous referee for a
detailed reading of the paper and many excellent comments which
contributed significantly to the presentation.
\end{section}

\end{document}